\documentclass[reqno,10pt]{amsart}
\usepackage[square, numbers]{natbib}
\usepackage[utf8x]{inputenc}
\usepackage[english]{babel}

\usepackage{enumitem}

\usepackage[pdftex]{graphicx}
\usepackage{float}
\usepackage{morefloats}
\usepackage[font=footnotesize, labelfont=bf]{caption}
\usepackage{subcaption}
\usepackage{placeins}

\usepackage{amsmath}
\usepackage{amssymb}
\usepackage{amsbsy}
\usepackage{empheq}
\usepackage{geometry}
\usepackage{amsthm}
\usepackage{mathrsfs}
\usepackage{xfrac}
\usepackage{ifthen}

\usepackage[colorlinks,citecolor=blue,urlcolor=black]{hyperref}

\theoremstyle{plain} 
\newtheorem{theorem}{Theorem}[section]
\newtheorem{lem}[theorem]{Lemma}
\newtheorem{prop}[theorem]{Proposition}

\theoremstyle{definition}
\newtheorem{definition}[theorem]{Definition}

\theoremstyle{remark}
\newtheorem{rem}[theorem]{Remark}

\numberwithin{equation}{section}

\newcommand{\step}[2][]{\\\indent \textbf{#2:} 
    \ifthenelse{\equal{#1}{}}{}{(\textit{#1})}}

\providecommand{\newoperator}[3]{%
  \newcommand*{#1}{\mathop{#2}#3}}
\newoperator{\supp}{\mathrm{supp}}{\nolimits}
\newoperator{\diam}{\mathrm{diam}}{\nolimits}
\newoperator{\dist}{\mathrm{dist}}{\nolimits}
\newoperator{\dd}{\mathrm{d}}{\nolimits}

\newcommand{\dsum}[3][]{\sum_{\begin{smallmatrix}#2\\#3\end{smallmatrix}}
	\ifthenelse{\equal{#1}{}}{}{^{#1}}}

\newcommand{\N}{\ensuremath{\mathbb{N}}}   
\newcommand{\Z}{\ensuremath{\mathbb{Z}}}   
\newcommand{\R}{\ensuremath{\mathbb{R}}}   

\newcommand{\eg}{{\it e.g.}, }
\newcommand{\ie}{{\it i.e.}, }

\newcommand{\Rn}{\R^{n}} 
\newcommand{\Rd}{\R^{d}} 
 
\newcommand{\Rdn}{\R^{d \times n}} 
\newcommand{\tomega}{\tilde{\Omega}}
\newcommand{\tF}{\tilde{F}}
\newcommand{\tphi}{\tilde{\phi}}
\newcommand{\tu}{\tilde{u}}
\newcommand{\clA}{\overline{A}}
\newcommand{\tA}{\tilde{A}}
\newcommand{\wcont}{\subset\subset}

\newcommand{\I}{\ensuremath{\mathcal{I}}}
\newcommand{\J}{\ensuremath{\mathcal{J}}}
\newcommand{\M}{\ensuremath{\mathcal{M}}}

\newcommand{\Areg}{\ensuremath{\mathcal{A}^{reg}}}
\newcommand{\A}{\ensuremath{\mathcal{A}}}
\newcommand{\LL}{\ensuremath{\mathcal{L}}}

\newcommand{\NN}{\ensuremath{\mathcal{N}}}
\newcommand{\m}{\ensuremath{\mathbf{m}}}

\newcommand{\Hn}{\mathcal{H}^{n-1}}

\newcommand{\clBone}{\ensuremath{\overline{B}^{|\cdot|_1}}}

\newcommand{\e}{\ensuremath{\varepsilon}}
\newcommand{\eps}{\ensuremath{\varepsilon}}
\newcommand{\loc}{\ensuremath{\mathrm{loc}}}

\newcommand{\nequiv}{\ensuremath{\not\equiv}}

\newcommand{\dx}{\ensuremath{\,dx}}

\newcommand{\dt}{\ensuremath{\,dt}}

\newcommand{\dHn}{\ensuremath{\,d\Hn}}
\newcommand{\scalarproduct}[2]{\langle#1, #2\rangle}

\newcommand{\jumpset}[1]{S_{#1}}
\newcommand{\jumpsetu}{\jumpset{u}}
\newcommand{\clsu}{\overline{\jumpsetu}}

\allowdisplaybreaks

\begin{document}
\title[Free-discontinuity functionals from discrete to continuum]{Discrete-to-continuum limits of multi-body systems with bulk and surface long-range interactions}

\author{Annika Bach}
\address[Annika Bach]{Zentrum Mathematik, Technische Universit\"at M\"unchen, Boltzmannstra{\ss}e 3, 85748 Garching bei M\"unchen, Germany}
\email{annika.bach@ma.tum.de}	
\author{Andrea Braides}
\address[Andrea Braides]{Dipartimento di Matematica, Universit\`a di Roma ``Tor Vergata'',
via della Ricerca Scientifica 1, 00133 Rome, Italy} 
\email{braides@mat.uniroma2.it}
\author{Marco Cicalese}
\address[Marco Cicalese]{Zentrum Mathematik, Technische Universit\"at M\"unchen, Boltzmannstra{\ss}e 3, 85748 Garching bei M\"unchen, Germany}
\email{cicalese@ma.tum.de}
\date{}

\begin{abstract}\noindent 
We study the atomistic-to-continuum limit of a class of energy functionals for crystalline materials via $\Gamma$-convergence. We consider energy densities that may depend on interactions between all points of the lattice and we give conditions that ensure compactness and integral-representation of the continuum limit on the space of special functions of bounded variation. This abstract result is complemented by a homogenization theorem, where we provide sufficient conditions on the energy densities under which bulk- and surface contributions decouple in the limit. The results are applied to long-range and multi-body interactions in the setting of weak-membrane energies.
\end{abstract}

\maketitle

\section{Introduction}
The passage from atomistic to continuum models is of major interest in the description and understanding of many physical phenomena and in models in applied sciences. Even for those atomistic systems which are driven by simple lattice energies, the choice of the method to analyse their asymptotic behavior as the interatomic distance tends to zero is nontrivial. Compare for instance the results obtained by taking pointwise limits (\cite{BBL02, BBL07, EM07}) to those obtained by variational methods (see \cite{BG06, Braides14} for an overview). There, the choice of the limit process underlines some assumptions on the model, which are translated in the definition of convergence of discrete to continuum functions, and may lead to different results.

In this paper we work within the variational framework, which amounts to allow for a very general definition of convergence of discrete functions and is translated in analyzing the asymptotic behavior of discrete systems in terms of $\Gamma$-convergence. This has proven to be a powerful tool in Materials Science to predict or better understand the macroscopic response of a material to microscopic deformations, but has also been used in other applied fields such as Computer Vision to provide discrete approximations of given continuum energies that might be used, \eg for numerical simulations, or in Data Science, to provide continuum minimal-cut approximations to problems in Machine Learning. We will use the terminology of `atoms' and keep the application to physical problems in mind, even though in the frameworks just mentioned discrete domains can be thought of as composed by pixels or labels of data.  We restrict our description to the case when the reference configuration of a material at the atomistic scale can be assumed to be a (Bravais) lattice ({\em crystallization}); this assumption could be relaxed to considering non-Bravais or disordered lattices, at the expense of a more complex notation.
In our case it is not restrictive to assume the reference configuration to be (a portion of) the cubic lattice $\Z^n$ in $\Rn$, scaled by a small parameter.
More precisely, fixing $\e>0$ one describes the atomistic deformation of a material occupying an open bounded domain $\Omega\subset\Rn$ through a map $u:Z_\e(\Omega)\to\Rd$, where $Z_\e(\Omega):=\Omega\cap\e\Z^n$ denotes the set of $\e$-spaced material points (or simply atoms) of the system. In the most general case, one can assume such a system to be driven by an energy of the form
\begin{equation}\label{intro:energy}
F_{\e}(u)=\displaystyle\sum_{i\in Z_\e(\Omega)}\e^{n}\phi^{\e}_i(\{u^{i+j}\}_{j\in Z_\e(\Omega-i)}).
\end{equation}
Here for fixed $i$ the function $\phi_i^\e:(\R^d)^{Z_\e(\Omega -i)}\to [0,+\infty)$ should be thought of as the potential energy at scale $\e$ describing the interaction between the atom at position $i$ and the whole configuration $\{u^{j}\}_{j\in Z_{\e}(\Omega)}$. As a consequence, energies as in \eqref{intro:energy} can model systems which are (at the same time) non-homogeneous, multi-body, non-local and multi-scale.

\subsection{Aim of the paper}
In this paper we are interested in the variational description (via $\Gamma$-convergence) of the limit of the $F_{\e}$ above as the lattice spacing $\e$ vanishes while the density of the atoms is kept constant thanks to the scaling factor $\e^{n}$. We refer to such a coarse-graining procedure as discrete-to-continuum limit. As a matter of fact a fine description of the discrete-to-continuum limit of physical systems driven by energies as those in \eqref{intro:energy} turns out to be a very challenging task unless the potentials are explicitly known and take some very special form. Until now the most general result in this direction has been obtained in \cite{BK18}, where the authors establish a set of assumptions on the potential energies $\phi_i^\e$ which ensure that up to subsequences the $\Gamma$-limit of energies as in \eqref{intro:energy} is an integral functional defined on a Sobolev space. The aim of the present paper is the extension of such a general result to the setting of special functions of bounded variations, that is to find sufficient conditions on $\phi_i^\e$ under which the variational limit energy of the sequence $(F_\e)$ is of the form
\begin{equation}\label{intro:energy-lim}
F(u)=\displaystyle\int_\Omega f(x,\nabla u)\dx\,+\int_{S_u}g(x,u^{+}-u^{-},\nu_u)\dHn
\end{equation}
defined on those $u$ (here we use the same notation $u$ for both microscopic and macroscopic fields) belonging to $SBV(\Omega;\R^{d})$. Energies of this type are usually referred to as free-discontinuity functionals and are widely used to model a number of phenomena in fracture mechanics, image reconstruction or in the theory of liquid crystals, to make only a few examples (\cite{Aubert-Kornprobst,braides98,BourdinFrancfortMarigo,MorelSolimini}). The discrete-to-continuum analysis performed in the present paper thus provides a very general framework on the one hand for atomistic systems whose macroscopic behavior can be studied in the context of fracture mechanics and on the other hand for possible discrete approximations of energies used in image reconstruction, such as for instance the approximations studied in \cite{Chambolle95, Chambolle99, BG02, Ruf17}. We point out that our analysis is also connected to some recent results in Data Science  \cite{Slepcev1,SlepcevChambolle,slepcev2019analysis}.

The assumptions on the potentials $\phi_i^\e$ that are needed to restrict the class of possible discrete-to-continuum limits to functionals of the form \eqref{intro:energy-lim} are carefully listed in Section \ref{SOP}. Here we limit ourselves to highlight the main ideas behind them in the case that $u$ represents the elastic deformation field of a physical system to be studied within the theory of fracture mechanics. In this case the two energy terms in \eqref{intro:energy-lim} can be interpreted as follows. The bulk integral represents the (hyper-)elastic energy stored in the system due to the contribution of bounded microscopic deformation gradients, that is of deformations with $|u^{i}-u^{j}|/\e$ of order one. The surface term represents the energy the system needs to produce the fracture $S_{u}$ in $\Omega$ with opening $u^{+}-u^{-}$. Such an energy is instead due to microscopic deformation gradients of order $1/\e$. In the simplest possible case $f(x,M)=|M|^{p}$ and $g=const$ the bulk and surface energies are proportional to the $p$-th power of the $L^p$ norm of the macroscopic deformation gradient $\nabla u$ and to the length of the fracture, respectively. 

Within this framework the assumptions on the potentials $\phi_i^\e$ read as follows.
\begin{enumerate}[label={\rm (H\arabic*)}]
\item\label{intro:H1} invariance under translations in $u$: This ensures that the integrand $f$ in \eqref{intro:energy-lim} does not depend explicitly on $u$ and $g$ depends on $u^+$ and $u^-$ only through their difference;
\item\label{intro:H2} monotonicity in the strain: the potential energy is assumed to be a non-decreasing in the finite differences $|u^{i}-u^{j}|$ - in the simple case of pairwise interactions this translates to the fact that the elastic energy increases as the modulus of the deformation gradient increases;
\item\label{intro:H3} weak Cauchy-Born type upper bound: we only require that the potential energy of any microscopic affine deformation is bounded from above by the $p$-th power ($p>1$) of the norm of its gradient;
\item\label{intro:H4} lower bound that allows to deduce that the limit is defined on $SBV(\Omega)$: keeping in mind the interpretation above, of finite differences as deformation gradients, $\phi_i^\e(\{u^{i+j}\})$ is assumed to be bounded from below by $|(u^i-u^j)/\e|^p$ whenever this quantity is of order $1$, and otherwise by $1/\e$; 
\item\label{intro:H5} mild non-locality: the potential energies $\phi_{i}^{\e}$ of different deformations that agree in a cube of side length $\alpha$ centred at a point $i$ are comparable up to an error that vanishes for large $\alpha$ as $\e\to 0$ uniformly in $i$. This ensures that the $\Gamma$-limit is a local integral functional;
\item\label{intro:H6} controlled non-convexity: the energy stored by a convex combination of two deformations is asymptotically controlled by the sum of the energies corresponding to each single deformation. This technical assumption allows us to use the abstract methods of $\Gamma$-convergence (see below) and is needed here to tame the effect of the possibly diverging number of multi-body interactions.
\end{enumerate} 
We take the discrete-to-continuum limit of the energies in \eqref{intro:energy} under this set of assumptions. To this end we regard a discrete field $u$ as belonging to $L^{1}(\Omega;\R^{d})$ by identifying it with its piecewise constant interpolation on the cells of the $\e$ lattice. Outside this set of functions we extend $F_{\e}$ to $L^{1}(\Omega;\R^{d})$ by setting it equal to $+\infty$. We then define the discrete-to-continuum limit of $F_{\e}$ as its $\Gamma$-limit as $\e\to 0$ with respect to the strong $L^{1}$-convergence. 
\subsection{Main results, methods of proof and comparison with existing results}
In this paper we prove compactness, integral-representation and homogenization results for energies of the form \eqref{intro:energy}. More precisely, in Theorem \ref{thm:int:rep} we show that, up to subsequences, the discrete energies $F_{\e}$ $\Gamma$-converge to a free-discontinuity functional of the type \eqref{intro:energy-lim}. Using this integral representation we then prove the homogenization Theorem \ref{thm:homogenization}. There we show that under additional assumptions on $\phi_i^\e$ which will be discussed at the end of this section the whole sequence $(F_\e)$ $\Gamma$-converges to
\begin{align}\label{intro:energy-hom}
F_{\rm hom}(u)=\int_\Omega f_{\rm hom}(\nabla u)\dx+\int_{S_u}g_{\rm hom}(u^+-u^-,\nu_u)\dHn,
\end{align}
where $f_{\rm hom}$ and $g_{\rm hom}$ are some homogenized bulk and surface-energy densities, respectively.

\medskip

The proof of Theorem \ref{thm:int:rep} relies on the so-called localization method of $\Gamma$-convergence (see \cite[Chapters 14--20]{DalMaso93} and also \cite[Chapter 16]{braides02}). 
Following this method we consider energies $F_\e$ as functions defined both on $u$ and on the open subsets of $\Omega$ by defining for every pair $(u,A)$ with $u:Z_\e(\Omega)\to\Rd$ and $A\subset\Omega$ open the localized energy $E_\e(u,A)$ according to \eqref{intro:energy} where now the sum is taken only over $i\in Z_\e(A)$. We then prove a general compactness result (Theorem \ref{thm:gbar:comp}) which ensures that for every sequence of positive numbers converging to zero there exist a subsequence $(\e_j)$ and a functional $F$ such that for every $A\subset\Omega$ open and with Lipschitz boundary the localized energies $F_{\e_j}(\cdot,A)$ $\Gamma$-converge to $F(\cdot,A)$. Subsequently, thanks to assumptions \ref{intro:H1}--\ref{intro:H6} we recover enough information on $F$ both as a function in $u$ and as a set function to write it as a free-discontinuity functional of the form \eqref{intro:energy-lim} by using the general integral-representation result in \cite{BFLM02}. Before we comment on the homogenization result below we give a short overview on the use of the localization method in the context of discrete systems.

The method was originally proposed by De Giorgi and has been successfully used in the context of homogenization of multiple integrals in the continuum setting (see \cite{braides98a} and references therein). It has been first adapted to study discrete-to-continuum limits in \cite{AC04} in the context of pairwise-interacting discrete systems modeling nonlinear hyper-elastic materials and giving rise to continuum functionals finite on Sobolev spaces of the form $\int_\Omega f(x,\nabla u)\dx$. After that the application of the localization method to discrete systems at a bulk scaling has been extended into several directions including stochastic lattices \cite{ACG11, CGR18}, more general interaction potentials \cite{BrSch13, BCS15, BK18} and has also been combined with dimension-reduction techniques \cite{ABC08}. The most general result for discrete systems on deterministic lattices with limit energies on Sobolev spaces is by now contained in \cite{BK18}. 

At the surface scaling the analysis of discrete systems has required the use of the abstract method for the first time in \cite{ACR15}. That paper derives the continuum domain-wall theory in ferromagnetism from pairwise interacting Ising-type spin systems on (possibly stochastic) lattices (see also \cite{BCR17} for thin films). The extension of this result to more general magnetic interactions has been considered in \cite{AG16}. There the authors give examples of systems not satisfying (the analog of) assumption \ref{intro:H5} whose discrete-to-continuum limit is a nonlocal functional (see also \cite{B00}). A first general result for discrete systems with multi-body and long-range interactions at this scaling has been obtained in \cite{BC17} in the context of spin-like systems with spatially modulated phases.

We point out that in the above mentioned papers the discrete energies under consideration involve either a pure bulk or a pure surface scaling. In order to obtain a $\Gamma$-limit of the type \eqref{intro:energy-lim} one needs to consider discrete energies where both scalings are present at the same time. In this case, however, it becomes more difficult to find the correct set of assumptions which makes the localization method applicable. A first result in this direction has been obtained in \cite{Ruf17}, where the author considers energies of the form \eqref{intro:energy} on a possibly stochastic lattice. The interaction potentials $\phi_i^\e$ however are independent of $i$ and $\e$, have finite range and depend on finitely many particles uniformly in $\e$. Moreover, they depend on the configuration $\{u^j\}_j$ through the set of discrete differences $\{|u^i-u^j|\}_{i,j}$. This type of dependence is essential to decouple the contribution of bulk and surface scalings in the continuum limit, which finally allows to prove the full $\Gamma$-limit result (without extraction of a subsequence) in the case of a stationary stochastic lattice. This is done by exploiting for the first time in the discrete setting the theory of maximal functions introduced in \cite{FMP} and used in \cite{BDV96} in the context of homogenization. This technique turns out to be useful also in the proof of the present homogenization result Theorem \ref{thm:homogenization}, which we finally describe below.

\medskip

Theorem \ref{thm:homogenization} falls into the framework of periodic homogenization and thus requires the restriction to a special class of periodic interaction-energy densities. As our interaction-energy densities at a point $i$ may depend on the whole configuration $\{u^{i+j}\}_{j\in Z_\e(\Omega-i)}$ the meaning of periodicity needs to be clarified. A proper definition of periodicity (at least in the interior of $\Omega$) is possible when restricting to finite-range interactions. This modeling assumption also helps to decouple the bulk and the surface scaling in the $\Gamma$-limit, which is central to characterize the homogenized integrands $f_{\rm  hom}$ and $g_{\rm hom}$ in \eqref{intro:energy-hom}. We highlight that even under the finite-range assumption this task still requires a major effort due to the lack of a gradient structure in the interaction potentials. In fact, a crucial step in proving the homogenization result consists in establishing sufficient conditions on the potential $\phi_i^\e$ (without enforcing an explicit gradient structure) which make it possible to distinguish between the discretization of a macroscopic affine deformation of the form $u_M(x)=Mx$ with $M\in\R^{d\times n}$ and of a macroscopic jump, that is, a mapping of the form $u_\zeta(x_1,\ldots,x_n)=\zeta\chi_{\{x_n>0\}}$ with $\zeta\in\Rd$. More in detail, to derive formulas for the homogenized integrands $f_{\rm hom}$ and $g_{\rm hom}$ in \eqref{intro:energy-hom} it is essential that the potentials $\phi_i^\e$ reflect the different scaling properties of $u_M$ and $u_\zeta$ when passing from the scaled lattice $\e\Z^n$ to the integer lattice $\Z^n$. Indeed, the affine function $u_M$ satisfies $u_M(j)=\e\, u_M(j/\e)$ for every $j\in\e\Z^n$, while for the jump function $u_\zeta$ there holds $u_\zeta(j)=u_\zeta(j/\e)$ for every $j\in\e\Z^n$. It thus seems natural to require that for a given discrete function $u:\Z^n\to\Rd$ and $i\in\Z^n$ asymptotically there holds
\begin{align*}
\e^n\phi_{\e i}^\e(\{\e u^{j/\e}\})\sim\e^n\psi_i^b(\{u^j\}),\qquad\e^n\phi_{\e i}^\e(\{u^{j/\e}\})\sim\e^{n-1}\psi_i^s(\{u^j\}),
\end{align*}
for some discrete bulk and surface potentials $\psi_i^b$, $\psi_i^s$. This heuristic argument is made rigorous in Section \ref{sect:separation}, where we carefully state the correct hypotheses on the interaction potentials and we refer the reader to this section for more details.
\subsection{Plan of the paper}
The paper is organized as follows. In Section \ref{SOP} we recall some basic notation and we introduce the discrete functionals under consideration together with the precise assumptions on the potential $\phi_i^\e$. Section \ref{sect:int:rep} is then devoted to the proof of the integral-representation Theorem \ref{thm:int:rep} and to the treatment of Dirichlet boundary problems. The latter allows us to obtain asymptotic minimization formulas for the integrands $f$ and $g$ in \eqref{intro:energy-lim} (see Remark \ref{rem:dirichlet}), which are a key ingredient to prove the homogenization result Theorem \ref{thm:homogenization}. This is done in Section \ref{sect:homogenization}, where we also state precisely the periodicity- and the separation-of-scales assumptions. We conclude the paper by giving some examples that fall into the framework of our discrete energies in Section \ref{sect:examples}.

\section{Setting of the problem}\label{SOP}
\noindent{\bf Notation.} Let $n\geq 1$ be a fixed integer and $\Omega\subset\R^n$ an open, bounded set with Lipschitz boundary. We denote by $\A(\Omega)$ the family of all open subsets of $\Omega$ and by $\Areg(\Omega)$ the family of all open subsets of $\Omega$ with Lipschitz boundary. 

Let $\{e_1,\ldots,e_n\}$ denote the standard orthonormal basis in $\Rn$.
If $\nu,\xi\in\Rn$ we use the notation $\scalarproduct{\nu}{\xi}$ for the scalar product between $\nu$ and $\xi$ and by $|\nu|:=\sqrt{\scalarproduct{\nu}{\nu}}$ and $|\nu|_\infty:=\sup_{1\leq k\leq n}|\scalarproduct{\nu}{e_k}|$ we denote the euclidian norm and the supremum norm of $\nu$, respectively. 
Moreover, we set $S^{n-1}:=\{\nu\in\Rn\colon |\nu|=1\}$ and for every $\nu\in S^{n-1}$ we denote by $\Pi_\nu:=\{x\in\Rn\colon \scalarproduct{x}{\nu}=0\}$ the hyperplane orthogonal to $\nu$ and passing through the origin and $p_\nu:\Rn\to\Pi_\nu$ is the orthogonal projection onto $\Pi_\nu$.
Further, $Q^\nu$ denotes a unit cube centered at the origin and with one face orthogonal to $\nu$, and for every $x_0\in\Rn$ and $\rho>0$ we set $Q_\rho^\nu(x_0):=x_0+\rho Q^\nu$. If $\nu=e_k$ for some $k\in\{1,\ldots,n\}$ we simply write $Q$ and $Q_\rho(x_0)$ in place of $Q^{e_k}$ and $Q_\rho^{e_k}(x_0)$.

For every $A\subset\Rn$ we write $|A|$ for the $n$-dimensional Lebesgue measure of $A$, while $\Hn$ denotes the $(n-1)$-dimensional Hausdorff measure in $\Rn$. If $p\in[1,+\infty]$ and $d\geq 1$ is a fixed integer we use standard notation for Lebesgue spaces $L^p(\Omega;\Rd)$ and Sobolev spaces $W^{1,p}(\Omega;\Rd)$. Moreover, $SBV(\Omega;\R^d)$ denotes the space of $\Rd$-valued special functions of bounded variation in $\Omega$ (see, \eg \cite{AFP} for the general theory). If $u\in SBV(\Omega;\Rd)$ we write $\nabla u$ for the approximate gradient of $u$, $S_u$ for the approximate discontinuity set of $u$ and $\nu_u$ is the generalized outer normal to $S_u$. Moreover, $u^+$ and $u^-$ are the the traces of $u$ on both sides of $S_u$ and we set $[u]:=u^+-u^-$. We also consider the larger space $GSBV(\Omega;\Rd)$ defined as the space of all functions $u:\Omega\to\Rd$ such that $\varphi\circ u\in SBV_\loc(\Omega;\Rd)$ for every $\varphi\in C^1(\Rd;\Rd)$ with $\supp(\nabla\varphi)\wcont\Rd$. For $p\in (1,+\infty)$ it is also convenient to consider the spaces
\[SBV^p(\Omega;\Rd):=\{u\in SBV(\Omega;\Rd)\colon \nabla u\in L^p(\Omega;\Rdn),\ \Hn(S_u)<+\infty\}\]
and
\[GSBV^p(\Omega;\Rd):=\{u\in GSBV(\Omega;\Rd)\colon \nabla u\in L^p(\Omega;\Rdn),\ \Hn(S_u)<+\infty\}.\]
Note that $GSBV^p(\Omega;\Rd)$ is a vector space and for every $u\in GSBV^p(\Omega;\Rd)$ and $\varphi\in C^1(\Rd;\Rd)$ with $\supp(\nabla\varphi)\wcont\Rd$ there holds $\varphi\circ u\in SBV^p(\Omega;\Rd)\cap L^\infty(\Omega;\Rd)$ (see, \eg \cite[Section 2]{DMFT05}).

For $x_0\in\Rn$, $\nu\in S^{n-1}$, $\zeta\in\Rd$ and $M\in\R^{d\times n}$ we will frequently consider the jump function $u_{\zeta,x_0}^\nu:\Rn\to\Rd$ and the affine function $u_{M,x_0}:\Rn\to\Rd$ defined by setting
\begin{align}\label{def:boundarydata}
u_{\zeta,x_0}^\nu(x):=
\begin{cases}
\zeta &\text{if}\ \langle x-x_0,\nu\rangle\geq 0,\\
0 &\text{if}\ \langle x-x_0,\nu\rangle <0,
\end{cases}
\qquad\text{and}\qquad u_{M,x_0}(x):=M(x-x_0),
\end{align}
for every $x\in\Rn$.

\medskip

{\bf Setting.} In all that follows $\e>0$ denotes a parameter varying in a strictly decreasing sequence of positive real numbers converging to zero. For any $\e>0$, $u:\Rn\to\Rd$, $\xi\in\Z^n$ and $x\in\Rn$ we denote by
\[D_\e^\xi u(x):=\frac{u(x+\e\xi)-u(x)}{\e|\xi|}\]
the difference quotient of $u$ at $x$ in direction $\xi$. If $\xi=e_k$ for some $k\in\{1,\ldots,n\}$ we write $D_\e^k u(x)$ in place of $D_\e^{e_k}u(x)$.

We now introduce the discrete functionals considered in this paper. To this end, for every $A\subset\R^n$ let $Z_\e(A):=A\cap\e\Z^n$ and set  $\A_\e(\Omega;\Rd):=\{u: Z_\e(\Omega)\to\Rd\}$. It is then convenient to identify discrete functions $u\in\A_\e(\Omega;\Rd)$ with their piecewise-constant counterpart belonging to $L^1(\Omega;\Rd)$ defined by setting
\begin{equation}\label{def:pcint}
u(x):=u(i)=:u^i\quad\text{for every}\ x\in i+[0,\e)^n,\ i\in Z_\e(\Omega).
\end{equation}
If $(u_\e)$ is a sequence in $\A_\e(\Omega;\Rd)$  we say that $(u_\e)$ converges in $L^1(\Omega;\Rd)$ to a function $u\in L^1(\Omega;\Rd)$ if the sequence of the piecewise-constant interpolations of $u_\e$ defined as in \eqref{def:pcint} does so.

Finally, for every $i\in Z_\e(\Omega)$ it is convenient to consider the translated set $\Omega_i:=\Omega-i$. We then consider functions $\phi_i^\e:(\R^d)^{Z_\e(\Omega_i)}\to [0,+\infty)$ and we define the discrete functionals $F_\e:L^1(\Omega;\R^d)\times\A(\Omega)\to [0,+\infty]$ as
\begin{equation}\label{def:energy}
F_\e(u,A):=
\begin{cases}
\displaystyle\sum_{i\in Z_\e(A)}\e^n\phi_i^\e(\{u^{i+j}\}_{j\in Z_\e(\Omega_i)}) &\text{if}\ u\in\A_\e(\Omega;\Rd),\\
+\infty &\text{otherwise in}\ L^1(\Omega;\Rd).
\end{cases}
\end{equation}
In the case $A=\Omega$ we omit the dependence on the set and simply write $F_\e(u)$ in place of $F_\e(u,\Omega)$.
With the identification as in \eqref{def:pcint} and the corresponding $L^1(\Omega;\Rd)$-convergence we aim to describe the $\Gamma$-limit of the functionals $F_\e$ in the strong $L^1(\Omega)$-topology under suitable conditions on the energy densities $\phi_i^\e$.
Namely, we assume that the functions $\phi_i^\e:(\R^d)^{Z_\e(\Omega_i)}\to [0,+\infty)$ satisfy the following hypotheses for every $\e>0$ and $i\in Z_\e(\Omega)$.
\begin{enumerate}[label={\rm (H\arabic*)}]
\item (translational invariance)\label{H1} For all $w\in\R^d$ and $z:Z_\e(\Omega_i)\to\R^d$,
\[\phi_i^\e(\{z^j+w\}_{j\in Z_\e(\Omega_i)})=\phi_i^\e(\{z^j\}_{j\in Z_\e(\Omega_i)});\]
\item (monotonicity)\label{H6} for all $z,w: Z_\e(\Omega_i)\to\R^d$ with $|z^j-z^l|\leq|w^j-w^l|$ for every $j,l\in Z_\e(\Omega_i)$ we have
\[\phi_i^\e(\{z^j\}_{j\in Z_\e(\Omega_i)})\leq\phi_i^\e(\{w^j\}_{j\in Z_\e(\Omega_i)});\]
\item (upper bound for linear functions)\label{H2} there exist $c_1>0$ and $p\in (1,+\infty)$ such that for every $M\in\R^{d\times n}$ we have
\[\phi_i^\e(\{(Mx)^j\}_{j\in Z_\e(\Omega_i)})\leq c_1 (|M|^p+1),\]
where by $(Mx)$ we denote the linear function defined by $(Mx)^j:=Mj$;
\item (lower bound)\label{H3} there exists $c_2>0$ such that
\[\phi_i^\e(\{z^j\}_{j\in Z_\e(\Omega_i)})\geq c_2\min\left\{\sum_{k=1}^n|D_\e^kz(0)|^p,\frac{1}{\e}\right\},\]
for all $i\in Z_\e(\Omega)$ with $i+\e e_k\in Z_\e(\Omega)$ for every $k\in\{1,\ldots, n\}$ and every $z: Z_\e(\Omega_i)\to\R^d$.
\end{enumerate}
Moreover, we require that the following is satisfied.
\begin{enumerate}[label={\rm (H\arabic*)}]
\setcounter{enumi}{4}
\item (mild non-locality)\label{H4} For every $\e>0$, $\alpha\in\N$, $j\in Z_\e(\Rn)$ and $\xi\in\Z^n$ there exists $c_{\e,\alpha}^{j,\xi}\geq 0$ such that for every $i\in Z_\e(\Omega)$ and for all $z,w:Z_\e(\Omega_i)\to\R^d$ with $z^j=w^j$ for all $j\in Z_\e(\e \alpha Q)$ there holds
\begin{align*}
\phi_i^\e(\{z^j\}_{j\in Z_\e(\Omega_i)}) &\leq \phi_i^\e(\{w^j\}_{j\in Z_\e(\Omega_i)})\\
&+\sum_{j\in Z_\e(\Omega_i)}\sum_{\begin{smallmatrix}\xi\in\Z^n\\j+\e\xi\in\Omega_i\end{smallmatrix}}c_{\e,\alpha}^{j,\xi}\min\left\{|D_\e^\xi z(j)|^p,\frac{1+|z(j+\e\xi)-w(j+\e\xi)|}{\e}\right\},
\end{align*}
and the sequence $(c_{\e,\alpha}^{j,\xi})$ satisfies that following:
\begin{equation}\label{H4:summability1}
\limsup_{\e\to 0}\sum_{\alpha\in\N}\sum_{j\in Z_\e(\R^n)}\sum_{\xi\in \Z^n}c_{\e,\alpha}^{j,\xi}<+\infty
\end{equation}
and for every $\eta>0$ there exists a sequence $(M_\eta^\e)$ with $\e M_\eta^\e\to 0$ as $\e\to 0$ such that
\begin{equation}\label{H4:summability2}
\limsup_{\e\to 0}\sum_{\max\left\{\alpha,\frac{1}{\e}|j|,|\xi|\right\}> M_\eta^\e}c_{\e,\alpha}^{j,\xi}<\eta;
\end{equation}
\item (controlled non-convexity)\label{H5} there exists $c_3>0$ and for every $\e>0$, $j\in Z_\e(\Rn)$ and $\xi\in\Z^n$ there exists $c_\e^{j,\xi}\geq 0$ with
\begin{equation}\label{H5:summability}
\limsup_{\e\to 0}\sum_{j\in Z_\e(\R^n)}\sum_{\xi\in \Z^n}c_{\e}^{j,\xi}<+\infty
\end{equation}
such that for all $i\in Z_\e(\Omega)$, every $z,w:Z_\e(\Omega_i)\to\R^d$ and every cut-off $\varphi:\Rn\to[0,1]$ we have
\begin{align*}
\phi_i^\e(\{\varphi^jz^j+(1-\varphi^j)w^j\}_{j\in Z_\e(\Omega_i)}) &\leq c_3\left(\phi_i^\e(\{z^j\}_{j\in Z_\e(\Omega_i)})+\phi_i^\e(\{w^j\}_{j\in Z_\e(\Omega_i)})\right)\\
&+R_i^\e(z,w,\varphi),
\end{align*}
where
\begin{align*}
R_i^\e(z,w,\varphi) &:=\sum_{j\in Z_\e(\Omega_i)}\sum_{\begin{smallmatrix}\xi\in\Z^n\\j+\e\xi\in\Omega_i\end{smallmatrix}}c_\e^{j,\xi}\Bigg(\sup_{\begin{smallmatrix}l\in Z_\e(\Omega_i)\\k\in\{1,\ldots,n\}\end{smallmatrix}}|D_\e^k\varphi(l)|^p|z(j+\e\xi)-w(j+\e\xi)|^p\Bigg)\\
&+\sum_{j\in Z_\e(\Omega_i)}\sum_{\begin{smallmatrix}\xi\in\Z^n\\j+\e\xi\in\Omega_i\end{smallmatrix}}c_\e^{j,\xi}\Bigg(\min\left\{|D_\e^\xi z(j)|^p,\frac{1}{\e|\xi|}\right\}+\min\left\{|D_\e^\xi w(j)|^p,\frac{1}{\e|\xi|}\right\}\Bigg);
\end{align*} 
\end{enumerate}
\begin{rem}\label{rem:hypotheses:phi}
Hypotheses \ref{H1} together with \ref{H2} imply that for every $\e>0$, $i\in Z_\e(\Omega)$ and for any constant function $z:Z_\e(\Omega_i)\to\R^d$, $z^j=w$ for all $j\in Z_\e(\Omega_i)$ we have
\begin{align}\label{upperbound:constant}
\phi_i^\e(\{z^j\}_{j\in Z_\e(\Omega_i)})=\phi_i^\e(\{0j+w\}_{j\in\Z_\e(\Omega_i)})=\phi_i^\e(\{0j\}_{j\in\Z_\e(\Omega_i)})\leq c_1+1.
\end{align}
Note that the condition on the decaying tail of the sequence $(c_{\e,\alpha}^{j,\xi})$ in \ref{H4} is slightly more genereal then the corresponding conditions in \cite{AC04} and \cite{BK18}. In fact, therein the authors choose for every $\eta>0$ a constant $M_\eta>0$ uniformly in $\e$ such that the analog of \eqref{H4:summability2} is satisfied. Here we show that this assumption can be weakened by allowing $M_\eta^\e$ to depend on $\e$ as long as $\e M_\eta^\e\to 0$. This weaker condition makes it possible to rephrase an example considered in \cite{B00} in our framework (see Section \ref{sect:example:nonlocal}).
\end{rem}
\begin{rem}[Smooth truncation]\label{rem:truncation}
In order to apply \ref{H6} we will need to truncate $\R^d$-valued functions in a suitable way. To this end, following the approach in \cite{CDSZ17} we consider $\varphi\in C_c^\infty(\R)$ with $\varphi(t)=t$ for all $t\in\R$ with $|t|\leq 1$, $\varphi(t)=0$ for all $t\geq 3$ and $\|\varphi'\|_\infty\leq 1$ and we define $\phi\in C_c^\infty(\R^d;\R^d)$ by setting
\[\phi(\zeta):=
\begin{cases}
\varphi(|\zeta|)\frac{\zeta}{|\zeta|} &\text{if}\ \zeta\neq 0,\\
0 &\text{if}\ \zeta=0.
\end{cases}\]
The function $\phi$ is $1$-Lipschitz \cite[Section 4]{CDSZ17} and for every $k\in\N$ the function $\phi_k$ defined as $\phi_k(\zeta):=k\phi(\frac{\zeta}{k})$ is also $1$-Lipschitz. In particular, since $\phi_k(0)=0$, we have 
\begin{equation}\label{est:trunc1}
|\phi_k(\zeta)|\leq|\zeta|\ \text{for every}\ \zeta\in\R^d.
\end{equation}
For every $u:\Rn\to\Rd$ we now define the truncation $T_ku:=\phi_k(u)$ and we observe that thanks to \eqref{est:trunc1} and the $1$-Lipschitzianity of $\phi_k$ \ref{H6} yields
\begin{equation}\label{est:trunc2}
F_\e(T_ku,A)\leq F_\e(u,A),
\end{equation}
for every $k\in\N$, $\e>0$, $A\in\A(\Omega)$ and $u\in \A_\e(\Omega;\Rd)$.
Moreover, for every $u\in GSBV^p(\Omega;\R^d)$ and every $k\in\N$ the truncation $T_ku$ belongs to $SBV^p(\Omega;\R^d)\cap L^\infty(\Omega;\R^d)$ and $\|T_ku\|_{L^\infty}\leq 3k$. Finally, if $u\in GSBV^p(\Omega;\Rd)\cap L^1(\Omega;\Rd)$ there holds (see \cite[Lemma 2.1]{Ruf17})
\begin{enumerate}[label=(\roman*)]
\item $T_ku\to u$ a.e.~and in $L^1(\Omega;\R^d)$ as $k\to +\infty$,
\item $\nabla T_ku(x)=\nabla\phi_k(u(x))\nabla u(x)$ and in particular $|\nabla T_ku(x)|\leq|\nabla u(x)|$ for a.e.~$x\in\Omega$ and every $k\in\N$,
\item $S_{T_ku}\subset S_u$ and $([u],\nu_{u})=([T_k u],\nu_{T_ku})$ $\Hn$-a.e.~on $S_{T_ku}$ up to a simultaneous change of sign of $[T_ku]$ and $\nu_{T_ku}$, and by Lipschitzianity $|(T_ku)^+-(T_ku)^-|\leq |u^+-u^-|$ for every $k\in\N$. Moreover $\lim_{k\to +\infty}\Hn(S_{T_ku})=\Hn(S_u)$.
\end{enumerate}
\end{rem}
\begin{rem}[$\Gamma$-liminf and $\Gamma$-limsup]\label{rem:g-li-ls}
In all that follows we use standard notation for the $\Gamma$-liminf and the $\Gamma$-limsup, \ie for every pair $(u,A)\in L^1(\Omega;\Rd)\times\A(\Omega)$ we set
\begin{align*}
F'(u,A) &:=\Gamma\hbox{-}\liminf_{\e\to 0}F_\e(u,A):=\inf\{\liminf_{\e\to 0}F_\e(u_\e,A)\colon u_\e\to u\ \text{in}\ L^1(\Omega;\Rd)\},\\
F''(u,A) &:=\Gamma\hbox{-}\limsup_{\e\to 0}F_\e(u,A):=\inf\{\limsup_{\e\to 0}F_\e(u_\e,A)\colon u_\e\to u\ \text{in}\ L^1(\Omega;\Rd)\}.
\end{align*}
If $A=\Omega$ we write $F'(u)$ and $F''(u)$ in place of $F'(u,\Omega)$ and $F''(u,\Omega)$.

The functional $F'$ is superadditive as a set function \cite[Proposition 16.12]{DalMaso93} and both the functionals $F'$ and $F''$ are increasing as set functions \cite[Proposition 6.7]{DalMaso93} and $L^1(\Omega;\Rd)$-lower semicontinuous in $u$ \cite[Proposition 6.8]{DalMaso93}. Moreover, from \eqref{est:trunc2} we deduce that $F'(T_ku,A)\leq F'(u,A)$ and $F''(T_ku,A)\leq F''(u,A)$ for every $(u,A)\in L^1(\Omega;\Rd)\times\A(\Omega)$ and $k\in\N$. Hence, the $L^1(\Omega;\Rd)$-lower semicontinuity together with (i) in Remark \ref{rem:truncation} ensure that
\begin{align}\label{conv:trunc}
\lim_{k\to +\infty}F'(T_ku,A) &=F'(u,A),\nonumber\\
\lim_{k\to +\infty}F''(T_ku,A) &=F''(u,A).
\end{align}
Finally, we also consider the inner-regular envelopes of $F'$ and $F''$ defined as
\begin{align}\label{def:irenv}
F'_{-}(u,A) &:=\sup\{F'(u,A')\colon A'\in\A(\Omega),\ A'\wcont A\},\nonumber\\
F''_{-}(u,A) &:=\sup\{F''(u,A')\colon A'\in\A(\Omega),\ A'\wcont A\},
\end{align}
respectively. Then $F'_{-}$ and $F''_{-}$ are inner regular by definition, increasing and $L^1(\Omega;\Rd)$-lower semicontinuous \cite[Remark 15.10]{DalMaso93}.
\end{rem}
\section{Compactness and integral representation}\label{sect:int:rep}
In this section we state and prove the first main result of the paper, which is the following integral-representation result for the $\Gamma$-limit of the functionals $F_\e$.
\begin{theorem}[Integral representation]\label{thm:int:rep}
Let $F_\e$ be as in \eqref{def:energy} and suppose that $\phi_i^\e:(\R^d)^{Z_\e(\Omega_i)}\to[0,+\infty)$ satisfy \ref{H1}-\ref{H5}. For every sequence of positive numbers converging to $0$ there exists a subsequence $(\e_j)$ such that $(F_{\e_j})$ $\Gamma$-converges to a functional $F:L^1(\Omega;\R^d)\to[0,+\infty]$ of the form
\begin{equation}\label{int:form}
F(u)=
\begin{cases}
\displaystyle\int_\Omega f(x,\nabla u)\dx+\int_{S_u}g(x,[u],\nu_u)\dHn &\text{if}\ u\in GSBV^p(\Omega;\R^d)\cap L^1(\Omega;\Rd),\\
+\infty &\text{otherwise in}\ L^1(\Omega;\R^d).
\end{cases}
\end{equation}
Here, for every $x_0\in\Rn$, $\nu\in S^{n-1}$, $\zeta\in\Rd$ and $M\in\R^{d\times n}$ the integrands are given by the formulas
\begin{equation}\label{derivationformula}
f(x_0,M) =\limsup_{\rho\to 0}\frac{1}{\rho^n}\m(u_{M,x_0},Q_\rho^\nu(x_0)),\qquad
g(x_0,\zeta,\nu) =\limsup_{\rho\to 0}\frac{1}{\rho^{n-1}}\m(u_{\zeta,x_0}^\nu, Q_\rho^\nu(x_0)),
\end{equation}
where $u_{M,x_0}, u_{\zeta,x_0}^\nu$ are given by \eqref{def:boundarydata} and for every $\bar{u}\in SBV^p(\Omega;\R^d)$ and every $A\in\Areg(\Omega)$ we have set
\begin{align}\label{def:min:prob}
\m(\bar{u},A):=\inf\{F(u,A)\colon u\in SBV^p(A;\R^d),\ u=\bar{u}\ \text{in a neighborhood of}\ \partial A\}.
\end{align}
In particular, $g(x,t,\nu)=g(x,-t,-\nu)$ for every $(x,t,\nu)\in\Omega\times\R^d\times S^{n-1}$. 
Moreover, for every $A\in\Areg(\Omega)$ and every $u\in GSBV^p(\Omega;\Rd)$ there holds
\begin{align}\label{localrepresentation}
\Gamma\hbox{-}\lim_{j\to+\infty}F_{\e_j}(u,A)=\int_A f(x,\nabla u)\dx+\int_{S_u\cap A}g(x,[u],\nu_u)\dHn.
\end{align}
\end{theorem}
\noindent We will prove Theorem \ref{thm:int:rep} gathering Propositions \ref{prop:loc}, \ref{prop:lb}, \ref{prop:ub}, \ref{prop:subad} and \ref{prop:ir}  below which together with the general compactness result Theorem \ref{thm:gbar:comp} ensure that the $\Gamma$-limit $F$ exists up to subsequences and that a suitable perturbation of $F$ satisfies all hypotheses of \cite[Theorem 1]{BFLM02}. As a first step we show that $F''(\cdot,A)$ is local for every $A\in\Areg(\Omega)$.
\begin{prop}[Locality]\label{prop:loc}
Let $\phi_i^\e:(\R^d)^{Z_\e(\Omega_i)}\to[0,+\infty)$ satisfy hypotheses \ref{H1}-\ref{H5}. Then for any $A\in\Areg(\Omega)$ and $u,v\in GSBV^p(\Omega;\Rd)\cap L^1(\Omega;\Rd)$ with $u=v$ a.e.~in $A$ we have
\[F''(u,A)=F''(v,A).\]
\end{prop}
\begin{proof}
Let $A,u,v$ be as in the statement. Thanks to \eqref{conv:trunc} it suffices to consider the case $u,v\in SBV^p(\Omega;\Rd)\cap L^\infty(\Omega;\Rd)$. We first show that $F''(u,A)\leq F''(v,A)$. To this end, choose $u_\e,v_\e\in\A_\e(\Omega;\Rd)$ converging in $L^1(\Omega;\Rd)$ to $u,v$, respectively  and satisfying
\begin{equation}\label{conv:locality}
\lim_{\e\to 0} F_\e(u_\e,A)=F''(u,A),\qquad\lim_{\e\to 0} F_\e(v_\e,A)=F''(v,A).
\end{equation}
Up to considering the truncated functions $T_{\|u\|_{L^\infty}}u_\e$, $T_{\|v\|_{L^\infty}}v_\e$ we can assume that $\|u_\e\|_{L^\infty}\leq 3\|u\|_{L^\infty}$, $\|v_\e\|_{L^\infty}\leq 3\|u\|_{L^\infty}$.

For fixed $\eta>0$ and every $\e>0$ let $M_\eta^\e>0$ be given by \eqref{H4:summability2} and define $w_\e\in\A_\e(\Omega;\Rd)$ by setting
\begin{align*}
w_\e^i:=
\begin{cases}
v_\e^i &\text{if}\ \dist_\infty(i,A)\leq\e M_\eta^\e,\\
u_\e^i &\text{otherwise in}\ Z_\e(\Omega).
\end{cases}
\end{align*}
Since the sequences $(u_\e),(v_\e)$ are bounded in $L^\infty(\Omega;\Rd)$ uniformly in $\e$ and $u=v$ a.e.~in $A$ we have
\begin{align*}
\|w_\e-u\|_{L^1(\Omega)}\leq\|v_\e-v\|_{L^1(A)}+\|u_\e-u\|_{L^1(\Omega\setminus A)}+c\e^n\#\{i\in Z_\e(\Omega)\colon\dist(i,\partial A)<\e M_\eta^\e\}.
\end{align*}
Moreover, since $\partial A$ is Lipschitz, it admits an upper Minkowsky content, hence
\begin{align*}
(\e M_\eta^\e)^{n-1}\#\{i\in Z_\e(\Omega)\colon\dist(i,\partial A)<\e M_\eta^\e\}\leq c\Hn(\partial A)+o_{\e M_\eta^\e}(1).
\end{align*}
Thus, the assumption on $M_\eta^\e$ ensures that $w_\e\to u$ in $L^1(\Omega;\Rd)$, which implies that
\begin{align}\label{est:locality}
F''(u,A)\leq\limsup_{\e\to 0}F_\e(w_\e,A).
\end{align}
We now come to estimate $F_\e(w_\e,A)$. For every $i\in Z_\e(A)$ we set
\begin{align*}
\alpha_\e(i):=\sup\{\alpha\in\N\colon w_\e^j=v_\e^j\ \text{for every}\ j\in Z_\e(i+\e\alpha Q)\},
\end{align*}
so that condition \ref{H4} yields
\begin{align}\label{est:locality1}
&F_\e(w_\e,A)\leq\sum_{i\in Z_\e(A)}\e^n\phi_i^\e(\{v_\e^{i+j}\}_{j\in Z_\e(\Omega_i)})\nonumber\\
&\hspace{1em}+\sum_{i\in Z_\e(A)}\e^n\sum_{j\in Z_\e(\Omega_i)}\dsum{\xi\in\Z^n}{j+\e\xi\in\Omega_i}c_{\e,\alpha_\e(i)}^{j,\xi}\min\Bigg\{|D_\e^\xi w_\e^{i+j}|^p,\frac{1+|w_\e^{i+j+\xi}-v_\e^{i+j+\xi}|}{\e}\Bigg\}.
\end{align}
We observe that by construction $\alpha_\e(i)>M_\eta^\e$ for every $i\in Z_\e(A)$.
Estimating the minimum in \eqref{est:locality1} with $(1+|w_\e^{i+j+\xi}-v_\e^{i+j+\xi}|)/\e$ and using the uniform bound on $\|v_\e\|_{L^\infty}$ and $\|w_\e\|_{L^\infty}$ thus gives
\begin{align*}
F_\e(w_\e,A)\leq F_\e(v_\e,A)+(1+3\|u\|_{L^\infty}+3\|v\|_{L^\infty})\sum_{\alpha>M_\eta^\e}\sum_{j\in Z_\e(\Rn)}\sum_{\xi\in\Z^n}c_{\e,\alpha}^{j,\xi}\e^{n-1}\#\{i\in Z_\e(A)\colon \alpha_\e(i)=\alpha\}.
\end{align*}
Moreover, the Lipschitz regularity of $A$ yields
\begin{align*}
\e^{n-1}\#\{i\in Z_\e(A)\colon \alpha_\e(i)=\alpha\}\leq c\Hn(\partial A)
+o_\e(1),
\end{align*}
which in view of the choice of $M_\eta^\e$ and \eqref{H4:summability2} gives
\begin{align*}
\limsup_{\e\to 0} F_\e(w_\e,A)\leq\limsup_{\e\to +\infty} F_\e(v_\e,A)+c\eta.
\end{align*}
Gathering \eqref{conv:locality} and \eqref{est:locality} we thus obtain
\begin{align*}
F''(u,A)\leq F''(v,A)+c\eta,
\end{align*}
and the desired inequality follows by the arbitrariness of $\eta>0$.
\end{proof}
As a next step towards the proof of Theorem \ref{thm:int:rep} the following two propositions show that $F'$ and $F''$ are finite only on $GSBV^p(\Omega;\Rd)\cap L^1(\Omega;\Rd)$ and satisfy suitable growth conditions.
\begin{prop}[Lower bound]\label{prop:lb}
Let $F_\e$ be given by \eqref{def:energy} and suppose that the functions $\phi_i^\e:(\R^d)^{Z_\e(\Omega_i)}\to[0,+\infty)$ satisfy \ref{H3}. Let $A\in \Areg(\Omega)$ and $u\in L^1(\Omega;\R^d)$ with $F'(u,A)<+\infty$. Then $u\in GSBV^p(A;\R^d)$ and
\begin{equation}\label{lowerbound}
F'(u,A)\geq c\left(\int_A|\nabla u|^p\dx+\Hn(S_u\cap A)\right)
\end{equation}
for some $c>0$ independent of $u$ and $A$.
\end{prop}
\begin{proof}
Let $(u,A)\in L^1(\Omega;\R^d)\times\Areg(\Omega)$ be as in the statement and let $(u_\e)\subset\A_\e(\Omega;\Rd)$ be a sequence converging to $u$ in $L^1(\Omega;\R^d)$ and satisfying $\sup_\e F_\e(u_\e,A)<+\infty$. In view of \ref{H3} we have
\begin{equation}\label{est:lb:1}
F_\e(u_\e,A)\geq c_2\sum_{i\in Z_\e(A)}\e^n\min\left\{\sum_{k=1}^n|D_\e^ku_\e(i)|^p,\frac{1}{\e}\right\},
\end{equation}
hence \cite[Lemma 3.3]{Ruf17} applied to $\mathcal{L}=\Z^n$ and $f(p)=\min\{\|p\|_1,\frac{1}{\e}\}$ together with the uniform bound on $F_\e(u_\e,A)$ yield $u\in GSBV^p(A;\R^d)$. Moreover, from \cite[Lemma 3.3]{Ruf17} and \eqref{est:lb:1} we also deduce
\[\liminf_{\e\to 0}F_\e(u_\e,A)\geq c\left(\int_A|\nabla u|^p\dx+\Hn(S_u\cap A)\right)\]
for some $c>0$ independent of $u$ and $A$, hence \eqref{lowerbound} follows.
\end{proof}
In order to prove an upper bound for $F''(u)$ we need to restrict to a suitable dense class of functions. To this end, it is convenient to introduce the following definition of a regular triangulation.
\begin{definition}\label{def:triangualation}
Let $A\subset\R^n$ be open, bounded and with Lipschitz boundary. We say that a family $(U_l)_{l=1,\ldots,N}$ of pairwise disjoint open $n$-simplices $U_1,\ldots,U_N$ is a regular triangulation of $A$ if $A\subset\bigcup_{l=1}^N\overline{U}_l$ and if for any $(l,l')\in\{1,\ldots,N\}^2$ the intersection $S_{l,l'}:=\overline{U}_l\cap\overline{U}_{l'}$ is either the empty set or an $(n-k)$-dimensional simplex for some $k\in\{1,\ldots,n\}$. The $(n-1)$-dimensional simplices $S_{l,l'}$ are called the faces of the triangulation and by $\theta\in (0,\pi)$ we denote the minimal angle between two faces of such a triangulation.
\end{definition}
\begin{prop}[Upper bound]\label{prop:ub}
Let $A\in\Areg(\Omega)$ and $u\in GSBV^p(A;\R^d)\cap L^1(\Omega;\R^d)$ and suppose that the functions $\phi_i^\e$ satisfy \ref{H1}--\ref{H5}. Then
\begin{equation}\label{upperbound}
F''(u,A)\leq c\left(\int_A\left(|\nabla u|^p+1\right)\dx+\int_{S_u\cap A}(1+|u^+(y)-u^-(y)|)\dHn(y)\right)
\end{equation}
for some $c>0$ independent of $u$ and $A$.
\end{prop}
\begin{proof}
Let $\tomega\subset\Rn$ be any open bounded set with Lipschitz boundary such that $\Omega\wcont\tomega$.
\begin{step}{Step 1}
As a preliminary step we prove the existence of some constant $c>0$ such that for any $u\in SBV^p(\tomega;\R^d)\cap L^\infty(\tomega;\R^d)$ and any $A\in\Areg(\Omega)$ there holds
\begin{equation}\label{ub:prel}
F''(u,A)\leq c\left(\int_A\left(|\nabla u|^p+1\right)\dx+\int_{S_u\cap\clA}(1+|u^+(y)-u^-(y)|)\dHn(y)\right).
\end{equation}
We first prove \eqref{ub:prel} for $A$ polyhedral set.

Thanks to \cite[Theorem 3.1]{CT99} (see also \cite[Theorem 3.9]{cortesani97}), employing a standard density argument it suffices to prove \eqref{ub:prel} for $u\in SBV^p(\tomega;\R^d)\cap L^\infty(\tomega;\R^d)$ such that $S_u$ is essentially closed (i.e., $\Hn(\clsu\setminus S_u)=0$), $\clsu$ is the intersection of $\Omega$ with a finite union of $(n-1)$-dimensional simplices and $u\in W^{1,\infty}(\tomega\setminus\clsu;\R^d)$. Moreover, since $u\in W^{1,\infty}(\tomega\setminus\clsu;\R^d)$, arguing again by density we may assume that $u$ is piecewise affine on $\tomega\setminus\clsu$. 
More precisely, we may assume that there exist a regular triangulation $(U_l)_{l=1,\ldots,N}$ of $\tomega$ and $M_1,\ldots,M_N\in\R^{d\times n}$, $b_1,\ldots,b_N\in\R^d$ such that $u$ satisfies the following.
\begin{enumerate}[label=(\roman*)]
\item $u(x)=\sum_{l=1}^N\chi_{U_l\cap\tomega}(x)(M_lx+b_l)$ for any $x\in\tomega\cap\bigcup_{l=1}^NU_l$;
\item $\clsu=\tomega\cap\bigcup_{k=1}^K S_{l_k,l_k'}$, where $(S_{l_k,l_k'})_{k=1,\ldots,K}$ is a collection of faces of the triangulation;
\item for any face $S_{l,l'}$ with $(l,l')\neq (l_k,l_k')$ for every $k\in\{1,\ldots,K\}$ we have
\[u(x)= M_lx+b_l=M_{l'}x+b_{l'}\quad\text{for every}\ x\in S_{l,l'}.\]
\end{enumerate}
Since $A$ is a polyhedral set, up to refining the triangulation and renumbering the simplices we may also assume that 
\[\clA=\bigcup_{l=1}^L \overline{U}_l\]
for some $L<N$.
Finally, we can assume that $\bigcup_{l,l'} S_{l,l'}\cap\e\Z^n=\emptyset$, since otherwise we may consider the shifted lattice $\e\Z^n+\xi_\e$ for a suitable sequence $\xi_\e\to 0$. 
We then define a sequence $(u_\e)\subset\A_\e(\tomega;\Rd)$ by setting
\[u_\e^i:=u(i)\quad\text{for every}\ i\in Z_\e(\tomega)\]
and we note that $u_\e\to u\in L^1(\Omega;\R^d)$. Moreover, we write
\begin{equation}\label{est:ub:01}
F_\e(u_\e,A)=\sum_{l=1}^L F_\e(u_\e,U_l),
\end{equation}
and we estimate $F_\e(u_\e,U_l)$ for every $l\in\{1,\ldots,L\}$. To this end, for $l\in\{1,\ldots,L\}$ fixed and for $i\in Z_\e(U_l)$ set
\[\alpha_\e^l(i):=\sup\{\alpha\in\N\colon u_\e^j=M_lj+b_l\ \text{for every}\ j\in i+\e\alpha Q\}.\]
Thanks to \ref{H1} and \ref{H4} we deduce
\begin{align}\label{est:ub:02}
&F_\e(u_\e,U_l)\leq\sum_{i\in Z_\e(U_l)}\e^n\phi_i^\e(\{(M_lx)(i+j)\}_{j\in Z_\e(\Omega_i)})\nonumber\\
&+\sum_{i\in Z_\e(U_l)}\hspace*{-0.5em}\e^n\sum_{j\in Z_\e(\Omega_i)}\hspace*{-0.5em}\sum_{\begin{smallmatrix}\xi\in\Z^n\\j+\e\xi\in \Omega_i\end{smallmatrix}}\hspace*{-0.5em} c_{\e,\alpha_\e^l(i)}^{j,\xi}\min\left\{|D_\e^\xi u(i+j)|^p,\frac{1+|u(i+j+\e\xi)-(M_lx+b_l)(i+j+\e\xi)|}{\e}\right\}\nonumber\\
&=:I_{\e,1}^l+I_{\e,2}^l.
\end{align}
Moreover, \ref{H2} gives 
\begin{equation}\label{est:ub:03}
I_{\e,1}^l\leq c_1\sum_{i\in Z_\e(U_l)}\e^n(|M|^p+1)=c_1\int_{U_l}\left(|\nabla u|^p+1\right)\dx+ o(1),
\end{equation}
so that it remains to estimate $I_{\e,2}^l$. To do so, we need to introduce some notation.
In what follows for $\e>0$, $i\in Z_\e(U_l)$, $j\in Z_\e(\Omega_i)$ and $\xi\in\Z^n$ we use the abbreviation
\[\m_{\e,l}^{j,\xi}u(i):=\min\left\{|D_\e^\xi u(i+j)|^p,\frac{1+|u(i+j+\e\xi)-(M_lx+b_l)(i+j+\e\xi)|}{\e}\right\}.\]
Further, by
\[\NN(l):=\big\{l'\in\{1,\ldots,N\}\colon S_{l,l'}\ \text{is an $(n-1)$-dimensional simplex}\big\}\]
we denote the set of all indices which label the ``neighboring'' simplices of $U_l$. 
Moreover, for $\eta>0$ fixed and every $\e>0$ we choose $M_\eta>0$ such that 
\[\limsup_{\e\to 0}\sum_{\max\left\{\alpha,\frac{1}{\e}|j|,|\xi|\right\}> M_\eta^\e}c_{\e,\alpha}^{j,\xi}<\eta,\]
and we find $m_\e\in\N$ such that $\e m_\e\to 0$ and $m_\e>\tfrac{4M_\eta\cos\theta	}{\sin\theta}$, where $\theta\in (0,\pi)$ is as in Definition \ref{def:triangualation}.
Finally, for any $l'\in\NN(l)$ set
\[\I^{l'}_{\e}:= \{i\in Z_\e(U_l)\colon\dist_\infty(i,U_{l'})\leq\e m_\e\}\]
and
\[\J^{l'}_{\e}:=Z_\e(U_l)\setminus\I^{l'}_{\e}.\]
Setting $U_l^\e:=\{x\in U_l\colon\dist_\infty(x,\Rn\setminus U_l)>\e\}$ we get
\[\bigcap_{l'\in\NN(l)}\J_{\e}^{l'}=Z_\e(U_l^\e).\]
For $l'\in\NN(l)$ we also set
\[\LL_{\e}^{l'}:=\bigcap_{\begin{smallmatrix}l''\in\NN(l)\\l''\neq l'\end{smallmatrix}}\J_{\e}^{l''}\]
and we rewrite $I_{\e,2}^l$ as
\begin{align}\label{est:ub:04}
I_{\e,2}^{l} &=\sum_{i\in Z_\e(U_l^\e)}\e^n\sum_{j\in Z_\e(\Omega_i)}\sum_{\begin{smallmatrix}\xi\in\Z^n\\j+\e\xi\in\Omega_i\end{smallmatrix}}c_{\e,\alpha_\e^l(i)}^{j,\xi}\m_{\e,l}^{j,\xi}u(i)\nonumber\\
&+\dsum{l',l''\in\NN(l)}{l'\neq l''}\sum_{i\in\I_{\e,m}^{l'}\cap\I_{\e,m}^{l''}}\e^n\sum_{j\in Z_\e(\Omega_i)}\sum_{\begin{smallmatrix}\xi\in\Z^n\\j+\e\xi\in\Omega_i\end{smallmatrix}}c_{\e,\alpha_\e^l(i)}^{j,\xi}\m_{\e,l}^{j,\xi}u(i)\nonumber\\
&+\sum_{l'\in\NN(l)}\sum_{i\in\I^{l'}_{\e}\cap\mathcal{L}_{\e}^{l'}}\e^n\sum_{j\in Z_\e(\Omega_i)}\sum_{\begin{smallmatrix}\xi\in\Z^n\\j+\e\xi\in\Omega_i\end{smallmatrix}}c_{\e,\alpha_\e^l(i)}^{j,\xi}\m_{\e,l}^{j,\xi}u(i).
\end{align}
In order to estimate the first term in \eqref{est:ub:04} we note that
\[\e^{n-1}\#\big\{i\in Z_\e(U_l^\e)\colon \alpha_\e^l(i)=\alpha\big\}\leq c\Hn(\partial U_l)+o_\e(1)\]
for every $\alpha\in\N$. Moreover, for every $i\in Z_\e(U_l^\e)$ we have $\alpha_\e^l(i)\geq 2m_\e$.
Thus, the estimate $\m_{\e,l}^{j,\xi}u(i)\leq (2\|u\|_{L^\infty}+1)\e^{-1}$ yields
\begin{align}\label{est:ub:05}
 &\sum_{i\in Z_\e(U_l^\e)}\e^n\sum_{j\in Z_\e(\Omega_i)}\sum_{\begin{smallmatrix}\xi\in\Z^n\\j+\e\xi\in\Omega_i\end{smallmatrix}}c_{\e,\alpha_\e^l(i)}^{j,\xi}\m_{\e,l}^{j,\xi}u(i)\nonumber\\
 &\leq(1+2\|u\|_{L^\infty})\sum_{\alpha\geq 2m}\sum_{j\in Z_\e(\R^n)}\sum_{\xi\in\Z^n}c_{\e,\alpha}^{j,\xi}\e^{n-1}\#\big\{i\in Z_\e(U_l^\e)\colon \alpha_\e^l(i)=\alpha\big\}\nonumber\\
 &\leq c(u)\sum_{\max\left\{\alpha,\frac{1}{\e}|j|,|\xi|\right\}> M_\eta^\e}c_{\e,\alpha}^{j,\xi}.
\end{align}
To bound the second term in \eqref{est:ub:04} we observe that for every $l',l''\in\NN(l)$ with $l'\neq l''$ and every $\alpha\in\N$ we have
\[\e^{n-1}\#\big\{i\in \I_{\e}^{l'}\cap\I_{\e}^{l''}\colon \alpha_\e^l(i)=\alpha\big\}\leq \e m_\e c\left(\Hn(S_{l,l'})+\Hn(S_{l,l''})+o_\e(1)\right).\]
Hence, as in \eqref{est:ub:05} we obtain
\begin{align}\label{est:ub:07}
&\dsum{l',l''\in\NN(l)}{l'\neq l''}\sum_{i\in\I_{\e,m}^{l'}\cap\I_{\e,m}^{l''}}\e^n\sum_{j\in Z_\e(\Omega_i)}\sum_{\begin{smallmatrix}\xi\in\Z^n\\j+\e\xi\in\Omega_i\end{smallmatrix}}c_{\e,\alpha_\e^l(i)}^{j,\xi}\m_{\e,l}^{j,\xi}u(i)\nonumber\\
&\leq(1+2\|u\|_{L^\infty})\dsum{l',l''\in\NN(l)}{l'\neq l''}\sum_{\alpha\in\N}\sum_{j\in Z_\e(\Rn)}\sum_{\xi\in\Z^n}c_{\e,\alpha}^{j,\xi}\e^{n-1}\#\big\{i\in \I_{\e,m}^{l'}\cap\I_{\e,m}^{l''}\colon \alpha_\e^l(i)=\alpha\big\}\nonumber\\
&\leq c(u,n)\e m\sum_{\alpha\in\N}\sum_{j\in Z_\e(\Rn)}\sum_{\xi\in\Z^n}c_{\e,\alpha}^{j,\xi}\ \to 0\ \text{as}\ \e\to 0.
\end{align}
Finally, the last term in \eqref{est:ub:04} can be estimated as follows. If $j\in\e\Z^n$ and $\xi\in\Z^n$ are such that $\max\{\tfrac{1}{\e}|j|,|\xi|\}\geq \tfrac{m_\e\sin\theta}{4\cos\theta}$ then the choice of $m_\e$ allows us to deduce that
\begin{align}\label{est:ub:08a}
&\dsum{l',l''\in\NN(l)}{l'\neq l''}\sum_{i\in\I^{l'}_{\e}\cap\mathcal{L}_{\e}^{l'}}\e^n\sum_{\max\{\frac{1}{\e}|j|,|\xi|\}\geq\frac{m_\e\sin\theta}{4\cos\theta}}\hspace{-1.5em}c_{\e,\alpha_\e^l(i)}^{j,\xi}\m_{\e,l}^{j,\xi}u(i)\nonumber\\
&\leq(1+2\|u\|_{L^\infty})\dsum{l',l''\in\NN(l)}{l'\neq l''}\sum_{\alpha\in\N}\sum_{\max\{\frac{1}{\e}|j|,|\xi|\}>M_\eta^\e}\hspace{-1.5em}c_{\e,\alpha}^{j,\xi}\e^{n-1}\#\big\{i\in \I_{\e}^{l'}\cap\LL_{\e}^{l'}\colon\alpha_\e^l(i)=\alpha\big\}\nonumber\\
&\leq c(u,n)\sum_{\max\{\alpha,\frac{1}{\e}|j|,|\xi|\}> M_\eta^\e}\hspace{-1.5em}c_{\e,\alpha}^{j,\xi},
\end{align}
where in the last step we have used that 
\[\e^{n-1}\#\{i\in\I_{\e}^{l'}\cap\LL_{\e}^{l'}\colon \alpha_\e^l(i)=\alpha\}\leq c\Hn(S_{l,l'})+o_\e(1).\]
Otherwise, for every $l'\in\NN(l)$, $i\in\I_{\e}^{l'}\cap\LL_{\e}^{l'}$ and $j\in\Z_\e(\Omega_i)$, $\xi\in\Z^n$ with $\max\{\tfrac{1}{\e}|j|,|\xi|\}< \tfrac{m_\e\sin\theta}{4\cos\theta}$ we have $[i+j,i+j+\e\xi]\subset U_{l}\cup U_{l'}$. We now distinguish between the case where $S_{l,l'}$ does not belong to $\overline{S_u}$ (\ie $(l,l')\neq (l_k,l_k')$ for every $k\in\{1,\ldots,K\}$) and the case where $(l,l')=(l_k,l_k')$ for some $k\in\{1,\ldots,K\}$. 

In the first case we have $u\in W^{1,\infty}(U_l\cup U_{l'};\R^d)$; hence, the inclusion $[i+j,i+j+\e\xi]\subset U_{l}\cup U_{l'}$ together with Jensen's inequality yield
\begin{align*}
\m_{\e,l}^{j,\xi}u(i) &\leq|D_\e^\xi u(i+j)|^p=\frac{1}{|\xi|^p}\left|\int_0^1\nabla u(i+j+\e t\xi)\xi\dt\right|^p\nonumber\\
&\leq\frac{1}{|\xi|^p}\int_0^1|\nabla u(i+j+\e t\xi)|^p|\xi|^p\dt\leq \|\nabla u\|_{L^\infty(U_l\cup U_{l'};\R^d)},
\end{align*}
so that
\begin{align}\label{est:ub:08}
&\sum_{i\in\I^{l'}_{\e}\cap\mathcal{L}_{\e}^{l'}}\dsum{j\in Z_\e(\Omega_i)}{|j|<\tfrac{\e m_\e\sin\theta}{4\cos\theta}}\sum_{\begin{smallmatrix}\xi\in\Z^n\\|\xi|<\tfrac{m_\e\sin\theta}{4\cos\theta}\end{smallmatrix}}c_{\e,\alpha_\e^l(i)}^{j,\xi}\e^n\m_{\e,l}^{j,\xi}u(i)\nonumber\\
&\leq\sum_{\alpha\in\N}\sum_{j\in Z_\e(\Rn)}\sum_{\xi\in\Z^n}\|\nabla u\|_{L^\infty(U_l\cup U_{l'};\R^d)}c_{\e,\alpha}^{j,\xi}\e^{n}\#\big\{i\in \I_{\e}^{l'}\cap\LL_{\e}^{l'}\colon\alpha_\e^l(i)=\alpha\big\}\nonumber\\
&\leq \e c(u)\sum_{\alpha\in\N}\sum_{j\in Z_\e(\R^n)}\sum_{\xi\in\Z^n}c_{\e,\alpha}^{j,\xi}\ \to 0\ \text{as}\ \e\to 0.
\end{align}
Suppose finally that $S_{l,l'}=S_{l_k,l_k'}$ for some $k\in\{1,\ldots,K\}$. Then we may estimate $\e^n\m_{\e,l}^{j,\xi}u(i)$ as follows,
\begin{align*}
&\e^n\m_{\e,l_k}^{j,\xi}u(i)\leq \e^{n-1}\big(1+|(M_{l_k'}x+b_{l_k'})(i+j+\e\xi)-(M_{l_k}x+b_{l_k})(i+j+\e\xi)|\big)\\
&\leq\e^{n-1}\big(1+|(M_{l_k'}x+b_{l_k'})(p_{\nu_k}(i)+\dist(i,\Pi_{\nu_k})+j+\e\xi)\\
&\hspace*{4.6em}-(M_{l_k}x+b_{l_k})(p_{\nu_k}(i)+\dist(i,\Pi_{\nu_k})+j+\e\xi)|\big)\\
&\leq \e^{n-1}\Big(1+|M_{l_k'}p_{\nu_k}(i)+b_{l_k'}-(M_{l_k}p_{\nu_k}(i)+b_{l_k})|+|M_{l_k'}-M_{l_k}|\Big(\sqrt{n}+\frac{\sin\theta}{4\cos\theta}\Big)\e m_\e\Big)\\
&\leq c\int_{p_{\nu_k}(i)+[0,\e)^{n-1}}\Big(1+|M_{l_k'}p_{\nu_k}(i)+b_{l_k'}-(M_{l_k}p_{\nu_k}(i)+b_{l_k})|+\e m_\e|M_{l_k'}-M_{l_k}|\Big)\dHn(y)\\
&\leq c\int_{p_{\nu_k}(i)+[0,\e)^{n-1}}\Big(1+|M_{l_k'}y+b_{l_k'}-(M_{l_k}y+b_{l_k})|+\e(m_\e+1)|M_{l_k'}-M_{l_k}|\Big)\dHn(y).
\end{align*}
Note that $M_{l'_k}y+b_{l_k'}=u^+(y)$, $M_{l_k}y+b_{l_k}=u^-(y)$ for $\Hn$-a.e.~$y\in S_{l_k,l_k'}$.
Hence, we obtain
\begin{align}\label{est:ub:09}
&\sum_{i\in\I^{l_k'}_{\e}\cap\mathcal{L}_{\e}^{l_k'}}\dsum{j\in Z_\e(\Omega_i)}{|j|<\tfrac{\e m_\e\sin\theta}{4\cos\theta}}\sum_{\begin{smallmatrix}\xi\in\Z^n\\|\xi|<\tfrac{m_\e\sin\theta}{4\cos\theta}\end{smallmatrix}}c_{\e,\alpha_\e^l(i)}^{j,\xi}\e^n\m_{\e,l}^{j,\xi}u(i)\nonumber\\
&\leq c\sum_{\alpha\in\N}\sum_{j\in Z_\e(\R^n)}\sum_{\xi\in\Z^n}c_{\e,\alpha}^{j,\xi}\dsum{i\in\I_{\e}^{l_k'}\cap\LL_{\e}^{l_k'}}{\alpha_{\e}^{l_k}(i)=\alpha}\int_{p_{\nu_k}(i)+[0,\e)^{n-1}}\Big(1+|u^+(y)-u^-(y)|+c(u)\e m_\e\Big)\dHn(y)\nonumber\\
&\leq \Big(c\int_{S_{l_k,l_k'}}(1+|u^+(y)-u^-(y)|)\dHn(y)+c(u)\e m_\e\Hn(S_{l_k,l_k'})\Big)\sum_{\alpha\in\N}\sum_{j\in\Z_\e(\R^n)}\sum_{\xi\in\Z^n}c_{\e,\alpha}^{j,\xi}.
\end{align}
Eventually, summing up over $l$ and gathering \eqref{est:ub:01}-\eqref{est:ub:09}, thanks to the choice of $M_\eta^\e$ and $m_\e$ we deduce that
\[\limsup_{\e\to 0}F_\e(u_\e,A)\leq c\left(\int_A\left(|\nabla u|^p+1\right)\dx+\int_{S_u\cap\clA}(1+|u^+(y)-u^-(y)|)\dHn(y)\right)+c(u)\eta,\]
hence \eqref{ub:prel} follows by the arbitrariness of $\eta>0$. 

In the general case $A\in\Areg(\Omega)$ we choose $A'$ polyhedral with $A\wcont A'\wcont\tomega$. Since $F''$ is increasing in $A$ we then obtain
\[F''(u,A)\leq F''(u,A')\leq c\left(\int_{A'}\left(|\nabla u|^p+1\right)\dx+\int_{S_u\cap\overline{A'}}(1+|u^+(y)-u^-(y)|)\dHn(y)\right),\]
and \eqref{ub:prel} follows by letting $A'\searrow A$.
\end{step}
\begin{step}{Step 2}
We now prove \eqref{upperbound} for $A\in\Areg(\Omega)$ and $u\in SBV^p(A;\R^d)\cap L^\infty(A;\R^d)$. Thanks to the Lipschitz-regularity of $A$, using a local reflection argument we can extend $u$ to a function $\tu\in SBV^p(\tomega)\cap L^\infty(\tomega)$ in such a way that $\Hn(S_{\tu}\cap\partial A)=0$. Thus Step 1 together with Proposition \ref{prop:loc} give
\begin{align*}
F''(u,A)=F''(\tu_{|A},A) &\leq c\left(\int_{A}\left(|\nabla u|^p+1\right)\dx+\int_{S_u\cap\clA}(1+|u^+(y)-u^-(y)|)\dHn(y)\right)\\
&=c\left(\int_A\left(|\nabla u|^p+1\right)\dx+\int_{S_u\cap A}(1+|u^+(y)-u^-(y)|)\dHn(y)\right).
\end{align*}
\end{step}
\begin{step}{Step 3}
We finally remove the assumption $u\in SBV^p(A;\R^d)\cap L^\infty(A;\R^d)$ by considering the truncated functions introduced in Remark \ref{rem:truncation}. More precisely, for any $u\in GSBV^p(A;\R^d)\cap L^1(\Omega;\R^d)$ and any $k\in\N$ consider the truncation $T_ku\in SBV^p(A;\R^d)\cap L^\infty(\Omega;\R^d)$. Combining Step 2 with \eqref{conv:trunc} we then obtain
\begin{align*}
F''(u,A) &=\lim_{k\to+\infty}F''(u_k,A)\\
&\leq c\limsup_{k\to +\infty}\left(\int_A\left(|\nabla T_ku|^p+1\right)\dx+\int_{S_{T_ku}\cap\clA}(1+|(T_ku)^+(y)-(T_ku)^-(y)|)\dHn(y)\right),
\end{align*}
hence \eqref{upperbound} follows by Properties (ii) and (iii) in Remark \ref{rem:truncation}.
\end{step}
\end{proof}
As a next step we establish an almost subadditivity of the functionals $F''$. As a preliminary step we prove a version of \cite[Lemma 3.6]{AC04} adapted to our setting.
\begin{lem}\label{lem:fin:diff}
There exists $c>0$ depending only on $n$ such that for any $u\in\A_\e(\Omega;\Rd)$ and any $\xi\in\Z^n$ we have
\[\dsum{i\in Z_\e(\Omega)}{i+\e\xi\in\Omega}\min\left\{|D_\e^\xi u(i)|^p,\frac{1}{\e|\xi|}\right\}\leq c\sum_{i\in Z_\e(B_R)}\min\left\{\sum_{k=1}^n|D_\e^k u(i)|^p,\frac{1}{\e}\right\},\]
where $B_R\subset\R^n$ is any open ball with $\Omega\wcont B_R$.
\end{lem}
\begin{proof}
Following the same procedure as in \cite[Lemma 3.6]{AC04} for $\xi\in\Z^n$ and $i\in Z_\e(\R^n)$ we set
\[\I_\e^\xi(i):=\{j\in Z_\e(\R^n)\colon (j+[-\e,\e]^n)\cap [i,i+\e\xi]\neq\emptyset\},\]
and for $i\in Z_\e(\Omega)$ with $i+\e\xi\in\Omega$ we choose a sequence $(i_h)_{h=0}^{|\xi|_1}\subset\I_\e^\xi(i)$ satisfying
\[i_0=i,\quad i_{|\xi|_1}=i+\e\xi,\quad i_h=i_{h-1}+\e e_{i(h)}\ \text{for some}\ i(h)\in\{1,\ldots,n\},\]
so that
\[D_\e^\xi u(i)=\frac{1}{|\xi|}\sum_{h=1}^{|\xi|_1}D_\e^{i(h)}u(i_{h-1}).\]
As in \cite[Lemma 3.6]{AC04}, applying Jensen's inequality we obtain
\[|D_\e^\xi u(i)|^p\leq\frac{n^\frac{p}{2}}{|\xi|_1}\sum_{h=1}^{|\xi|_1}|D_\e^{i(h)}u(i_{h-1})|^p,\]
hence the fact that $\min$ is non-decreasing yields
\begin{align}\label{est:fin:diff:01}
\min &\left\{|D_\e^\xi u(i)|^p,\frac{1}{\e|\xi|}\right\}\leq\min\left\{\frac{n^\frac{p}{2}}{|\xi |_1}\sum_{h=1}^{|\xi|_1}|D_\e^{i(h)}u(i_{h-1})|^p,\frac{1}{\e|\xi|}\right\}\nonumber\\
&=\frac{n^\frac{p}{2}}{|\xi|_1}\min\left\{\sum_{h=1}^{|\xi|_1}|D_\e^{i(h)}u(i_{h-1})|^p,\frac{|\xi|_1}{\e|\xi|n^\frac{p}{2}}\right\} \leq\frac{n^\frac{p}{2}}{|\xi|_1}\min\left\{\sum_{j\in\I_\e^\xi(i)}\sum_{k=1}^n|D_\e^ku(j)|^p,\frac{|\xi|_1}{\e|\xi|n^\frac{p}{2}}\right\}\nonumber\\
&\leq\frac{n^\frac{p}{2}}{|\xi|_1}\sum_{j\in\I_\e^\xi(i)}\min\left\{\sum_{k=1}^n|D_\e^ku(j)|^p,\frac{|\xi|_1}{\e|\xi|n^\frac{p}{2}}\right\},
\end{align}
where in the last step we have used the subadditivity of $\min$. Let $B_R\subset\R^n$ be any open ball with $\Omega\wcont B_R$. Note that for $\xi\in\Z^n$, $i\in Z_\e(\Omega)$ with $i+\e\xi\in\Omega$ and $\e$ sufficiently small there holds $\I_\e^\xi(i)\subset Z_\e(B_R)$. Thus, from \eqref{est:fin:diff:01} together with the fact that $\frac{|\xi|_1}{|\xi|n^\frac{p}{2}}\leq 1$ we deduce
\begin{align}\label{est:fin:diff:02}
\dsum{i\in Z_\e(\Omega)}{i+\e\xi\in\Omega}\left\{|D_\e^\xi u(i)^p,\frac{1}{\e|\xi|}\right\}\leq\frac{n^\frac{p}{2}}{|\xi|_1}\sum_{j\in Z_\e(B_R)}\#\J_\e^\xi(j)\min\left\{\sum_{k=1}^n|D_\e^ku(j)|^p,\frac{1}{\e}\right\},
\end{align}
where for any $j\in Z_\e(B_R)$ we have set
\[\J_\e^\xi(j):=\{i\in Z_\e(\Omega)\colon i+\e\xi\in\Omega,\ j\in\I_\e^\xi(i)\}.\]
In \cite[Lemma 3.6]{AC04} it has been proved that $\#\J_\e^\xi(j)\leq c(n)|\xi|$ for some $c(n)>0$ independent of $\e,j,\xi$, hence the result follows from \eqref{est:fin:diff:02} taking $c=c(n)n^\frac{p}{2}$ upon noticing that $|\xi|\leq|\xi|_1$.
\end{proof}
\begin{prop}[Subadditivity]\label{prop:subad}
Let $u\in GSBV^p(\Omega;\R^d)\cap L^1(\Omega;\Rd)$ and $A,B\in\A(\Omega)$ and suppose that $\phi_i^\e$ satisfy \ref{H1}--\ref{H5}. For every $A',B'\in\Areg(\Omega)$ with $A'\wcont A$ and $B'\wcont B$ we have
\begin{equation}\label{subadditivity}
F''(u,A'\cup B')\leq F''(u,A)+F''(u,B).
\end{equation}
\end{prop}
\begin{proof}
It suffices to prove the result for $u\in SBV^p(\Omega;\R^d)\cap L^\infty(\Omega;\R^d)$, then the general case follows by arguing as in Step 3 of Proposition \ref{prop:ub}. Moreover, we can assume that $F''(u,A)+F''(u,B)<+\infty$, otherwise the inequality trivially holds. Let $(u_\e),(v_\e)\subset\A_\e(\Omega;\Rd)$ be two sequences converging in $L^1(\Omega;\R^d)$ to $u$ with
\begin{align}
\limsup_{\e\to 0}F_\e(u_\e,A) &=F''(u,A)<+\infty,\label{est:subad:01}\\
\limsup_{\e\to 0}F_\e(v_\e,B) &=F''(u,B)<+\infty.\label{est:subad:01a}
\end{align}
Thanks to \ref{H6}, upon considering the truncated sequences $(T_Mu_\e),(T_Mv_\e)$ with $M=\|u\|_{L^\infty(\Omega;\R^d)}$ we can always assume that $\|u_\e\|_{L^\infty(\Omega;\R^d)},\|v_\e\|_{L^\infty(\Omega;\R^d)}\leq 3\|u\|_{L^\infty(\Omega;\R^d)}$ for every $\e>0$, which implies that $u_\e\to u$, $v_\e\to u$ also in $L^p(\Omega;\R^d)$.
Moreover, in view of \ref{H3} we get
\begin{align}
&\sup_{\e>0}\sum_{i\in Z_\e(A'')}\e^n\min\left\{\sum_{k=1}^n|D_\e^ku_\e(i)|^p,\frac{1}{\e}\right\}<+\infty,\label{est:subad:02}\\
&\sup_{\e>0}\sum_{i\in Z_\e(B'')}\e^n\min\left\{\sum_{k=1}^n|D_\e^kv_\e(i)|^p,\frac{1}{\e}\right\}<+\infty,\label{est:subad:03}
\end{align}
for every $A''\wcont A$, $B''\wcont B$. 
\begin{step}{Step 1}
We first replace $(u_\e)$ and $(v_\e)$ by sequences $(\tu_\e),(\tilde{v}_\e)$ satisfying \eqref{est:subad:02} and \eqref{est:subad:03} with $B_R$ in place $A''$ (respectively $B''$), where $B_R\subset\R^n$ is an open ball with $\Omega\wcont B_R$. To do so, arguing as in Proposition \ref{prop:ub} we extend $u\in SBV^p(\Omega;\R^d)\cap L^\infty(\Omega;\R^d)$ to a function $\tu\in SBV^p(B_R;\R^d)\cap L^\infty(B_R;\R^d)$ with
\begin{equation}\label{est:subad:04}
F''(\tu_{|\Omega},\Omega)\leq c\left(\int_\Omega(|\nabla u|^p+1)\dx+\int_{S_u}(1+|u^+(y)-u^-(y)|)\dHn(y)\right)<+\infty.
\end{equation}
In view of \eqref{est:subad:04} there exists a sequence $(w_\e)\subset\A_\e(\Omega;\Rd)$ converging in $L^1(\Omega;\R^d)$ to $\tu_{|\Omega}=u$ with
\[\limsup_{\e\to 0}F_\e(w_\e,\Omega)=F''(\tu_{|\Omega},\Omega)<+\infty.\]
Arguing again by truncation we can assume that $\|w_\e\|_{L^\infty(\Omega;\R^d)}\leq 3\|u\|_{L^\infty(\Omega;\R^d)}$ for every $\e>0$ and thus $w_\e\to u$ in $L^p(\Omega;\R^d)$. Moreover, appealing once more to \ref{H3}, upon extending $w_\e$ by $0$ outside of $\Omega$ we get
\begin{equation}\label{est:subad:05}
\sup_{\e>0}\sum_{i\in Z_\e(B_R)}\e^n\min\left\{\sum_{k=1}^n|D_\e^kw_\e(i)|^p,\frac{1}{\e}\right\}<+\infty.
\end{equation}
We now choose $A'',A''',B'',B'''\in\Areg(\Omega)$ with $A'\wcont A''\wcont A'''\wcont A$ and $B'\wcont B''\wcont B'''\wcont B$ and cut-off functions $\varphi_A$ between $A''$ and $A'''$  and $\varphi_B$ between $B''$ and $B'''$. Set
\begin{align*}
&\tu_\e:=\varphi_A u_\e+(1-\varphi_A)w_\e\\
&\tilde{v}_\e:=\varphi_B v_\e+(1-\varphi_B)w_\e.
\end{align*}
We still have $\tu_\e,\tilde{v}_\e\to u$ in $L^p(\Omega;\R^d)$, so that
\begin{equation}\label{est:subad:06}
\lim_{\e\to 0}\sum_{i\in Z_\e(\Omega)}\e^n|\tu_\e-\tilde{v}_\e|^p=0.
\end{equation}
Further, for every $i\in Z_\e(B_R)$ and every $k\in\{1,\ldots,n\}$ there holds
\[D_\e^k\tu_\e(i)=\varphi_A(i+\e e_k)D_\e^ku_\e(i)+(1-\varphi_A(i+\e e_k))D_\e^kw_\e(i)+D_\e^k\varphi_A(i)(v_\e^i-w_\e^i).\]
Thus, \eqref{est:subad:03} and \eqref{est:subad:05} together with the equi-boundedness of $\|v_\e\|_{L^p(\Omega;\R^d)}, \|w_\e\|_{L^p(\Omega;\R^d)}$ and the fact that $\{\varphi_A>0\}\wcont A$ yield
\begin{equation}\label{est:subad:07}
\sup_{\e>0}\sum_{i\in Z_\e(B_R)}\e^n\min\left\{\sum_{k=1}^n|D_\e^k\tu_\e(i)|^p,\frac{1}{\e}\right\}<+\infty.
\end{equation}
Analogously we also obtain
\begin{equation}\label{est:subad:08}
\sup_{\e>0}\sum_{i\in Z_\e(B_R)}\e^n\min\left\{\sum_{k=1}^n|D_\e^k\tilde{v}_\e(i)|^p,\frac{1}{\e}\right\}<+\infty.
\end{equation}
\end{step}
\begin{step}{Step 2}
For fixed $\eta>0$ we now construct a sequence $(\tilde{w}_\e)\subset\A_\e(\Omega;\Rd)$ converging to $u$ in $L^1(\Omega;\R^d)$ and satisfying
\begin{equation}\label{est:subad:08a}
\limsup_{\e\to 0} F_\e(\tilde{w}_\e,A'\cup B')\leq (1+\eta)\big(F''(u,A)+F''(u,B)\big)+c(u,A',B')\eta,
\end{equation}
then \eqref{subadditivity} follows by the arbitrariness of $\eta>0$.

Let $\eta>0$ be arbitrary and for every $\e>0$ let $M^\e_\eta>0$ be as in \eqref{H4:summability2} in \ref{H4} with
\[\limsup_{\e\to 0}\sum_{\max\left\{\alpha,\frac{1}{\e}|j|,|\xi|\right\}> M_\eta^\e}c_{\e,\alpha}^{j,\xi}<\eta.\]
Moreover, set $\dd_A:=\dist(A',\Rn\setminus A'')$, choose $L\in\N$ and for every $l\in\{1,\ldots,L\}$ set
\[A_l:=\Big\{x\in A''\colon \dist(x,A')<\frac{l\dd_A}{L}\Big\},\]
and let $A_0:=A'$. Note that up to choosing $A''$ such that $\dd_A$ is small enough the sets $A_l$ have Lipschitz-boundary for every $l\in\{1,\ldots,L\}$ and satisfy $\Hn(\partial A_l)\leq\Hn(\partial A')+1$.

For every $l\in\{1,\ldots,L-1\}$ let $\varphi_l$ be a cut-off function between $A_l$ and $A_{l+1}$, so that $\varphi_l\equiv 1$ on $A_l$, $\varphi_l\equiv 0$ on $\Omega\setminus A_{l+1}$ and $\|\nabla \varphi_l\|_{L^\infty(\Omega,\R^n)}\leq\frac{2L}{\dd_A}$.

We also set $\dd_B:=\dist(B',\Rn\setminus B'')$ and we choose $\e_0>0$ such that $\e\sqrt{n}M_\eta^\e <\min\{\dd_B,\tfrac{\dd_A}{L}\}$ for every $\e\in (0,\e_0)$. For every $l\in\{1,\ldots,L-3\}$ and $\e\in (0,\e_0)$ we then define a function $w_{\e,l}\in\A_\e(\Omega;\Rd)$ by setting
\[w_{\e,l}^i:=\varphi_l(i)\tu_\e^i+(1-\varphi_l(i))\tilde{v}_\e^i,\]
and we remark that $w_{\e,l}\to u$ in $L^1(\Omega;\R^d)$ as $\e\to 0$. Moreover,
\begin{equation}\label{est:subad:09}
F_\e(w_{\e,l},A'\cup B') = F_\e(w_{\e,l},A_{l-1})+F_\e(w_{\e,l},(A_{l+2}\setminus A_{l-1})\cap B')+F_\e(w_{\e,l},B'\setminus A_{l+2}).
\end{equation}
We estimate the three terms on the right-hand side of \eqref{est:subad:09} separately. We start with the estimate for $F_\e(w_{\e,l},A_{l-1})$. To this end, for every $i\in Z_\e(A_{l-1})$ we set
\[\alpha_\e^l(i):=\sup\{\alpha\in\N\colon i+\e\alpha Q\subset A_l\}.\]
Since $\e\sqrt{n}M_\eta^\e<\tfrac{\dd_A}{L}$, we have $\alpha_\e^l(i)>M_\eta^\e$ for every $i\in Z_\e(A_{l-1})$. Further, 
\[w_{\e,l}^{i+j}=\tu_\e^{i+j}=u_\e^{i+j}\ \text{for every}\ j\in Z_\e(\e\alpha_\e^l(i)Q),\]
and for every $\alpha\in\N$ we have
\[\e^{n-1}\#\{i\in Z_\e(A_{l-1})\colon \alpha_\e^l(i)=\alpha\}\leq c\Hn(\partial A_l)+o_\e(1)\leq c(\Hn(\partial A')+1).\]
Hence, \ref{H2} yields
\begin{align}\label{est:subad:10}
&F_\e(w_{\e,l},A_{l-1})\leq\sum_{i\in Z_\e(A_{l-1})}\e^n\phi_i^\e(\{u_\e^{i+j}\}_{j\in Z_\e(\Omega_i)})\nonumber\\
&+\sum_{i\in Z_\e(A_{l-1})}\sum_{j\in Z_\e(\Omega_i)}\dsum{\xi\in \Z^n}{j+\e\xi\in\Omega_i}c_{\e,\alpha_\e^l(i)}^{j,\xi}\min\left\{|D_\e^\xi w_{\e,l}(i+j)|^p,\frac{1+|w_{\e,l}(i+j+\e\xi)-u_\e(i+j+\e\xi)|}{\e}\right\}\nonumber\\
&\leq F_\e(u_\e,A)+(1+6\|u\|_{L^\infty})\sum_{\alpha>M_\eta^\e}\sum_{j\in Z_\e(\Rn)}\sum_{\xi\in\Z^n}c_{\e,\alpha}^{j,\xi}\e^{n-1}\#\{i\in Z_\e(A_{l-1})\colon \alpha_\e^l(i)=\alpha\}\nonumber\\
&\leq F_\e(u_\e,A)+c(1+6\|u\|_{L^\infty})(\Hn(\partial A')+1)\sum_{\alpha>M^\e_\eta}\sum_{j\in Z_\e(\Rn)}\sum_{\xi\in\Z^n}c_{\e,\alpha}^{j,\xi}.
\end{align}
Analogously, for every $i\in Z_\e(B'\setminus A_{l+2})$ we set
\[\beta_\e^l(i):=\sup\{\beta\in\N\colon i+\e\beta Q\subset B''\setminus A_{l+1}\},\]
and we observe that $\beta_\e^l(i)>M_\eta^\e$ for every $i\in Z_\e(B'\setminus A_{l+2})$ and
\[w_{\e,l}^{i+j}=\tilde{v}_\e^{i+j}=v_\e^{i+j}\ \text{for every}\ j\in Z_\e(\e\beta_\e^l(i)Q).\]
Thus, an analogous computation as in \eqref{est:subad:10} leads to
\begin{align}\label{est:subad:11}
&F_\e(w_{\e,l},B'\setminus A_{l+2})\nonumber\\
&\leq F_\e(v_\e,B)\nonumber+(1+6\|u\|_{L^\infty})\sum_{\beta>M_\eta^\e}\sum_{j\in Z_\e(\Rn)}\sum_{\xi\in\Z^n}c_{\e,\beta}^{j,\xi}\e^{n-1}\#\{i\in Z_\e(B'\setminus A_{l+2})\colon \beta_\e^l(i)=\beta\}\nonumber\\
&\leq F_\e(v_\e,B)+c(1+6\|u\|_{L^\infty})(\Hn(\partial A')+\Hn(\partial B')+1)\sum_{\beta>M_\eta^\e}\sum_{j\in Z_\e(\Rn)}\sum_{\xi\in\Z^n}c_{\e,\beta}^{j,\xi}.
\end{align}
Finally, in view of \ref{H5} we have
\begin{align}\label{est:subad:12}
F_\e(w_{\e,l},(A_{l+2}\setminus A_{l-1})\cap B')&\leq c_3\Bigg(\sum_{i\in Z_\e(S_l)}\e^n\phi_i^\e(\{\tu_\e^{i+j}\}_{j\in Z_\e(\Omega_i)})+\sum_{i\in Z_\e(S_l)}\e^n\phi_i^\e(\{\tilde{v}_\e^{i+j}\}_{j\in Z_\e(\Omega_i)})\Bigg)\nonumber\\
&+\sum_{i\in Z_\e(S_l)}\e^n R_i^\e(\tu_\e,\tilde{v}_\e,\varphi_l),
\end{align}
where $S_l:=(A_{l+2}\setminus A_{l-1})\cap B'$ and
\begin{align*}
R_i^\e(\tu_\e,\tilde{v}_\e,\varphi_l) &=\left(\frac{2L}{\dd_A}\right)^p\sum_{j\in Z_\e(\Omega)}\dsum{\xi\in\Z^n}{j+\e\xi\in\Omega}c_\e^{j-i,\xi}|\tu_\e(j+\e\xi)-\tilde{v}_\e(j+\e\xi)|^p\\
&+\sum_{j\in Z_\e(\Omega)}\dsum{\xi\in\Z^n}{j+\e\xi\in\Omega}c_\e^{j-i,\xi}\left(\min\left\{|D_\e^\xi\tu_\e(j)|^p,\frac{1}{\e|\xi|}\right\}+\min\left\{|D_\e^\xi\tilde{v}_\e(j)|^p,\frac{1}{\e|\xi|}\right\}\right).
\end{align*}
Note that the same computations as in \eqref{est:subad:10} and \eqref{est:subad:11} lead to
\begin{equation}\label{est:subad:13}
\sum_{i\in Z_\e(S_l)}\e^n\phi_i^\e(\{\tu_\e^{i+j}\}_{j\in Z_\e(\Omega_i)})\leq F_\e(u_\e,S_l)+c(u,A')\sum_{\alpha>M^\e_\eta}\sum_{j\in Z_\e(\Rn)}\sum_{\xi\in\Z^n}c_{\e,\alpha}^{j,\xi}
\end{equation}
and
\begin{equation}\label{est:subad:14}
\sum_{i\in Z_e(S_l)}\e^n\phi_i^\e(\{\tilde{v}_\e^{i+j}\}_{j\in Z_\e(\Omega_i)})\leq F_\e(v_\e,S_l)+c(u,A',B')\sum_{\alpha>M^\e_\eta}\sum_{j\in Z_\e(\Rn)}\sum_{\xi\in\Z^n}c_{\e,\alpha}^{j,\xi},
\end{equation}
respectively. Moreover, Lemma \ref{lem:fin:diff} together with \eqref{est:subad:07} and \eqref{est:subad:08} give
\begin{equation}\label{est:subad:15}
\sup_{\e>0}\sup_{\xi\in\Z^n}\dsum{j\in Z_\e(\Omega)}{j+\e\xi\in\Omega}\e^n\left(\min\left\{|D_\e^\xi\tu_\e(j)|^p,\frac{1}{\e|\xi|}\right\}+\min\left\{|D_\e^\xi\tilde{v}_\e(j)|^p,\frac{1}{\e|\xi|}\right\}\right)\leq M
\end{equation}
for some $M>0$.
For every $l$ we have $\#\{l'\neq l\colon S_l\cap S_{l'}\neq\emptyset\}\leq 5$. Thus, gathering \eqref{est:subad:09}-\eqref{est:subad:15}, summing up over $l$ and averaging we find $l(\e)\in\{1,\ldots,L-3\}$ such that
\begin{align*}
&F_\e(w_{\e,l(\e)},A'\cup B')\leq\frac{1}{L-4}\sum_{l=1}^{L-3}F_\e(w_{\e,l},A'\cup B')\\
&\leq \left(1+\frac{5c_3}{L-4}\right)(F_\e(u_\e,A)+F_\e(v_\e,B))+c(u,A',B')\sum_{\alpha>M^\e_\eta}\sum_{j\in Z_\e(\Rn)}\sum_{\xi\in\Z^n}c_{\e,\alpha}^{j,\xi}\\
&+\frac{5}{L-4}\left(\frac{2L}{\dd_A}\right)^p\sum_{i\in Z_\e(A''\cap B')}\sum_{j\in Z_\e(\Omega)}\dsum{\xi\in\Z^n}{j+\e\xi\in\Omega}\e^nc_{\e}^{j-i,\xi}|\tu_\e(j+\e\xi)-\tilde{v}_\e(j+\e\xi)|^p\\
&+\frac{5}{L-4}\sum_{i\in Z_\e(A''\cap B')}\e^n\sum_{j\in Z_\e(\Omega)}\dsum{\xi\in\Z^n}{j+\e\xi\in\Omega}c_{\e}^{j-i,\xi}\left(\min\left\{|D_\e^\xi\tu_\e(j)|^p,\frac{1}{\e|\xi|}\right\}+\min\left\{|D_\e^\xi\tilde{v}_\e(j)|^p,\frac{1}{\e|\xi|}\right\}\right)\\
&\leq \left(1+\frac{5c_3}{L-4}\right)(F_\e(u_\e,A)+F_\e(v_\e,B))+c(u,A',B')\sum_{\alpha>M_\eta^\e}\sum_{j\in Z_\e(\Rn)}\sum_{\xi\in\Z^n}c_{\e,\alpha}^{j,\xi}\\
&+\frac{5}{L-4}\left(\frac{2L}{\dd_A}\right)^p\sum_{\xi\in\Z^n}\sum_{z\in Z_\e(\R^n)}c_\e^{z,\xi}\dsum{j\in Z_\e(\Omega)}{j+\e\xi\in\Omega}|\tu_\e(j+\e\xi)-\tilde{v}_\e(j+\e\xi)|^p\\
&+\frac{5}{L-4}\sum_{\xi\in\Z^n}\sum_{z\in Z_\e(\R^n)}c_\e^{z,\xi}\dsum{j\in Z_\e(\Omega)}{j+\e\xi\in\Omega}\e^n\left(\min\left\{|D_\e^\xi\tu_\e(j)|^p,\frac{1}{\e|\xi|}\right\}+\min\left\{|D_\e^\xi\tilde{v}_\e(j)|^p,\frac{1}{\e|\xi|}\right\}\right)\\
&\leq \left(1+\frac{5c_3}{L-4}\right)(F_\e(u_\e,A)+F_\e(v_\e,B))+c(u,A',B')\sum_{\alpha>M^\e_\eta}\sum_{j\in Z_\e(\Rn)}\sum_{\xi\in\Z^n}c_{\e,\alpha}^{j,\xi}\\
&+\frac{5}{L-4}\left(\frac{2L}{\dd_A}\right)^p\left(\sum_{\xi\in\Z^n}\sum_{z\in Z_\e(\R^n)}c_\e^{z,\xi}\right)\sum_{i\in Z_\e(\Omega)}\e^n|\tu_\e^i-\tilde{v}_\e^i|^p+\frac{5M}{L-4}\sum_{\xi\in\Z^n}\sum_{z\in Z_\e(\R^n)}c_\e^{z,\xi},
\end{align*}
hence \eqref{est:subad:01},\eqref{est:subad:01a} and \eqref{est:subad:06} together with the choice of $M_\eta$ yield
\[\limsup_{\e\to 0}F_\e(w_{\e,l(\e)},A'\cup B')\leq\left(1+\frac{5c_3}{L-4}\right)(F''(u,A)+F''(u,B))+c(u,A',B')\eta+\frac{c}{L-4}.\]
It remains to choose $L\in\N$ sufficiently large such that $\frac{5c_3}{L-4}<\eta$ and $\frac{c}{L-4}<\eta$, then $\tilde{w}_\e:=w_{\e,l(\e)}$ is the required sequence satisfying \eqref{est:subad:08a}.
\end{step}
\end{proof}
\begin{rem}[Extension]\label{rem:extension}
As a last step we establish the inner regularity of $F''(u,\cdot)$ on Lipschitz sets. To this end it is convenient to extend the functionals $F_\e(\cdot,\cdot)$ to $\A_\e(\tomega;\Rd)\times\A(\tomega)\to[0,+\infty)$ for $\tomega\subset\R^n$ open bounded and with Lipschitz boundary such that $\Omega\wcont\tomega$ similar as in \cite[Proposition 3.6]{BK18}. More precisely, for every $\e>0$ and $i\in Z_\e(\tomega)$ set $\tomega_i:=\tomega-i$ and define $\tphi_i^\e:(\R^d)^{Z_\e(\tomega_i)}$ by setting
\begin{align*}
\tphi_i^\e(\{z^j\}_{j\in Z_\e(\tomega_i)}):=
\begin{cases}
\phi_i^\e(\{(z_{|\Omega})^j\}_{j\in Z_\e(\Omega_i)}) &\text{if}\ i\in Z_\e(\Omega),\\
\min\left\{\sum_{k=1}^n|D_\e^kz(0)|^p,\frac{1}{\e}\right\} &\text{if}\ i\in Z_\e(\tomega\setminus\Omega).
\end{cases}
\end{align*}
Then, for every $(u,A)\in\A_\e(\tomega;\Rd)\times\A(\tomega)$ we set
\begin{equation}\label{def:ftilde}
\tF_\e(u,A):=\sum_{i\in Z_\e(\tomega)}\e^n\tphi_i^\e(\{u^{i+j}\}_{j\in Z_\e(\tomega_i)}).
\end{equation}
Note that the functions $\tphi_i^\e$ still satisfy \ref{H1}-\ref{H5} with $\tomega$ in place of $\Omega$ and $c_1,c_2,c_3$ replaced by $\max\{c_1,\sqrt{n}\}$, $\min\{c_2,1\}$ and $\max\{c_3,3^{p-1}\}$. In particular, Propositions \ref{prop:ub} and \ref{prop:subad} hold true also with $\tomega$ and $\tF$ in place of $\Omega$ and $F$. Moreover, for every $u\in\A_\e(\Omega;\Rd)$, $\tu\in\A_\e(\tomega;\Rd)$ with $\tu^i=u^i$ for every $i\in Z_\e(\Omega)$ and $A\in\A(\Omega)$ the definition of $\tphi_i^\e$ implies that
\[\tF_\e(\tu,A)=F_\e(u,A).\]
Thus, for every $u\in GSBV^p(\Omega;\R^d)\cap L^1(\Omega;\Rd)$, $\tu\in GSBV^p(\tomega;\R^d)\cap L^1(\tomega;\Rd)$ with $\tu=u$ a.e.~in $\Omega$ and every $A\in\A(\Omega)$ we obtain
\begin{equation}\label{eq:ext}
\tF''(\tu,A)=F''(u,A).
\end{equation}
\end{rem}
\noindent The extension described above allows us to proof the following result.
\begin{prop}[Inner regularity]\label{prop:ir}
Suppose that $\phi_i^\e:(\R^d)^{Z_\e(\Omega_i)}\to[0,+\infty)$ satisfy \ref{H1}-\ref{H5}. Then for every $(u,A)\in GSBV^p(\Omega;\Rd)\cap L^1(\Omega;\Rd)\times \Areg(\Omega)$ there holds
\[F''(u,A)=F''_{-}(u,A),\]
where $F''_{-}(u,A)$ is as in \eqref{def:irenv}.
\end{prop}
\begin{proof}
Let $(u,A)\in GSBV^p(\Omega;\Rd)\cap L^1(\Omega;\Rd)\times\Areg(\Omega)$. Since $F''$ is increasing as a set function it suffices to prove $F''(u,A)\leq\sup\{F''(u,A')\colon A'\wcont A\}$. A standard way to prove this inequality consists in using the subadditivity together with the upper bound. In order to apply the same reasoning in our case we need to consider an open bounded set $\tomega\subset\Rn$ with Lipschitz boundary such that $\Omega\wcont\tomega$ and extend $F_\e$ to a functional $\tF_\e:\A_\e(\tomega;\Rd)\times\A(\tomega)\to[0,+\infty)$ as described in Remark \ref{rem:extension}. Then we apply Proposition \ref{prop:ub} and Proposition \ref{prop:subad} to $\tF$. 

Let $\tomega$ be as above; arguing as in Step 2 and Step 3 in the proof of Proposition \ref{prop:ub} we can assume that $u\in SBV^p(A;\Rd)\cap L^\infty(A;\Rd)$ and extend $u$ to a function $\tu\in SBV^p(\tomega;\Rd)\cap L^\infty(\tomega;\Rd)$ satisfying $\Hn (S_{\tu}\cap\partial A)=0$. 

Let $\eta>0$ be fixed; since $A$ has Lipschitz boundary and $\Hn (S_{\tu}\cap\partial A)=0$ we can find open bounded Lipschitz sets
\[U'\wcont U''\wcont V'\wcont V''\wcont A\wcont\tA\subset\tomega\]
such that $A\setminus\overline{U''}\in\Areg(\Omega)$, $\tA\setminus\overline{U'}\in\Areg(\tomega)$ and
\[\int_{\tA\setminus U'}(|\nabla\tu|^p+1)\dx+\int_{S_{\tu}\cap(\tA\setminus\overline{U'})}(1+|\tu^+(y)-\tu^-(y)|)\dHn(y)\leq\eta.\]
Note that $A\setminus\overline{U''}\wcont\tA\setminus\overline{U'}$. Thus, appealing to Propositions \ref{prop:ub} and \ref{prop:subad} with $\tF$ and $\tomega$ in place of $F$ and $\Omega$ we obtain
\begin{align*}
\tF''(\tu,A) &\leq\tF''(\tu,(A\setminus\overline{U''})\cup V')\leq\tF''(\tu,\tA\setminus\overline{U'})+\tF(\tu,V'')\\
&\leq c\left(\int_{\tA\setminus U'}(|\nabla\tu|^p+1)\dx+\int_{S_{\tu}\cap(\tA\setminus\overline{U'})}(1+|\tu^+(y)-\tu^-(y)|)\dHn(y)\right)+\tF(\tu,V'')\\
&\leq\sup\{\tF''(\tu,A')\colon A'\wcont A\}+c\eta.
\end{align*}
Thanks to \eqref{eq:ext} we deduce that
\[F''(u,A)\leq \sup\{F''(u,A')\colon A'\wcont A\}+c\eta\]
and we conclude by the arbitrariness of $\eta>0$.
\end{proof}
\begin{rem}\label{rem:ir}
Note that Proposition \ref{prop:ir} holds true also when $F''_{-}(u,A)$ is replaced by $$\sup\{F''(u,A')\colon A'\in\Areg(\Omega),\ A'\wcont A\}.$$
\end{rem}
\noindent On account of Propositions \ref{prop:loc}, \ref{prop:lb}, \ref{prop:ub}, \ref{prop:subad} and \ref{prop:ir} we can now prove the following compactness result.
\begin{theorem}[Compactness by $\Gamma$-convergence]\label{thm:gbar:comp}
Let $F_\e$ be as in \eqref{def:energy} and suppose that $\phi_i^\e:(\R^d)^{Z_\e(\Omega_i)}\to[0,+\infty)$ satisfy \ref{H1}-\ref{H5}. For every sequence of positive numbers converging to $0$ there exist a subsequence $(\e_j)$ and a functional $F:L^1(\Omega;\R^d)\times\A(\Omega)\to[0,+\infty)$ with
\begin{equation}\label{gbar:comp}
F(\cdot,A)=F'_{-}(\cdot,A)=F''_{-}(\cdot,A)\quad\text{on}\ GSBV^p(\Omega;\R^d)\cap L^1(\Omega;\Rd).
\end{equation}
Moreover, $F$ satisfies the following properties:
\begin{enumerate}[label={\rm(\roman*)}]
\item For every $A\in\A(\Omega)$ the functional $F(\cdot,A)$ is lower semicontinuous in the strong $L^1(\Omega;\R^d)$-topology and local;
\item there exists $c>0$ such that for every $(u,A)\in GSBV^p(\Omega;\R^d)\cap L^1(\Omega;\Rd)\times\A(\Omega)$ we have
\begin{align*}
\frac{1}{c}\left(\int_A|\nabla u|^p\dx+\Hn(S_u\cap A)\right) &\leq F(u,A)\\
&\leq c\left(\int_A(|\nabla u|^p+1)\dx+\int_{S_u\cap A}(1+|[u]|)\dHn\right);
\end{align*}
\item for every $u\in GSBV^p(\Omega;\R^d)\cap L^1(\Omega;\Rd)$ the set function $F(u,\cdot)$ is the restriction to $\A(\Omega)$ of a Radon measure;
\item for every $A\in\Areg(\Omega)$ there holds
\[F(\cdot,A)=F'(\cdot,A)=F''(\cdot,A)\quad\text{on}\ GSBV^p(\Omega;\R^d)\cap L^1(\Omega;\Rd).\]
\item $F$ is invariant under translations in $u$.
\end{enumerate}
\end{theorem}
\begin{proof}
Thanks to the general compactness theorem \cite[Theorem 16.9]{DalMaso93} we obtain a subsequence $(\e_j)$ and a functional $F$ satisfying \eqref{gbar:comp}. Moreover, Remark \ref{rem:g-li-ls} yields the $(L^1(\Omega;\R^d)$-lower semicontinuity, while Proposition \ref{prop:loc} combined with Remark \ref{rem:ir} ensures that $F(\cdot,A)$ is local for every $A\in\A(\Omega)$. Further, for every $u\in GSBV^p(\Omega;\R^d)\cap L^1(\Omega;\Rd)$ the estimates in (ii) are a consequence of the corresponding estimates for regular sets in Propositions \ref{prop:lb} and \ref{prop:ub} together with the inner regularity of the set functions $F_1(u,\cdot),F_2(u,\cdot)$ defined as $F_1(u,A):=\int_A|\nabla u|^p\dx+\Hn(S_u\cap A)$ and $F_2(u,A):=\int_A(|\nabla u|^p+1)\dx+\int_{S_u\cap A}(1+|[u]|)\dHn$. 

Since the set function $F(u,\cdot)$ is inner regular by construction, increasing and superadditive (Remark \ref{rem:g-li-ls}), in order to obtain (iii) it suffices to prove that $F(u,\cdot)$ is also subadditive, then the claim follows thanks to the De Giorgi and Letta measure criterion and the upper bound in (ii). Let $u\in GSBV^p(\Omega;\R^d)\cap L^1(\Omega;\Rd)$ and $A,B\in\A(\Omega)$ and $U\in\A(\Omega)$ with $U\wcont A\cup B$. We now show that $F''(u,U)\leq F(u,A)+F(u,B)$, then the subadditivity follows by passing to the supremum over $U$. To this end we remark that we can find $A',A'',B',B''\in\Areg(\Omega)$ with $A'\wcont A''\wcont A$ and $B'\wcont B''\wcont B$ such that $U\wcont A'\cup B'$. Thus, since $F''$ is increasing as a set function from Proposition \ref{prop:subad} we deduce
\begin{align*}
F''(u,U)\leq F''(u,A'\cup B')\leq F''(u,A'')+F''(u,B'')\leq F(u,A)+F(u,B).
\end{align*}
Finally, in view of Proposition \ref{prop:subad} we have $F''(u,A)=F(u,A)$ for every $(u,A)\in GSBV^p(\Omega;\R^d)\cap L^1(\Omega;\Rd)\times\Areg(\Omega)$, hence (iv) follows by \eqref{gbar:comp} together with the trivial inequality $F'_{-}(u,A)\leq F'(u,A)\leq F''(u,A)$. It remains to remark that (v) is a direct consequence of the fact that thanks to \ref{H6} the functionals $F_\e$ are invariant under translation in $u$.
\end{proof}
We are now in a position to prove Theorem \ref{thm:int:rep}.
\begin{proof}[Proof of Theorem \ref{thm:int:rep}]
Let $(\e_j)$ and $F$ be as in Theorem \ref{thm:gbar:comp}. Then Propositions \ref{prop:lb} and \ref{prop:ub} ensure that the domain of $F$ coincides with $GSBV^p(\Omega;\R^d)\times L^1(\Omega;\Rd)$. Moreover, in view of Theorem \ref{thm:gbar:comp} the restriction of the functional $F$ to $SBV^p(\Omega;\R^d)\times\A(\Omega)$ satisfies all hypotheses of \cite[Theorem 1]{BFLM02} except for the lower bound. In order to recover the lower bound we use a standard perturbation argument, that is, for every $\sigma>0$ we consider the functional $F_\sigma:SBV^p(\Omega;\R^d)\times\A(\Omega)\to [0,+\infty)$ defined as
\[F_\sigma(u,A):=F(u,A)+\sigma\int_{S_u\cap A}|[u]|\dHn.\]
We observe that for every $\sigma>0$ $F_\sigma$ satisfies all hypotheses of \cite[Theorem 1]{BFLM02} which thus provides us with two functions $f_0^\sigma:\Omega\times\R^d\times\R^{d\times n}\to [0,+\infty)$ and $g_0^\sigma:\Omega\times\R^d\times\R^d\times S^{n-1}\to[0,+\infty)$ such that
\[F_\sigma(u,A)=\int_A f_0^\sigma(x,u,\nabla u)\dx+\int_{S_u\cap A}g_0^\sigma(x,u^+,u^-,\nu_u)\dHn,\]
for every $u\in SBV^p(\Omega;\R^d)$ and $A\in\A(\Omega)$. Moreover, since $F$ and then also $F_\sigma$ is invariant under translation in $u$, formulas (2) and (3) in \cite[Theorem 1]{BFLM02} imply that $f_0^\sigma$ does not depend on $u$ and $g_0^\sigma$ depends on the values $u^+$ and $u^-$ only through their difference $[u]$, i.e., $f_0^\sigma(x,u,\xi)=f^\sigma(x,\xi)$ and $g_0^\sigma(x,a,b,\nu)=g^\sigma(x,a-b,\nu)$ for some functions $f^\sigma:\Omega\times\R^{d\times n}\to [0,+\infty)$, $g^\sigma:\Omega\times\R^d\times S^{n-1}\to [0,+\infty)$. Finally, formulas (2) and (3) in \citep[Theorem 1]{BFLM02} also imply that $f^\sigma$ and $g^\sigma$ decrease as $\sigma$ decreases. Hence, setting $f(x,\xi):=\lim_{\sigma\to 0^+}f^\sigma(x,\xi)$, $g(x,t,\nu):=\lim_{\sigma\to 0^+}g^\sigma(x,t,\nu)$, from the pointwise convergence of $F_\sigma$ to $F$ and the Monotone Convergence Theorem we deduce
\[F(u,A)=\int_A f(x,\nabla u)\dx+\int_{S_u\cap A}g(x,[u],\nu_u)\dHn,\]
for every $u\in SBV^p(\Omega;\R^d)$ and $A\in\A(\Omega)$. In particular, thanks to Theorem \ref{thm:gbar:comp} (iv) we deduce that \eqref{localrepresentation} holds for every $u\in SBV^p(\Omega;\Rd)$ and $A\in\Areg(\Omega)$, and choosing $A=\Omega$ in the formula above we obtain the desired integral representation on $SBV^p(\Omega;\R^d)$. We finally observe that formulas (2) and (3) in \citep[Theorem]{BFLM02} imply that the integrands $f$ and $g$ are given by \eqref{derivationformula}.

Eventually, we show that the integral representation also extends to $GSBV^p(\Omega;\R^d)\cap L^1(\Omega;\Rd)$. To this end, for every $u\in GSBV^p(\Omega;\R^d)\cap L^1(\Omega;\Rd)$ and every $k\in\N$ we consider again the truncation $T_ku$ as in Remark \ref{rem:truncation}. Using (ii) and (iii) in Remark \ref{rem:truncation} together with \eqref{conv:trunc} and appealing to the Monotone Convergence Theorem we get
\begin{align*}
\Gamma\hbox{-}\lim_{j\to +\infty}F_{\e_j}(u) &=\lim_{k\to +\infty}F(T_ku)=\lim_{k\to +\infty}\left(\int_\Omega f(x,\nabla T_ku)\dx+\int_{S_{T_ku}}g(x,[T_ku],\nu_{T_ku})\dHn\right)\\
&=\int_{\Omega}f(x,\nabla u)\dx+\int_{S_u}g(x,[u],\nu_u)\dHn.
\end{align*}
\end{proof}
\subsection{Treatment of Dirichlet problems}
For further use in Section \ref{sect:homogenization}, we study here the asymptotic behavior of minimum problems for $F_{\e}$ when suitable Dirichlet boundary conditions are taken into account. More precisely, for every $\delta>0$, every $A\in\Areg(\Omega)$ and every pointwise well-defined function $\bar{u}\in L^1(\Omega;\R^d)$ we consider the minimization problem
\[\m_{\e}^\delta(\bar{u},A):=\inf\{F_{\e}(u,A)\colon u\in\A_{\e}^\delta(\bar{u},A)\},\]
where
\[\A_{\e}^\delta(\bar{u},A):=\{u\in\A_{\e}(\Omega;\Rd)\colon  u(i)=\bar{u}(i)\  \text{if}\ \dist(i,\R^n\setminus A)<\delta\},\]
and we study the asymptotic behavior of $\m_{\e}^\delta(\bar{u},A)$ when first $\e\to 0$ and then $\delta\to 0$. For our purpose it is sufficient to consider boundary data $\bar{u}\in SBV^p(\Omega;\Rd)\cap L^\infty(\Omega;\Rd)$ satisfying $\Hn(S_{\bar{u}}\cap\partial A)=0$ and such that
\begin{equation}\label{cond:ubar}
\bar{u}_\e\to \bar{u}\quad\text{in}\ L^1(\Omega;\Rd),\quad\limsup_{\e\to 0}F_\e(\bar{u}_\e,B)\leq c\left(\int_B|\nabla\bar{u}|^p\dx+\Hn(S_{\bar{u}}\cap\overline{B})\right)
\end{equation}
where $\bar{u}_\e\in\A_\e(\Omega;\Rd)$ is defined by setting $\bar{u}_\e^i:=\bar{u}(i)$ and $B\in\Areg(\Omega)$. For $\bar{u}$ as above we can prove the following convergence result.
\begin{lem}\label{lem:dirichlet}
Let $\phi_i^\e:(\R^d)^{Z_\e(\Omega_i)}\to[0,+\infty)$ satisfy hypotheses \ref{H1}--\ref{H5} and let $F_{\e_j}$ be the subsequence provided by Theorem \ref{thm:int:rep}. Moreover, let $A\in\Areg(\Omega)$ with $A\wcont\Omega$. For every pointwise well-defined function $\bar{u}\in SBV^p(\Omega;\Rd)\cap L^\infty(\Omega;\Rd)$ with $\Hn(S_{\bar{u}}\cap\partial A)=0$ and satisfying \eqref{cond:ubar} we have
\[\lim_{\delta\to 0}\liminf_{j\to+\infty}\m_{\e_j}^\delta(\bar{u},A)=\lim_{\delta\to 0}\limsup_{j\to +\infty}\m_{\e_j}^\delta(\bar{u},A)=\m(\bar{u},A),\]
where $\m(\bar{u},A)$ is as in \eqref{def:min:prob}.
\end{lem}
\begin{rem}\label{rem:dirichlet}
Lemma \ref{lem:dirichlet} together with \eqref{derivationformula} provide us with asymptotic formulas for the integrands $f$ and $g$ given by Theorem \ref{int:form}. Indeed, for $x_0\in\Omega$, $\nu\in S^{n-1}$ and $\rho>0$ sufficiently small we have $Q_\rho^\nu(x_0)\wcont\Omega$. Moreover, for every $\zeta\in\Rd$ and $M\in\R^{d\times n}$ the functions $u_{M,x_0},u_{\zeta,x_0}^\nu$ as in \eqref{def:boundarydata} satisfy the hypotheses of Lemma \ref{lem:dirichlet}. Thus, passing to the upper limit as $\rho\to 0$ we obtain the following formulas for $f$ and $g$
\begin{align*}
f(x_0,M) &=\limsup_{\rho\to 0}\frac{1}{\rho^n}\lim_{\delta\to 0}\liminf_{j\to+\infty}\m_{\e_j}^\delta(u_{M,x_0})=\limsup_{\rho\to 0}\frac{1}{\rho^n}\lim_{\delta\to 0}\limsup_{j\to+\infty}\m_{\e_j}^\delta(u_{M,x_0}),\\
g(x_0,\zeta,\nu) &=\limsup_{\rho\to 0}\frac{1}{\rho^{n-1}}\lim_{\delta\to 0}\liminf_{j\to+\infty}\m_{\e_j}^\delta(u_{\zeta,x_0}^\nu)=\limsup_{\rho\to 0}\frac{1}{\rho^{n-1}}\lim_{\delta\to 0}\limsup_{j\to+\infty}\m_{\e_j}^\delta(u_{\zeta,x_0}^\nu).
\end{align*}
\end{rem}
\begin{proof}[Proof of Lemma \ref{lem:dirichlet}]
Let $A,\bar{u}$ be as in the statement. Observe that due to monotonicity the limit as $\delta\to 0$ exists.
We show that $\mathbf{m}(\bar{u},A)$ is both an asymptotic lower and an asymptotic upper bound for $\mathbf{m}_{\e_j}^\delta(\bar{u},A)$.
\begin{step}{Step 1} We first establish the inequality
\begin{align}\label{lowerbound:dirichlet}
\m(\bar{u},A)\leq \lim_{\delta\to 0}\liminf_{j\to+\infty}\m_{\e_j}^\delta(\bar{u},A).
\end{align}
To this end, let $\delta>0$ be fixed and let $u_j\in\A_{\e_j}(\Omega;\Rd)$ be admissible for $\m_{\e_j}^\delta(\bar{u},A)$ with
\[F_{\e_j}(u_j,A)=\m_{\e_j}^\delta(\bar{u},A).\]
Thanks to Remark \ref{rem:truncation} we can assume that $\|u_j\|_{L^\infty}\leq 3\|\bar{u}\|_{L^\infty}$. In particular, the sequence $(u_j)$ is equi-integrable, hence \ref{H2} together with \cite[Lemma 5.6]{Ruf17} yield the existence of a subsequence (not relabeled) converging in $L^1(\Omega;\Rd)$ to some $u\in GSBV^p(A;\Rd)\cap L^1(A;\Rd)$. Since $u_j=\bar{u}_{\e_j}$ on $\partial A+B_\delta(0)$, \eqref{cond:ubar} ensures that $u=\bar{u}$ on $\partial A+B_\delta(0)$, hence $u$ is admissible for $\m(\bar{u},A)$. Thus, Theorem \ref{thm:int:rep} yields
\[\m(\bar{u},A)\leq F(u,A)\leq\liminf_{j\to+\infty}F_{\e_j}(u_j,A)=\liminf_{j\to+\infty}\m^\delta_{\e_j}(\bar{u},A)\]
hence \eqref{lowerbound:dirichlet} follows by letting $\delta\to 0$.
\end{step}
\begin{step}{Step 2}
We now prove that
\[\lim_{\delta\to 0}\limsup_{j\to+\infty}\m_{\e_j}^\delta(\bar{u},A)\leq\m(\bar{u},A).\]
To this end, for fixed $\eta>0$ we choose $u\in SBV^p(A;\Rd)$ with $u=\bar{u}$ in a neighborhood of $\partial A$ and $F(u,A)\leq\m(\bar{u},A)+\eta$. Thanks to Proposition \ref{prop:loc} we can extend $u$ to $\Omega\setminus A$ by $\bar{u}$ without changing $F(u,A)$. Moreover, Theorem \ref{thm:int:rep} provides us with a sequence of functions $u_j\in A_{\e_j}(\Omega;\Rd)$ converging to $u$ in $L^1(\Omega;\Rd)$ and satisfying
\begin{equation}\label{dirich:recovery}
\limsup_{j\to+\infty}F_{\e_j}(u_j,A)=F(u,A).
\end{equation}
We now modify $u_j$ to fulfill the required discrete boundary condition. Since $u=\bar{u}$ in a neighborhood of $\partial A$, we can find $A'\in\Areg(\Omega)$, $A'\wcont A$ such that $u=\bar{u}$ on $A\setminus\overline{A'}$ (and by extension $u=\bar{u}$ on $\Omega\setminus\overline{A'}$). Since moreover $\Hn(S_{\bar{u}}\cap\partial A)=0$ we can choose further sets $A'',\tilde{A}\in\Areg(\Omega)$ with $A'\wcont A''\wcont A\wcont \tilde{A}$ and 
\[\int_{\tilde{A}\setminus A'}|\nabla\bar{u}|^p\dx+\Hn(S_{\bar{u}}\cap\overline{\tilde{A}\setminus A'})\leq\eta.\]
Arguing as in the proof of Proposition \ref{prop:subad} we can construct a sequence $(w_j)$ with $w_j=u_j$ on $A''$, $w_j=\bar{u}$ on $\Omega\setminus\overline{A}$ and 
\begin{align}\label{est:dirich:01}
\limsup_{j\to +\infty} F_{\e_j}(w_j,A) &=\limsup_{j\to+\infty}F_{\e_j}(w_j,A''\cup A\setminus\overline{A''})\nonumber\\
&\leq(1+\eta)\left(\limsup_{j\to+\infty}F_{\e_j}(u_j,A)+\limsup_{j\to+\infty}F_{\e_j}(\bar{u}_j,\tilde{A}\setminus \overline{A'})\right)+\eta,
\end{align}
where $\bar{u}_j$ is as in \eqref{cond:ubar} with $\e_j$ in place of $\e$. In view of \eqref{cond:ubar} and the choice of $A',\tilde{A}$ we have
\[\limsup_{j\to+\infty}F_{\e_j}(\bar{u}_j,\tilde{A}\setminus \overline{A'})\leq c\left(\int_{\tilde{A}\setminus A'}|\nabla\bar{u}|^p\dx+\Hn(S_{\bar{u}}\cap\overline{\tilde{A}\setminus A'})\right)\leq c\eta.\]
Moreover, for $\delta$ sufficiently small $w_j$ is admissible for $\m_{\e_j}^\delta(\bar{u},A)$. Thus, gathering \eqref{dirich:recovery}--\eqref{est:dirich:01} thanks to the choice of $u$ we deduce that
\[\lim_{\delta\to 0}\limsup_{j\to+\infty}\m_{\e_j}^\delta(\bar{u},A)\leq\limsup_{j\to+\infty} F_{\e_j}(w_j,A)\leq (1+\eta)\m(\bar{u},A)+c\eta\]
and we conclude by the arbitrariness of $\eta>0$.
\end{step}
\end{proof}
\section{Homogenization}\label{sect:homogenization}
In this section we consider a special class of periodic interaction-energy densities $\phi_i^\e$ for which we can show that the $\Gamma$-limit provided by Theorem \ref{thm:int:rep} does not depend on the $\Gamma$-converging subsequence, which in turn implies that the whole sequence $(F_\e)$ $\Gamma$-converges. We first need to specify what periodicity means in the case of interaction-energy densities $\phi_i^\e:(\Rd)^{Z_\e(\Omega_i)}\to[0,+\infty)$ that may depend on the whole state $\{z^j\}_{j\in Z_\e(\Omega_i)}$. This difficulty is also present in \citep[Section 5]{BK18}. To avoid the dependence of $\phi_i^\e$ on $\Omega_i$ in \citep{BK18} the authors use a sequence of periodic finite-range interactions $\phi_i^k$ defined on the entire lattice $(\Rd)^{\Z^n}$ whose range increases as $k$ increases and which converge for every $i\in\Z^n$ to a long-range interaction-energy density $\phi_i:(\Rd)^{\Z^n}\to[0,+\infty)$ as $k\to +\infty$. For $i\in Z_\e(\Omega)$ the functions $\phi_i^\e$ are then obtained by a rescaling of a suitably chosen $\phi_{\frac{i}{\e}}^{k(\e)}$, where $\e k(\e)$ is proportional to the distance of $i$ to the boundary of $\Omega$. Since the energy densities $\phi_i^\e$ that we consider here contain both a bulk and a surface scaling the approach in \citep{BK18} cannot be adapted to our setting. Instead, here we consider functions $\psi_i^\e:(\R^d)^{Z_\e(\Rn)}\to [0,+\infty)$ defined on the entire scaled lattice $\e\Z^n$ which have only finite range. This finite-range assumption will be crucial to decouple the bulk and the surface scaling in the $\Gamma$-limit. 

We now state our precise hypotheses. Let $K\in\N$, $L\in\N$ and consider functions $\psi_i^\e:(\R^d)^{Z_\e(\R^n)}\to[0,+\infty)$ which are $\e K$-periodic in $i$ and satisfy hypotheses \ref{H1}--\ref{H5} with $Z_\e(\Omega_i)$ replaced by $Z_\e(\R^n)$; where in addition the sequences $(c_{\e,\alpha}^{j,\xi})$ and $(c_{\e}^{j,\xi})$ provided by \ref{H4} and \ref{H5}, respectively, satisfy
\begin{align}\label{cond:finite-range0}
\begin{split}
c_{\e,\alpha}^{j,\xi} &=0\qquad\text{if}\ \max\{\alpha,2|\tfrac{j}{\e}|_\infty,2|\xi|_\infty,2|\tfrac{j}{\e}+\xi|_\infty\}\geq L,\\
c_{\e}^{j,\xi} &=0\qquad\text{if}\ \max\{2|\tfrac{j}{\e}|_\infty,2|\xi|_\infty,2|\tfrac{j}{\e}+\xi|_\infty\}\geq L.
\end{split}
\end{align}
In particular, whenever $z,w:Z_\e(\Rn)\to\R^d$ are such that $z^j=w^j$ for all $j\in Z_\e(\e LQ)$, we have
\begin{align}\label{cond:finite-range1}
\psi_i^\e(\{z^j\}_{j\in Z_\e(\e LQ)})=\psi_i^\e(\{w^j\}_{j\in Z_\e(\e LQ)}).
\end{align}
We also set
\[\Omega_\e^L:=\{x\in\Omega\colon\dist_\infty(x,\partial\Omega)>L\e\}\]
and we define $\phi_i^\e: (\R^d)^{Z_\e(\Omega_i)}\to[0,+\infty)$ by setting
\begin{align}\label{def:phi:periodic}
\phi_i^\e(\{z^j\}_{j\in Z_\e(\Omega_i)}):=
\begin{cases}
\psi_i^\e(\{z^j\chi^j_{\e LQ}\}_{j\in Z_\e(\R^n)}) &\text{if}\ i\in Z_\e(\Omega_\e^L),\\
\cr
\displaystyle\min\Big\{\hspace*{-0.5em}\dsum{k=1}{\e e_k\in\Omega_i}^n\hspace*{-0.5em}|D_\e^kz(0)|^p,\frac{1}{\e}\Big\} &\text{if}\ i\in Z_\e(\Omega\setminus \Omega_\e^L),
\end{cases}
\end{align}
which is well-defined thanks to \eqref{cond:finite-range1}. By construction, $\phi_i^\e:(\R^d)^{Z_\e(\Omega_i)}\to [0,+\infty)$ satisfy hypotheses \ref{H1}--\ref{H5}. We now aim to prove that for $\phi_i^\e:(\R^d)^{Z_\e(\Omega_i)}\to[0,+\infty)$ defined as in \eqref{def:phi:periodic} the integrands $f$ and $g$ provided by Theorem \ref{thm:int:rep} are independent of the position $x$.
\begin{prop}\label{cor:trans-inv}
Let $F_\e$ be as in \eqref{def:energy} with $\phi_i^\e:(\R^d)^{Z_\e(\Omega_i)}\to[0,+\infty)$ given by \eqref{def:phi:periodic}, where $\psi_i^\e:(\Rd)^{Z_\e(\Rn)}\to[0,+\infty)$ are $\e K$-periodic in $i$, satisfy \ref{H1}--\ref{H5} with $Z_\e(\Omega_i)$ replaced by $Z_\e(\Rn)$, and \eqref{cond:finite-range0}. Let $(\e_j)$ and $F$ be the subsequence and the functional provided by Theorem \ref{thm:int:rep}. Then $F$ is of the form
\begin{align}\label{hom:integrands}
F(u)=\int_\Omega\bar{f}(\nabla u)\dx+\int_{S_u}\bar{g}([u],\nu_u)\dHn,\qquad u\in GSBV^2(\Omega;\Rd),
\end{align}
for some functions $\bar{f}:\R^{d\times n}\to[0,+\infty)$ and $\bar{g}:\R^d\times S^{n-1}\to[0,+\infty)$ possibly depending on the $\Gamma$-converging subsequence. Moreover, for every $A\in\Areg(\Omega)$ and $u\in GSBV^2(\Omega;\Rd)$ there holds
\begin{align*}
\Gamma\hbox{-}\lim_{j\to +\infty}F_{\e_j}(u,A)=\int_A\bar{f}(\nabla u)\dx+\int_{S_u\cap A}\bar{g}([u],\nu_u)\dHn.
\end{align*}
\end{prop}
\noindent We prove \ref{cor:trans-inv} by adapting a well-known argument (see, \eg \citep[Lemma 2.7]{BDV96}) to our setting showing that the minimization problem $\m(\bar{u},A)$ defined in \eqref{def:min:prob} is invariant under translation for a suitable class of functions $\bar{u}$. We start by introducing some notation. For every $A\in\A(\Omega)$ and $y\in\R^n$ we set $\tau_y A:=A+y$. Moreover, for every $u:\Omega\to\R^d$ and every $A\in\A(\Omega)$ with $\tau_y A\subset\Omega$ we define $\tau_y u:\tau_y A\to\R^d$ by setting $\tau_y u(x):=u(x-y)$ for every $x\in\tau_y A$. For our purpose it is sufficient to consider pointwise well-defined functions $\bar{u}\in SBV^p_{\rm loc}(\Rn;\Rd)$ which satisfy
\begin{align}\label{cond:translation}
\tau_y\bar{u}_\e^i\to\tau_y\bar{u}\quad\text{in}\ L^1(\Omega;\Rd)\quad\text{for every}\ y\in\Rn,
\end{align}
where for every $y\in\Rn$ the function $\tau_y\bar{u}_\e\in\A_\e(\Omega;\Rd)$ is defined by setting $\tau_y\bar{u}_\e^i:=\tau_y\bar{u}(i)$ for every $i\in Z_\e(\Rn)$. We now prove the following auxiliary lemma.
\begin{lem}\label{lem:trans-inv}
Suppose that $\phi_i^\e:(\R^d)^{Z_\e(\Omega_i)}\to[0,+\infty)$ are given by \eqref{def:phi:periodic},  where $\psi_i^\e:(\Rd)^{Z_\e(\Rn)}\to[0,+\infty)$ are $\e K$-periodic in $i$, satisfy \ref{H1}--\ref{H5} with $Z_\e(\Omega_i)$ replaced by $Z_\e(\Rn)$, and \eqref{cond:finite-range0}. Let $A\in\Areg(\Omega)$ with $A\wcont\Omega$ and let $\bar{u}\in SBV^p_{\rm loc}(\Rn;\R^d)$ be a pointwise well-defined function satisfying \eqref{cond:translation}. For any $y\in\R^n$ with $\tau_y A\wcont\Omega$ there holds
\[\m(\bar{u},A)=\m(\tau_y\bar{u},\tau_y A),\]
where $\m(\bar{u},A), \m(\tau_y\bar{u},\tau_y A)$ are defined according to \eqref{def:min:prob}.
\end{lem}
\begin{proof}
Let $A,\bar{u}$ and $y$ be as in the statement and let us prove that
\begin{equation}\label{est:inv:01}
\m(\tau_y\bar{u},\tau_y A)\leq\m(\bar{u},A).
\end{equation}
To this end let $u\in SBV^p(A;\R^d)$ be admissible for $\m(\bar{u},A)$ and $A'\wcont A$ with $u=\bar{u}$ in $A\setminus\overline{A'}$. In view of Proposition \ref{prop:loc} we can extend $u$ to $\Omega\setminus A$ by $\bar{u}$ without changing $F(u,A)$. In order to simplify notation we still denote the subsequence provided by Theorem \ref{thm:int:rep} by $\e$ and we choose a sequence $(u_\e)\subset\A_\e(\Omega;\R^d)$ converging to $u$ in $L^1(\Omega;\R^d)$ and satisfying
\[\lim_{\e\to 0}F_\e(u_\e,A)=F(u,A).\]
We now construct a suitable sequence $(v_\e)$ converging to $\tau_y u$ in $L^1(\Omega;\R^d)$. We choose $A'',A'''\in\Areg(\Omega)$ with $A'\wcont A''\wcont A'''\wcont A$ and $\e_0$ sufficiently small such that for all $\e\in (0,\e_0)$ the following conditions are satisfied.
\begin{enumerate}[label=(\roman*)]
\setlength{\itemsep}{3pt}
\item $A\cup\tau_y A\subset\Omega_\e^L$;
\item $\tau_y A''\subset\tau_{y_\e}A'''$ and $\tau_y A'''\subset\tau_{y_\e}A$, where $y_\e:=\e K\lfloor\frac{y}{\e K}\rfloor$;
\item $\e L<\dist_\infty(A'',\partial A''')$.
\end{enumerate}
For $\e\in (0,\e_0)$ we then define $v_\e\in\A_\e(\Omega;\R^d)$ by setting
\begin{align*}
v_\e^i:=
\begin{cases}
u_\e^{i-y_\e} &\text{if}\ i\in Z_\e(\tau_y A'''),\\
\tau_y\bar{u}(i) &\text{if}\ i\in Z_\e(\Omega\setminus\tau_y A'''),
\end{cases}
\end{align*}
which is well-defined thanks to the second inclusion in (ii).

Since $u=\bar{u}$ in $\Omega\setminus\overline{A'}$, thanks to \eqref{cond:translation} we have that $v_\e\to\tau_y u$ in $L^1(\Omega;\R^d)$. Moreover, for all $i\in Z_\e(\tau_y A'')$ and $j\in Z_\e(\e LQ)$ assumption (iii) yields $i+j\in\tau_y A'''$, and hence
\[v_\e^{i+j}=u_\e^{i-y_\e+j}.\]
Thanks to the locality property \eqref{cond:finite-range1} and the periodicity assumption we thus obtain
\begin{align*}
F_\e(v_\e,\tau_y A'') &=\hspace*{-1em}\sum_{i\in Z_\e(\tau_y A'')}\hspace*{-1em}\e^{n}\psi_i^\e(\{u_\e^{i-y_\e+j}\chi_{\e LQ}^j\}_{j\in Z_\e(\R^n)})
\leq\hspace*{-0.7em}\sum_{i\in Z_\e(A''')}\hspace*{-0.7em}\e^{n}\psi_i^\e(\{u_\e^{i+j}\chi_{\e LQ}^j\}_{j\in Z_\e(\R^n)})\\
&\leq\hspace*{-0.4em}\sum_{i\in Z_\e(A)}\hspace*{-0.4em}\e^n\phi_i^\e(\{u_\e^{i+j}\}_{j\in\Z^n}) = F_\e(u_\e,A),
\end{align*}
where in the second inequality we have used the first inclusion in (ii). Together with the fact that $F(\cdot,\tau_y A'')=\Gamma\hbox{-}\lim_\e F_\e(\cdot,\tau_y A'')$ and $v_\e\to \tau_y u$ in $L^1(\Omega;\R^d)$ the above inequality allows us to deduce that
\[F(\tau_y u,\tau_y A'')\leq\liminf_{\e\to 0}F_\e(v_\e,\tau_y A'')\leq\lim_\e F_\e(u_\e,A)=F(u,A).\]
In view of Proposition \ref{prop:ir}, Remark \ref{rem:ir} and the arbitrariness of $A''\wcont A$ we finally get
\begin{equation}\label{est:inv:02}
F(\tau_yu,\tau_y A)\leq F(u,A).
\end{equation}
Hence, since $\tau_y u$ is admissible for $\m(\tau_y\bar{u},A)$ and $u$ was arbitrarily chosen we obtain \eqref{est:inv:01} by passing to the infimum on both sides of \eqref{est:inv:02}. To deduce the result it then suffices to remark that the opposite inequality follows by applying \eqref{est:inv:01} with $\tau_{-y}$.
\end{proof}
On account of Lemma \ref{lem:trans-inv} we now prove Proposition \ref{cor:trans-inv}.
\begin{proof}[Proof of Proposition \ref{cor:trans-inv}]
Let $F$ be as in Theorem \ref{thm:int:rep}. We claim that the integrands $f$ and $g$ as in \eqref{derivationformula} are independent of the position $x_0$, then $F$ can be written in the form \eqref{hom:integrands}.
To prove the claim we fix $x_0,y_0\in\Omega$ and choose $\rho>0$ sufficiently small such that $Q_\rho^\nu(x_0)\cup Q_\rho^\nu(y_0)\wcont\Omega$.
For every $M\in\R^{d\times n}$ and every $(\zeta,\nu)\in\R^d\times S^{n-1}$ the functions $u_{M,x_0}$ and $u_{\zeta,x_0}^\nu$ defined as in \eqref{def:boundarydata} satisfy the hypotheses of Lemma \ref{lem:trans-inv}. Thus, we obtain 
\[\m(u_{\zeta,y_0}^\nu,Q_\rho^\nu(y_0))=\m(\tau_{y_0-x_0}u_{\zeta,x_0}^\nu,\tau_{y_0-x_0}Q_\rho^\nu(x_0))=\m(u_{\zeta,x_0}^\nu,Q_\rho^\nu(x_0))\]
and
\[\m(u_{M,y_0},Q_\rho^\nu(y_0))=\m(\tau_{y_0-x_0}u_{M,x_0},\tau_{y_0-x_0}Q_\rho^\nu(x_0))=\m(u_{M,x_0},Q_\rho^\nu(x_0)).\]
We conclude by letting $\rho\to 0$.
\end{proof}
\subsection{Separation of bulk and surface effects}\label{sect:separation} In this subsection we give sufficient conditions on the functions $\psi_i^\e$ under which a separation of energy contributions takes place in the limit. We state the precise hypotheses after introducing some notation. For every $\e>0$, every $u:Z_\e(\Rn)\to\R^d$ and every $i\in Z_\e(\Rn)$ set
\[
|\nabla_\e u|(i):=\sum_{k=1}^n\big(|D_\e^{e_k} u(i)|+|D_\e^{-e_k} u(i)|\big),\qquad |\nabla_{\e,L}u|(i):=\sum_{\xi\in Z_1(LQ)}\left|\frac{u^i-u^{i+\e\xi}}{\e}\right|.
\]
We then assume that for every $i\in\Z^n$ there exist $\psi_i^b,\psi_i^s:(\R^d)^{\Z^n}\to[0,+\infty)$ such that the following holds. 
\begin{enumerate}[label={\rm (${\rm H}_\psi$\arabic*)}]
\item\label{Hpsi1} For every $\eta>0$ and every $\Lambda>0$ there exists $\bar{\e}=\bar{\e}(\eta,\Lambda)>0$ such that for every $\e\in (0,\bar{\e})$, for every $i\in\Z^n$ and for every $z:\Z^n \to\R^d$ with $|\nabla_{1,L} z|(0)<\Lambda$ we have
\begin{equation*}
|\psi_{\e i}^\e(\{\e z^{\tfrac{j}{\e}}\}_{j\in Z_\e(\R^n)})-\psi_i^b(\{z^j\}_{j\in\Z^n})|<\eta.
\end{equation*}
\item\label{Hpsi2} For every $\eta>0$ there exist $\Lambda(\eta)>0$ and $\hat{\e}=\hat{\e}(\eta)>0$ such that for every $\e\in(0,\hat{\e})$, for every $i\in\Z^n$ and every $z:\Z^n\to\R^d$ with $\e^{\frac{1-p}{p}}|\nabla_{1,L} z|(0)\geq\Lambda(\eta)$ or $|\nabla_{1,L}z|(0)=0$ we have
\begin{equation*}
|\e\psi_{\e i}^\e(\{z^{\tfrac{j}{\e}}\}_{j\in Z_\e(\R^n)})-\psi_i^s(\{z^j\}_{j\in\Z^n})|<\eta.
\end{equation*}
\end{enumerate}
Moreover, we assume that the functions $\psi_i^s$ satisfy the following continuity hypotheses.
\begin{enumerate}[label={\rm (${\rm H}_\psi$\arabic*)}]
\setcounter{enumi}{2}
\item\label{Hpsi3} There exists a constant $c_s>0$ such that for every $z,w:\Z^n\to\R^d$ with $|\nabla_{1,L}z|(0)>0$, $|\nabla_{1,L}w|(0)>0$ and for every $i\in\Z^n$ there holds
\begin{equation*}
|\psi_i^s(\{z^j\}_{j\in\Z^n})-\psi_i^s(\{w^j\}_{j\in\Z^n})|\leq c_s\sum_{j\in Z_1(Q_L(i))}\dsum{\xi\in Z_1(Q_L(i))}{j+\xi\in Q_L(i)}|z^{j+\xi}-w^{j+\xi}|.
\end{equation*}
\end{enumerate}
The main result of this section is the following theorem which states that under the additional assumptions \ref{Hpsi1}--\ref{Hpsi3} the bulk and surface interactions decouple in the $\Gamma$-limit. As a consequence we obtain asymptotic minimization formulas for the bulk and the surface energy density that are independent of the $\Gamma$-converging subsequence.
\begin{theorem}[Homogenization]\label{thm:homogenization}
Assume that $\phi_i^\e:(\R^d)^{Z_\e(\Omega_i)}\to[0,+\infty)$ are given by \eqref{def:phi:periodic},  where $\psi_i^\e:(\Rd)^{Z_\e(\Rn)}\to[0,+\infty)$ are $\e K$-periodic in $i$, satisfy \ref{H1}--\ref{H5} with $Z_\e(\Omega_i)$ replaced by $Z_\e(\Rn)$, and \eqref{cond:finite-range0}, and suppose that in addition \ref{Hpsi1}--\ref{Hpsi3} are satisfied. Then the functionals $F_\e:L^1(\Omega;\R^d)\to[0,+\infty]$ defined as in \eqref{def:energy} $\Gamma$-converge in the strong $L^1(\Omega;\R^d)$-topology to the functional $F_{\rm hom}:L^1(\Omega;\R^d)\to[0,+\infty]$ given by
\begin{align*}
F_{\rm hom}(u)=
\begin{cases}
\displaystyle\int_\Omega f_{\rm hom}(\nabla u)\dx+\int_{S_u} g_{\rm hom}([u],\nu_u)\dHn &\text{if}\ u\in GSBV^p(\Omega;\R^d),\\
+\infty &\text{otherwise in}\ L^1(\Omega;\R^d),
\end{cases}
\end{align*}
where $f_{\rm hom}:\R^{d\times n}\to[0,+\infty)$ and $g_{\rm hom}:\R^d\times S^{n-1}\to[0,+\infty)$ are given by
\begin{align}\label{fhom}
f_{\rm hom}(M)=\lim_{T\to+\infty}\frac{1}{T^n}\inf\Big\{\sum_{i\in Z_1(TQ)}\psi_i^b(\{u^{i+j}\}_{j\in\Z^n})\colon u\in\A_1^{\sqrt{n}L}(u_{M},TQ)\Big\}
\end{align}
and
\begin{align}\label{ghom}
g_{\rm hom}(\zeta,\nu)=\lim_{T\to+\infty}\frac{1}{T^{n-1}}\inf\Big\{\sum_{i\in Z_1(TQ^\nu)}\psi_i^s(\{u^{i+j}\}_{j\in\Z^n})\colon u\in\A_1^{\sqrt{n}L}(u_{\zeta,\nu},TQ^\nu)\Big\}.
\end{align}
\end{theorem}
\noindent The proof of Theorem \ref{thm:homogenization} will be established in Sections \ref{sect:bulk} and \ref{sect:surface} below in which we treat separately the bulk and the surface energy density. As a preliminary step it is useful to compare the two operators $|\nabla_{\e,L}|$ and $|\nabla_{\e}|$.
\begin{lem}\label{lem:comp:gradients}
There exist constants $\hat{c}_1, \hat{c}_2>0$ depending only on $n,p$ and $L$ such that for every $u:Z_\e(\Rn)\to\Rd$ and every $i\in Z_\e(\Rn)$ there holds
\begin{align}\label{comparison:gradients}
|\nabla_{\e,L}u|^p(i)\leq \hat{c}_1\hspace{-0.5em}\sum_{j\in Z_\e(Q_{\e L}(i))}\hspace{-1em}|\nabla_\e u|^p(j),
\end{align}
and for every $A\subset\Rn$ we have
\begin{align}\label{comparison:gradients:min}
\sum_{i\in Z_\e(A)}\min\Big\{|\nabla_{\e,L}u|^p(i),\frac{1}{\e}\Big\}\leq\hat{c}_2\hspace{-0.5em}\sum_{i\in Z_\e(A+\e L[-1,1]^n)}\hspace{-1.5em}\min\Big\{|\nabla_{\e}u|^p(i),\frac{1}{\e}\Big\}.
\end{align}
\end{lem}
\begin{proof}
Let $u:\Z_\e(\Rn)\to\Rd$ and $i\in Z_\e(\Rn)$. By Jensen's inequality we have
\begin{align}\label{est:comp:gradients:1}
|\nabla_{\e,L}u|^p(i)\leq (\#Z_1(LQ))^{p-1}\sum_{\xi\in Z_1(LQ)}\left|\frac{u^i-u^{i+\e\xi}}{\e}\right|^p.
\end{align} 
Moreover, for any $\xi\in Z_1(LQ)$ there exists a sequence of lattice points $i_0,\ldots, i_{|\xi|_1}\in Z_\e(Q_{\e L}(i))$ with the following properties: $i_0=i$, $i_{|\xi|_1}=\xi$ and for every $h\in\{1,\ldots,n\}$ there exists $i(h)\in\{1,\ldots,n\}$ such that $i_h\in\{i_{h-1}+e_{i(h)},i_{h-1}-e_{i(h)}\}$. Thus, using again Jensen's inequality we obtain
\begin{align*}
\left|\frac{u^i-u^{i+\e\xi}}{\e}\right|^p &=\Big|\sum_{h=1}^{|\xi|_1}D_\e^{\pm e_{i(h)}}u(i_{h-1})\Big|^p\leq |\xi|_1^{p-1}\sum_{h=1}^{|\xi|_1}|D_\e^{\pm e_{i(h)}}u(i_{h-1})|^p\leq |\xi|_1^p\sum_{j\in Z_\e(Q_{\e l}(i))}|\nabla_\e u|^p(j).
\end{align*}
Summing the above estimate over $\xi\in Z_1(LQ)$ from \eqref{est:comp:gradients:1} we deduce 
\begin{align*}
|\nabla_{\e,L}u|^p(i)\leq (\#(Z_1(LQ))^{p-1}\hspace{-0.5em}\sum_{\xi\in Z_1(LQ)}\hspace{-0.7em}|\xi|_1^p\hspace{-0.7em}\sum_{j\in Z_\e(Q_{\e L}(i))}\hspace{-1em}|\nabla_{\e}u|^p(j)\leq \frac{n^pL^p}{2^p}(\#(Z_1(LQ))^{p}\hspace{-0.5em}\sum_{j\in Z_\e(Q_{\e L}(i))}\hspace{-1em}|\nabla_{\e}u|^p(j),
\end{align*}
which gives \eqref{comparison:gradients} with $\hat{c}_1:=\frac{n^pL^p}{2^p}(\#(Z_1(LQ))^{p}$.

Now \eqref{comparison:gradients:min} is a direct consequence of \eqref{comparison:gradients}. In fact, using \eqref{comparison:gradients} together with the subadditvity of the $\min$, for any $A\subset\Rn$ we obtain
\begin{align*}
&\sum_{i\in Z_\e(A)}\min\Big\{|\nabla_{\e,L}u|^p(i),\frac{1}{\e}\Big\}\leq\sum_{i\in Z_\e(A)}\min\Big\{\hat{c}_1\hspace{-1em}\sum_{j\in Z_\e(Q_{\e L}(i))}\hspace{-1em}|\nabla_{\e}u|^p(j),\frac{1}{\e}\Big\}\\
&\hspace{1em}\leq\sum_{i\in Z_\e(A)}\sum_{j\in Z_\e(Q_{\e L}(i))}\hspace{-1em}\min\Big\{\hat{c}_1|\nabla_{\e}u|^p(j),\frac{1}{\e}\Big\}\leq\max\{\hat{c}_1,1\}\sum_{j\in Z_1(LQ)}\sum_{i\in Z_\e(A)}\min\Big\{|\nabla_\e u|^p(i+\e j),\frac{1}{\e}\Big\}\\
&\hspace{1em}\leq\max\{\hat{c}_1,1\}\#Z_1(LQ)\sum_{i\in Z_\e(A+\e L[-1,1]^n)}\min\Big\{|\nabla_\e u|^p(i),\frac{1}{\e}\Big\},
\end{align*}
hence \eqref{comparison:gradients:min} follows by setting $\hat{c}_2:=\max\{\hat{c}_1,1\}\#Z_1(LQ)$.
\end{proof}
\subsubsection{The bulk energy density}\label{sect:bulk}
In this section we show that the bulk energy density $\bar{f}$ in \eqref{hom:integrands} coincides with $f_{\rm hom}$ as in \eqref{fhom}. This will be done by comparing our functionals with a class of functionals that fall into the framework of \citep{BK18}. More precisely, we introduce rescaled interaction-energy densities $\psi_i^{\e,b}:(\Rd)^{Z_\e(\Omega_i)}\to[0,+\infty)$ given by
\begin{align*}
\psi_i^{\e,b}(\{z^j\}_{j\in Z_\e(\Omega_i)}):=
\begin{cases}
\psi_{\frac{i}{\e}}^b(\{\tfrac{1}{\e}z^{\e j}\chi_{LQ}^{\e j}\}_{j\in\Z^n}) &\text{if}\ i\in Z_\e(\Omega_\e^L),\\
\displaystyle\dsum{k=1}{\e e_k\in\Omega_i}^n|D_\e^kz(0)|^p &\text{if}\ i\in Z_\e(\Omega\setminus\Omega_\e^L),
\end{cases}
\end{align*}
and we consider the functionals $G_\e:L^1(\Omega;\Rd)\times\A(\Omega)\to[0,+\infty]$ defined by setting
\begin{align}\label{def:bulkfunctionals}
G_\e(u,A):=\sum_{i\in Z_\e(A)}\e^n\psi_{i}^{\e,b}(\{u^{i+j}\}_{j\in Z_\e(\Omega_i)}),\quad\text{for}\ u\in\A_\e(\Omega;\Rd),
\end{align}
and extended to $+\infty$ on $L^1(\Omega;\Rd)\setminus\A_\e(\Omega;\Rd)$. 

We show that the functions $\psi_i^b$ have the same properties as the functions $\phi_i^k:(\R^d)^{\Z^n}\to[0,+\infty)$ defined in \citep[Section 5]{BK18} for $k=L$ fixed. In addition, they satisfy a suitable upper bound (see \ref{Hb7} below).
\begin{lem}[Properties of $\psi_i^b$]\label{lem:hypotheses:psi}
Suppose that $\psi_i^\e:(\R^d)^{Z_\e(\R^n)}\to[0,+\infty)$ are $\e K$-periodic in $i$, satisfy \ref{H1}--\ref{H5} with $Z_\e(\Omega_i)$ replaced by $Z_\e(\R^n)$, and suppose that in addition  \eqref{cond:finite-range0} is satisfied. Assume moreover that there exists $\psi_i^b:(\R^d)^{\Z^n}\to[0,+\infty)$ such that \ref{Hpsi1} holds true. Then the functions $\psi_i^b$ are $K$-periodic in $i$ and satisfy conditions \ref{H1}--\ref{H2} with $Z_\e(\Omega_i)$ replaced by $\Z^n$. Moreover, the following holds true for every $i\in\Z^n$.
\begin{enumerate}[label={\rm(${\rm H_b}$\arabic*)}]
\setcounter{enumi}{3}
\item\label{Hb4} {\rm (lower bound)} For every $z:\Z^n\to\R^d$ there holds
\begin{align*}
\psi_i^b(\{z^j\}_{j\in\Z^n})\geq c_2\sum_{k=1}^n|D_1^kz(0)|^p;
\end{align*}
\item\label{Hb5} {\rm (locality)} for all $z,w:\Z^n\to\R^d$ with $z^j=w^j$ for all $j\in Z_1(LQ)$ we have
\begin{align*}
\psi_i^b(\{z^j\}_{j\in\Z^n})=\psi_i^b(\{w^j\}_{j\in\Z^n});
\end{align*}
\item\label{Hb6} {\rm (controlled non-convexity)} there exists $c_4>0$ such that for all $z,w:\Z^n\to\R^d$ and every cut-off $\varphi:\R^n\to[0,1]$ we have
\begin{align*}
&\psi_i^b(\{\varphi^jz^j+(1-\varphi^j)w^j\}_{j\in\Z^n})\leq c_3\big(\psi_i^b(\{z^j\}_{j\in\Z^n}+\psi_i^b(\{w^j\}_{j\in\Z^n}\big)\\
&\hspace*{1em}+c_4\sum_{j\in Z_1(LQ)}\dsum{\xi\in Z_1(LQ)}{j+\xi\in LQ}\Big(\sup_{\substack{l\in Z_1(LQ)\\k\in\{1,\ldots,n\}}}\hspace*{-0.5em}|D_1^k\varphi(l)|^p|z(j+\xi)-w(j+\xi)|^p+|D_1^\xi z(j)|^p+|D_1^\xi w(j)|^p\Big);
\end{align*}
\item\label{Hb7} {\rm (upper bound)} there exists $c_5=c_5(n,L,p)>0$ such that for all $z:\Z^n\to\R^d$ there holds
\begin{align*}
\psi_i^b(\{z^j\}_{j\in\Z^n})\leq c_5(|\nabla_{1,L}z|^p(0)+1).
\end{align*}
\end{enumerate}
\end{lem}
\begin{proof}
We first show that $\psi_i^b$ is $K$-periodic in $i$. Fix $\eta>0$ and let $z:\Z^n\to\R^d$ be arbitrary. We find $\bar{\e}=\bar{\e}(z,\eta)>0$ corresponding to \ref{Hpsi1} with $\Lambda_z=|\nabla_{1,L}z|(0)<+\infty$ such that for all $\e\in(0,\bar{\e})$ and for all $i\in\Z^n$ we have
\begin{align}\label{est:periodicity1}
\psi_{\e i}^\e(\{\e z^{\frac{j}{\e}}\}_{j\in Z_\e(\R^n)})-\eta<(\psi_i^b(\{z^j\}_{j\in\Z^n})<\psi_{\e i}^\e(\{\e z^{\frac{j}{\e}}\}_{j\in Z_\e(\R^n)})+\eta.
\end{align}
Thus, for all $k\in\{1,\ldots,n\}$ the $K$-periodicity of $\psi_{\e i}^\e$ together with the fact that \eqref{est:periodicity1} holds uniformly in $i$ ensure that
\begin{align*}
\psi_{i+Ke_k}^b(\{z^j\}_{j\in\Z^n})<\psi_{\e (i+Ke_k)}^\e(\{\e z^{\frac{j}{\e}}\}_{j\in Z_\e(\R^n)})+\eta=\psi_{\e i}^\e(\{\e z^{\frac{j}{\e}}\}_{j\in Z_\e(\R^n)})+\eta<\psi_{i}^b(\{z^j\}_{j\in\Z^n})+2\eta.
\end{align*}
Using the first inequality in \eqref{est:periodicity1} the same argument as above then leads to
\begin{align*}
\psi_i^b(\{z^j\}_{j\in\Z^n})-2\eta<\psi_{i+Ke_k}^b(\{z^j\}_{j\in\Z^n})<\psi_i^b(\{z^j\}_{j\in\Z^n})+2\eta,
\end{align*}
and we conclude by the arbitrariness of $\eta>0$. 

An analogous argument shows that \ref{H1}--\ref{H2} transfer from $\psi_{\e i}^\e$ to $\psi_i^b$ and that \ref{Hb5} follows from \eqref{cond:finite-range1}.
Moreover, for every $\eta>0$ and $z:\Z^n\to\R^d$  there exists $\bar{\e}=\bar{\e}(z,\eta)>0$ such that for all $\e\in (0,\bar{\e})$ and every $i\in\Z^n$ we have
\begin{align*}
\psi_i^b(\{z^j\}_{j\in\Z^n})>\psi_{\e i}^\e(\{\e z^{\frac{j}{\e}}\}_{j\in Z_\e(\e LQ)})-\eta\geq c_2\min\Big\{\sum_{k=1}^n|D_1^k z(0)|^p,\frac{1}{\e}\Big\}-\eta=c_2\sum_{k=1}^n|D_1^k z(0)|^p-\eta,
\end{align*}
hence \ref{Hb4} follows again by the arbitrariness of $\eta>0$.

We continue proving \ref{Hb6}. Let $(c_{\e}^{j,\xi})$ be the sequence provided by \ref{H5}. In view of \eqref{H5:summability} there exists $\e_0>0$ such that
\begin{align*}
c_4:=\sup_{\e\in(0,\e_0)}\sum_{j\in Z_\e(\e LQ)}\dsum{\xi\in Z_1(LQ)}{j+\e\xi\in\e LQ}c_\e^{j,\xi}<+\infty.
\end{align*}
Fix $\eta>0$; for any $z,w:\Z^n\to\R^d$ and $\varphi:\Z^n\to[0,1]$ we find $\bar{\e}=\bar{\e}(z,w,\varphi,\eta)\in (0,\e_0)$ such that for all $\e\in (0,\bar{\e})$ and for all $i\in\Z^n$ there holds
\begin{align*}
&\psi_i^b(\{\varphi^jz^j+(1-\varphi^j)w^j\}_{j\in\Z^n})\leq\psi_{\e i}^\e(\{\varphi^{\frac{j}{\e}}\e z^{\frac{j}{\e}}+(1-\varphi^{\frac{j}{\e}})\e w^{\frac{j}{\e}}\}_{j\in Z_\e(\R^n)})+\eta,\\
&\psi_{\e i}^\e(\{\e z^{\frac{j}{\e}}\}_{j\in Z_\e(\R^n)})+\psi_{\e i}^\e(\{\e w^{\frac{j}{\e}}\}_{j\in Z_\e(\R^n)})\leq \psi_i^b(\{z^j\}_{j\in\Z^n})+\psi_i^b(\{w^j\}_{j\in\Z^n})+\eta.
\end{align*}
Then \ref{H5} together with \eqref{cond:finite-range0} yield
\begin{align*}
\psi_i^b(\{\varphi^jz^j+(1-\varphi^j)w^j\}_{j\in\Z^n})\leq c_3\big(\psi_i^b(\{z^j\}_{j\in\Z^n})+\psi_i^b(\{w^j\}_{j\in\Z^n})+\eta\big)+R^\e(z,w,\varphi)+\eta,
\end{align*}
where
\begin{align*}
R^\e(z,w,\varphi)=\sum_{j\in Z_\e(\e LQ)}\dsum{\xi\in Z_1(LQ)}{j+\e\xi\in\e LQ} &c_\e^{j,\xi}\big(\sup_{\substack{l\in Z_1(LQ)\\k\in\{1,\ldots,n\}}}|D_1^k\varphi(l)|^p|z(\tfrac{j}{\e}+\xi)-w(\tfrac{j}{\e}+\xi)|^p\big)\\
&+c_\e^{j,\xi}\big(|D_1^\xi z(\tfrac{j}{\e})|^p+|D_1^\xi w(\tfrac{j}{\e})|^p\big).
\end{align*}
Since $\bar{\e}\in(0,\e_0)$ we have $c_\e^{j,\xi}\leq c_4$ for all $\e\in (0,\bar{\e})$, $j\in Z_\e(\e LQ)$ and $\xi\in Z_1(LQ)$. Hence
\begin{align*}
R^\e(z,w,\varphi)\leq c_4\hspace{-0.4em}\sum_{j\in Z_1(LQ)}\hspace*{-0.2em}\dsum{\xi\in Z_1(LQ)}{j+\xi\in LQ}\hspace*{-0.8em}\Big(\sup_{\substack{l\in Z_1(LQ)\\k\in\{1,\ldots,n\}}}\hspace*{-0.5em}|D_1^k\varphi(l)|^p|z(j+\xi)-w(j+\xi)|^p+|D_1^\xi z(j)|^p+|D_1^\xi w(j)|^p\Big)
\end{align*}
and \ref{Hb6} follows by the arbitrariness of $\eta>0$. 

Using a similar argument we eventually verify \ref{Hb7}. We consider the sequence $(c_{\e,\alpha}^{j,\xi})$ provided by \ref{H4} and we remark that thanks to \eqref{H4:summability1} there exists $\e_0>0$ such that
\begin{align}\label{eps0}
\bar{c}_5:=\sup_{\e\in(0,\e_0)}\sum_{j\in Z_\e(\e LQ)}\dsum{\xi\in Z_1(LQ)}{j+\e\xi\in\e LQ}c_{\e,1}^{j,\xi}<+\infty.
\end{align}
For any $z:\Z^n\to\Rd$ we choose $\bar{\e}=\bar{\e}(z)\in(0,\e_0)$ such that $\psi_i^b(\{z^j\}_{j\in\Z^n})<\psi_{\e i}^\e(\{\e z^{j/\e}\}_{j\in Z_\e(\R^n)})+1$ for every $\e\in(0,\bar{\e})$. Moreover we define a constant function $\hat{z}:\Z^n\to\R^d$ by setting $\hat{z}^j:=z^0$ for every $j\in\Z^n$. Since $\bar{\e}<\e_0$, \ref{H4} and \eqref{upperbound:constant} in Remark \ref{rem:hypotheses:phi} yield for any $\e\in(0,\bar{\e})$ the estimate
\begin{align*}
\psi_i^b(\{z^j\}_{j\in\Z^n}) &\leq c_1+2+\hspace*{-1em}\sum_{j\in Z_\e(\e LQ)}\dsum{\xi\in Z_1(LQ)}{j+\e\xi\in\e LQ}\hspace*{-1em}c_{\e,1}^{j,\xi}|D_1^\xi z(\tfrac{j}{\e})|^p\leq c_1+2+\bar{c}_5\hspace*{-0.5em}\sum_{j\in Z_1(LQ)}\dsum{\xi\in Z_1(LQ)}{j+\xi\in LQ}\hspace*{-0.5em}|D_1^\xi z(j)|^p.
\end{align*}
Finally, the last term in the estimate above can be bounded via
\begin{align*}
\sum_{j\in Z_1(LQ)}\dsum{\xi\in Z_1(LQ)}{j+\xi\in LQ}\hspace*{-0.5em}|D_1^\xi z(j)|^p\leq 2^{p-1}(1+\#Z_1(LQ))|\nabla_{1,L}z|^p(0),
\end{align*}
hence we obtain \ref{Hb7} by setting $c_5:=\max\{c_1+2,\bar{c}_52^{p-1}(1+\#Z_1(LQ))\}$.
\end{proof}
\begin{rem}\label{rem:upperbound}
The arguments used to verify \ref{Hb7} also show that for all $\e\in(0,\e_0)$ with $\eps_0$ as in \eqref{eps0}, for all $i\in Z_\e(\R^n)$ and for all $z:Z_\e(\R^n)\to\R^d$ there holds
\begin{align*}
\psi_i^\e(\{z^j\}_{j\in Z_\e(\R^n)})\leq c_5(|\nabla_{\e,L}z|^p(0)+1).
\end{align*}
\end{rem}
\noindent Thanks to Lemma \ref{lem:hypotheses:psi} the following is a consequence of \citep[Theorem 5.1]{BK18}.
\begin{theorem}\label{Gammaconvergence:bulk}
Let $G_\e:L^1(\Omega;\Rd)\times\A(\Omega)\to[0,+\infty]$ be given by \eqref{def:bulkfunctionals} and suppose that the functions $\psi_i^\e:(\Rd)^{Z_\e(\Rn)}\to[0,+\infty)$ are $\e K$-periodic in $i$, satisfy \ref{H1}--\ref{H5} with $Z_\e(\Omega_i)$ replaced by $Z_\e(\Rn)$, and the locality condition \eqref{cond:finite-range0}. Assume that in addition Hypotheses \ref{Hpsi1} holds true. Then $G_\e$ $\Gamma$-converges in the strong $L^p(\Omega;\Rd)$-topology to the functional $G:L^p(\Omega;\Rd)\to[0,+\infty]$ given by
\begin{align*}
G(u)=\int_\Omega f_{\rm hom}(\nabla u)\dx,\qquad u\in W^{1,p}(\Omega;\Rd)
\end{align*}
and extended by $+\infty$ in $L^p(\Omega;\Rd)\setminus W^{1,p}(\Omega;\Rd)$, where the integrand $f_{\rm hom}$ is given by \eqref{fhom}. In particular, the limit defining $f_{\rm hom}$ exists and is independent of the $\Gamma$-converging subsequence.
\end{theorem}
\begin{rem}\label{rem:convergence:bulk}
Note that Theorem \ref{Gammaconvergence:bulk} holds also locally, \ie for every $A\in\A(\Omega)$ and every $u\in W^{1,p}(\Omega;\Rd)$ we have
\begin{align*}
\Gamma\hbox{-}\lim_{\e\to 0}G_\e(u,A)=\int_A f_{\rm hom}(\nabla u)\dx.
\end{align*}
Moreover, thanks to the finite-range assumption \eqref{cond:finite-range0} the width of the boundary layer in the definition of $f_{\rm hom}$ can be chosen as $\sqrt{n}L$ (instead of $\sqrt{T}$ as in \citep[Theorem 5.1]{BK18}).
\end{rem}
Thanks to \ref{Hpsi1} we can compare the two discrete energies $F_\e$ and $G_\e$ following a similar strategy as in \citep{Ruf17}. To this end it is convenient to recall the notion of discrete maximal function and some of its properties that have been proved in \citep{Ruf17} (see also \citep{FMP}). 

Given $\e>0$, $v:Z_\e(\R^n)\to\R$ and $r>0$ we define the maximal function $\M_\e^r v:Z_\e(\R^n)\to[0,+\infty)$ by setting
\begin{align*}
\M_\e^rv(i):=\sup_{s\in (0,r)}\frac{1}{\# Z_\e(\overline{B}_s^{|\cdot|_1}(i))}\sum_{j\in Z_\e(\overline{B}_s^{|\cdot|_1}(i))}|v^j|,
\end{align*}
where $\overline{B}_s^{|\cdot|_1}(i)$ is the closed ball of radius $s$ around $i$ with respect to the $|\cdot|_1$-norm. The following lemma is a consequence of \citep[Lemma 5.16 and Remark 5.17]{Ruf17}.
\begin{lem}\label{lem:maxfunction}
There exists a constant $\bar{c}>0$ such that for all $\e>0$ and for every $u:Z_\e(\R^n)\to\R^d$ there holds
\begin{align*}
|u^i-u^j|\leq\bar{c}|i-j|_1\Big(\M_\e^{\bar{c}|i-j|_1}|\nabla_\e u|(i)+\M_\e^{\bar{c}|i-j|_1}|\nabla_\e u|(j)\Big)\quad\text{for every}\ i,j\in Z_\e(\R^n).
\end{align*}
\end{lem}
\noindent Moreover, the following result has been established in \citep[Lemma 5.18]{Ruf17}.
\begin{lem}\label{lem:equiint}
Let $x_0\in\R^n$, $\lambda>0$ and suppose that $u_\e:Z_\e(\R^n)\to\R^d$ satisfy
\begin{align*}
\sup_{\e>0}\sum_{i\in Z_\e(B_{(3+6\bar{c}\sqrt{n})\lambda}(x_0))}|\nabla_\e u_\e|^p(i)<+\infty,
\end{align*}
where $\bar{c}$ is as in Lemma \ref{lem:maxfunction}. Then there exist a subsequence $(\e_h)$ and functions $w_j:Z_{\e_h}(\R^n)\to\R^d$ such that $|\nabla_{\e_h}w_h|^p$ is equiintegrable on $B_{2\lambda}(x_0)$ and
\begin{align}\label{vanishingset1}
\lim_{h\to+\infty}\e_h^n\#\{i\in Z_{\e_h}(B_{2\lambda}(x_0))\colon u_{\e_h}\nequiv w_h\ \text{on}\ \clBone_{\e_h}(i)\}=0.
\end{align}
\end{lem}
\begin{rem}\label{rem:equiint}
Let the sequences $(u_\e)$, $(\e_h)$ and $(w_h)$ be as in Lemma \ref{lem:equiint}. Then we also have
\begin{align}\label{vanishingsetL}
\lim_{h\to+\infty}\e_h^n\#\{i\in Z_{\e_h}(B_\lambda(x_0))\colon u_{\e_h}\nequiv w_{\e_h}\ \text{on}\ Z_{\e_h}(Q_{\e_h L}(i))\}=0.
\end{align}
To verify \eqref{vanishingsetL} we denote by $\mathcal{U}_h$ the set in \eqref{vanishingset1} and by $\mathcal{U}_h^L$ the set in \eqref{vanishingsetL} and we remark that for every $i\in\mathcal{U}_h^L$ there exists $j_i\in Q_{\e_h L}(i)$ such that $u_{\e_h}^{j_i}\neq w_h^{j_i}$. Since $i\in B_\lambda(x_0)$ we have $j_i\in B_{2\lambda}(x_0)$ for $h$ sufficiently large, so that $j_i\in \mathcal{U}_h$. Hence for $h$ sufficiently large we get
\begin{align*}
\e_h^n\#\mathcal{U}_h^L\leq \e_h^n\sum_{j\in\mathcal{U}_h}\#\{i\in\mathcal{U}_h^L\colon j\in Q_{\e_h L}(i)\}\leq cL^n\e_h^n\#\mathcal{U}_h\ \to 0\ \text{as}\ h\to+\infty.
\end{align*}
\end{rem}
\noindent We are now in a position to prove the following result.
\begin{prop}\label{prop:bulk}
Let the sequence $(F_\e)$ be defined according to \eqref{def:energy} with $\phi_i^\e:(\R^d)^{Z_\e(\Omega_i)}\to[0,+\infty)$ as in \eqref{def:phi:periodic} and assume the functions $\psi_i^\e:(\Rd)^{Z_\e(\Rn)}\to[0,+\infty)$ are $\e K$-periodic in $i$, satisfy \ref{H1}--\ref{H5} with $Z_\e(\Omega_i)$ replaced by $Z_\e(\Rn)$, the locality condition \eqref{cond:finite-range0}, and \ref{Hpsi1}. Then $\bar{f}(M)=f_{\rm hom}(M)$ for every $M\in\R^{d\times n}$, where $\bar{f}$ is as in \eqref{hom:integrands}.
\end{prop}
\begin{proof}
The strategy used to derive the formula for $\bar{f}$ follows closely the one used in \citep[Proposition 5.19]{Ruf17}. A main difference with respect to the situation in \citep{Ruf17} is the fact that the interaction-energy densities $\psi_i^\e$ are bounded from below only in terms of $|\nabla_\e u|$, while they can be bounded from above in terms of the finite-range gradient $|\nabla_{\e,L}u|$. To circumvent this additional difficulty we will frequently use Lemma \ref{lem:comp:gradients}. 

The proof is divided into two major steps establishing separately a lower and an upper bound of $\bar{f}$ in terms of $f_{\rm hom}$.
\begin{step}{Step 1} $\bar{f}\geq f_{\rm hom}$

\noindent Fix $M\in\R^{d\times n}$ and let $x_0\in\Omega$ and $\rho>0$ with $B_{\rho}(x_0)\wcont\Omega$. Then
\[|B_1|\bar{f}(M)=\frac{1}{\rho^n}F(u_{M,x_0},B_{\rho}(x_0)).\]
We now estimate $F(u_{M,x_0},B_{\rho}(x_0))$ from below. Without loss of generality we assume $x_0=0$ and for fixed $\rho_0>0$ with $B_{\rho_0}\wcont\Omega$ we choose functions $u_\e\in\A_\e(\Omega;\R^d)$ converging in $L^1(\Omega;\R^d)$ to $u_M$ and satisfying
\[\lim_{\e\to 0} F_\e(u_\e,B_{\rho_0})=F(u_M,B_{\rho_0}).\]
Then $(u_\e)$ is a recovery sequence for $u_M$ on $B_\rho$ for every $\rho\in (0,\rho_0)$, since
\begin{align*}
F(u_M,B_\rho) &=F(u_M,B_{\rho_0})-F(u_M,B_{\rho_0}\setminus\overline{B_\rho})\\
&\geq \lim_{\e\to 0}F_\e(u_\e,B_{\rho_0})-\liminf_{\e\to 0}F_\e(u_\e,B_{\rho_0}\setminus\overline{B_\rho})\geq\limsup_{\e\to 0}F_\e(u_\e,B_\rho),
\end{align*}
where in the first step we used that $F(u_M, B_\rho)$ does not concentrate on the boundary of $B_\rho$.
In particular, we have
\begin{align}\label{est:bulk1}
|B_1|\bar{f}(M)\geq\frac{1}{\rho^n}\limsup_{\e\to 0}F_\e(u_\e,B_\rho)\quad\text{for every}\ \rho\in (0,\rho_0).
\end{align}
We now introduce a constant $\bar{k}>0$ satisfying
\[\bar{k}>3+6\overline{c}\sqrt{n}+|M|,\]
where $\overline{c}$ is as in Remark \ref{lem:maxfunction}. Since $|u_M|\leq |M|\rho\leq\bar{k}\rho$ on $B_\rho$, the truncated functions $T_{\bar{k}\rho}u_\e$ converge to $u_M$ in $L^1(B_\rho,\R^d)$. In particular, in view of Remark \ref{rem:truncation} they still provide a recovery sequence for $u_M$ on $B_\rho$. 

Fix $\eta>0$ and for every $\rho\in (0, (3\bar{k}^2)^{-1}\rho_0)$ let $\bar{\e}_\rho=\bar{\e}(\eta,\frac{\sqrt{n}L}{2}\bar{k}\Lambda_\rho\#Z_1(LQ))$ be given by \ref{Hpsi1} with $\Lambda_\rho$ to be chosen later. We choose $$\e_\rho<\min\Big\{\rho^2,\rho^{\tfrac{p}{p-1}},\bar{\e}_\rho,\frac{\dist_\infty(B_{\rho_0},\partial\Omega)}{L}\Big\}$$ non-decreasing in $\rho$ and satisfying
\begin{align}
F_{\e_\rho}(T_{\bar{k}\rho}u_{\e_\rho},B_{3\bar{k}^2\rho}) &\leq c(|M|^p+1)\rho^n,\label{est:energy}\\
\frac{1}{|B_1|\rho^n}\int_{B_\rho}|T_{\bar{k}\rho}u_{\e_\rho}-u_M|^p\dx &\leq\rho^{p+1}.\label{est:lpnorm}
\end{align}
Here, the first estimate can be realized thanks to \eqref{est:bulk1} and the fact that $\bar{f}(M)\leq c(|M|^p+1)$. Observe that since $\rho<(3\bar{k}^2)^{-1}\rho_0$ our choice of $\e_\rho$ implies that $B_{3\bar{k}^2\rho}\subset B_{\rho_0}\subset\Omega_{\e_\rho}^L$ and hence
\begin{align}\label{energy:vrho}
F_{\e_\rho}(T_{\bar{k}\rho} u_{\e_\rho},B_{3\bar{k}^2\rho})=\hspace*{-1.5em}\sum_{i\in Z_{\e_\rho}(B_{3\bar{k}^2\rho})}\hspace*{-1.8em}\e_\rho^n\psi_i^{\e_\rho}(\{u_{\e_\rho}^{i+j}\chi_{\e_\rho LQ}^j\}_{j\in Z_{\e_\rho}(\R^n)})=\hspace*{-1.5em}\sum_{i\in Z_{1}(B_{3\bar{k}^2\frac{\rho}{\e_\rho}})}\hspace*{-2em}\e_\rho^n\psi_{\e_\rho i}^{\e_\rho}(\{\e_\rho v_{\rho}^{i+j/\e_\rho}\}_{j\in Z_{\e_\rho}(\R^n)}),
\end{align}
where $v_\rho:\Z^n\to\R^d$ is defined by setting 
\[v_\rho^i:=\frac{1}{\e_\rho}T_{\bar{k}\rho}u_{\e_\rho}^{\e_\rho i}\chi_{\Omega}^{\e_\rho i},\quad\text{for every}\ i\in\Z^n.\]
\end{step}
\begin{step}{Substep 1a} Construction of Lipschitz-competitors

\noindent 
We now aim to replace $v_\rho$ by a Lipschitz function $\bar{v}_\rho$ with Lipschitz constant at most $\bar{k}\Lambda_\rho$. To this end we introduce the sets of regular and singular points defined as
\[\mathcal{R}_\rho:=\{i\in Z_1(B_{\bar{k}\frac{\rho}{\e_\rho}})\colon \M^{\bar{k}^2\frac{\rho}{\e_\rho}}_1|\nabla_1v_\rho|\leq\Lambda_\rho\},\qquad \mathcal{S}_\rho:=\{i\in \Z^n\colon |\nabla_1 v_\rho|(i)\geq\Lambda_\rho/2\},\]
respectively.
Note that for every $i,j\in \mathcal{R}_\rho$ thanks to Lemma \ref{lem:maxfunction} we have the Lipschitz estimate
\[|v_\rho^i-v_\rho^j|\leq\bar{c}\sqrt{n}|i-j|\Big(\M^{\bar{k}^2\frac{\rho}{\e_\rho}}_1|\nabla_1 v_\rho|(i)+\M^{\bar{k}^2\frac{\rho}{\e_\rho}}_1|\nabla_1 v_\rho|(j)\Big)\leq\bar{k}\Lambda_\rho|i-j|.\]
Using Kirszbraun's extension theorem we thus find a function $\bar{v}_\rho:\Z^n\to\R^d$ coinciding with $v_\rho$ on $\mathcal{R}_\rho$ and satisfying
$|\bar{v}_\rho^i-\bar{v}_\rho^j|\leq \bar{k}\Lambda_\rho|i-j|$ for every $i,j\in \Z^n$.
In particular, we have
\begin{align}\label{est:gradvbar}
|\nabla_{1,L}\bar{v}_\rho|(i)\leq \frac{\sqrt{n}L}{2}\bar{k}\Lambda_\rho\#Z_1(LQ)\quad \text{for every}\ i\in\Z^n.
\end{align}
In addition, by truncation with the operator $T_{3\bar{k}\frac{\rho}{\e_\rho}}$ we can assume that $\|\bar{v}_\rho\|_{\infty}\leq 9\bar{k}\frac{\rho}{\e_\rho}$.

In the remaining part of this substep we bound the number of points in which $v_\rho$ and $\bar{v}_\rho$ do not coincide, that is the cardinality of $Z_1(B_{\bar{k}\frac{\rho}{\e_\rho}})\setminus \mathcal{R}_\rho$. We first observe that for every $i\in Z_1(B_{\bar{k}\frac{\rho}{\e_\rho}})\setminus\mathcal{R}_\rho$ there exists $s_i\in (0,\bar{k}^2\frac{\rho}{\e_\rho})$ such that
\begin{align*}
\Lambda_\rho\#Z_1(\clBone_{s_i}(i))\leq\sum_{j\in Z_1(\clBone_{s_i}(i))}|\nabla_1v_\rho|(j).
\end{align*}
Applying Vitali's covering lemma we find $\I_\rho\subset Z_1(B_{\bar{k}\frac{\rho}{\e_\rho}})\setminus\mathcal{R}_\rho$ (finite) such that the family $(\clBone_{s_i}(i))_{i\in\I_\rho}$ is disjoint and
\begin{align*}
Z_1(B_{\bar{k}\frac{\rho}{\e_\rho}})\setminus\mathcal{R}_\rho\subset\bigcup_{i\in\I_\rho}\clBone_{5 s_i}(i),
\end{align*}
hence
\begin{align}\label{est:doubling}
\#Z_1(B_{\bar{k}\frac{\rho}{\e_\rho}})\setminus\mathcal{R}_\rho &\leq\#Z_1\Big(\bigcup_{i\in\I_\rho}\clBone_{5s_i}(i)\Big)\leq 5^n\#Z_1\Big(\bigcup_{i\in\I_\rho}\clBone_{s_i}(i)\Big).
\end{align}
To estimate the cardinality of $Z_1\big(\bigcup_i\clBone_{s_i}(i)\big)$ we distinguish between the lattice points in $\bigcup_i\clBone_{s_i}(i)$ belonging to $\mathcal{S}_\rho$ and those that belong to its complement. In fact, since the balls $\clBone_{s_i}(i)$ are disjoint, the definition of $\mathcal{S}_\rho$ implies that
\begin{align*}
\Lambda_\rho\#Z_1\Big(\bigcup_{i\in\I_\rho}\clBone_{s_i}(i)\Big) &\leq\hspace*{-1em}\sum_{j\in Z_1(\bigcup_{i}\clBone_{s_i}(i))}\hspace*{-1.5em}|\nabla_1 v_\rho|(j)\leq\hspace*{-1em}\sum_{j\in\bigcup_{i}\clBone_{s_i}(i)\cap\mathcal{S}_\rho}\hspace*{-1.5em}|\nabla_1 v_\rho|(j)+\frac{\Lambda_\rho}{2}\#Z_1\Big(\bigcup_{i\in\I_\rho}\clBone_{s_i}(i)\Big),
\end{align*}
hence
\begin{align}\label{est:balls}
\#Z_1\Big(\bigcup_{i\in\I_\rho}\clBone_{s_i}(i)\Big)\leq\frac{2}{\Lambda_\rho}\hspace*{-0.3em}\sum_{j\in\bigcup_{i}\clBone_{s_i}(i)\cap\mathcal{S}_\rho}\hspace*{-1.5em}|\nabla_1 v_\rho|(j).
\end{align}
We aim to bound the term on the right-hand side of \eqref{est:balls} via $F_{\e_\rho}(T_{\bar{k}\rho}u_{\e_\rho},B_{3\bar{k}^2\rho})$.
To this end we introduce the set of jump points
\begin{align*}
\J_\rho:=\Big\{i\in\Z^n\colon |\nabla_1v_\rho|^p(i)\geq 1/\e_\rho\Big\}
\end{align*}
and we use H\"older's inequality to obtain the estimate
\begin{align}\label{est:Jrho1}
\sum_{j\in\bigcup_i\clBone_{s_i}(i)\cap\mathcal{S}_\rho\setminus\J_\rho}\hspace*{-2em}|\nabla_1 v_\rho|(j) &\leq \Big(\#\big(\bigcup_{i\in\I_\rho}\clBone_{s_i}(i)\cap\mathcal{S}_\rho\setminus\J_\rho\big)\Big)^{\frac{p-1}{p}}\Big(\sum_{i\in\bigcup_i\clBone_{s_i}(i)\cap\mathcal{S}_\rho\setminus\J_\rho}\hspace*{-2em}|\nabla_1 v_\rho|^p(j)\Big)^{\frac{1}{p}}.
\end{align}
Then by definition for every $j\in\bigcup_i\clBone_{s_i}(i)\cap\mathcal{S}_\rho\setminus\J_\rho$ we have
\begin{align*}
|\nabla_1 v_\rho|^p(j)=\min\big\{|\nabla_1 v_\rho|^p(j),\frac{1}{\e_\rho}\big\}\leq (2n)^{p-1}\Big(\min\big\{\sum_{k=1}^n|D_\e^k v_\rho^j|^p,\frac{1}{\e_\rho}\big\}+\min\big\{\sum_{k=1}^n|D_\e^k v_\rho^{j-e_k}|^p,\frac{1}{\e_\rho}\big\}\Big),
\end{align*}
where in the second step we used the subadditivity of $\min$. Moreover, for every $j\in\bigcup_i\clBone_{s_i}(i)$ there holds
\begin{align}\label{inclusionBone}
j-e_k\in\bigcup_{i\in\I_\rho}\clBone_{s_i+\e_\rho}(i)\subset B_{3\bar{k}^2\frac{\rho}{\e_\rho}}\quad\text{for every}\ k\in\{1,\ldots,n\}.
\end{align}
Thus, from \ref{H3} together the energy bound \eqref{est:energy} we infer 
\begin{align}\label{est:Jrho1a}
\sum_{j\in\bigcup_i\clBone_{s_i}(i)\cap\mathcal{S}_\rho\setminus\J_\rho}\hspace*{-2em}|\nabla_1 v_\rho|^p(j)\leq 2(2n)^{p-1}\sum_{j\in Z_1(B_{3\bar{k}^2\frac{\rho}{\e_\rho}})}\min\big\{\sum_{k=1}^n|D_1^kv_\rho(j)|^p,\frac{1}{\e_\rho}\big\}\leq c\frac{\rho^n}{\e_\rho^n},
\end{align}
where the additional factor $2$ comes from the fact that each term is counted at most twice. Finally, since $|\nabla_{1}v_\rho|\geq\frac{\Lambda_\rho}{2}$ on $\mathcal{S}_\rho$, \eqref{est:Jrho1a} gives
\begin{align}\label{est:card:Srho}
\Big(\frac{\Lambda_\rho}{2}\Big)^p\#\bigcup_{i\in\I_\rho}\clBone_{s_i}(i)\cap\mathcal{S}_\rho\setminus\J_\rho\leq\hspace*{-1em}\sum_{j\in\bigcup_i \clBone_{s_i}(i)\cap\mathcal{S}_\rho\setminus\J_\rho}\hspace*{-1.5em}|\nabla_1v_\rho|^p(j)\leq c\frac{\rho^n}{\e_\rho^n}.
\end{align}
Gathering \eqref{est:Jrho1}, \eqref{est:Jrho1a} and \eqref{est:card:Srho} we eventually deduce that
\begin{align}\label{est:Jrho2}
\frac{2}{\Lambda_\rho}\sum_{j\in\bigcup_i\clBone_{s_i}(i)\cap\mathcal{S}_\rho\setminus\J_\rho}\hspace*{-2em}|\nabla_1 v_\rho|(j)\leq c\Lambda_\rho^{-p}\frac{\rho^n}{\e_\rho^n}.
\end{align}
To estimate the remaining contributions in \eqref{est:balls} we observe that for every $j\in\J_\rho$ there exists $k(j)\in\{1,\ldots,n\}$ such that either $|D_1^kv_\rho^j|^p\geq 1/\e_\rho(2n)^p$ or $|D_1^kv_\rho^{j-e_k}|^p\geq 1/\e_\rho(2n)^p$. Using the inclusion in \eqref{inclusionBone} once more we then obtain
\begin{align*}
\frac{1}{\e_\rho(2n)^p}\#\Big(\bigcup_{i\in\I_\rho}\clBone_{s_i}(i)\cap\J_\rho\Big)\leq 2\sum_{j\in Z_1(B_{3\bar{k}^2\frac{\rho}{\e_\rho}})}\min\big\{\sum_{k=1}^n|D_1^kv_\rho(j)|^p,\frac{1}{\e_\rho}\big\}\leq c\frac{\rho^n}{\e_\rho^n},
\end{align*}
where the additional factor $2$ results again from a possible double counting of interactions.
Moreover, the uniform bound on $v_\rho$ implies $|D_1^kv_\rho(j)|\leq c\bar{k}\frac{\rho}{\e_\rho}$ for every $j\in\Z^n$, so that the above estimate yields
\begin{align}\label{est:Jrho3}
\frac{2}{\Lambda_\rho}\sum_{j\in\bigcup_{i}\clBone_{s_i}(i)\cap\J_\rho}|\nabla_1v_\rho|(j)\leq c\bar{k}\Lambda_\rho^{-1}\frac{\rho}{\e_\rho}\#\Big(\bigcup_{i\in\I_\rho}\clBone_{s_i}(i)\cap\J_\rho\Big)\leq c\rho\Lambda_\rho^{-1}\frac{\rho^n}{\e_\rho^n}.
\end{align}
Combining \eqref{est:doubling}, \eqref{est:balls}, \eqref{est:Jrho2} and \eqref{est:Jrho3} and choosing $\Lambda_\rho=\rho^{\frac{1}{1-p}}$ we finally deduce that
\begin{align}\label{est:Rrho}
\#(Z_1(B_{\bar{k}\frac{\rho}{\e_\rho}})\setminus\mathcal{R}_\rho)\leq c(\rho\Lambda_\rho^{-1}+\Lambda_\rho^{-p})\frac{\rho^n}{\e_\rho^n}=c\rho^{\frac{p}{p-1}}\frac{\rho^n}{\e_\rho^n}.
\end{align}
\end{step}
\begin{step}{Substep 1b} From Lipschitz continuity to equiintegrable gradients

\noindent In this substep we show that the rescaled functions $\tilde{v}_\rho$ obtained by setting
\begin{align*}
\tilde{v}_\rho^i:= \frac{\e_\rho}{\rho}\bar{v}_\rho^{\frac{\rho}{\e_\rho} i}\quad\text{for every}\ i\in Z_{\frac{\e_\rho}{\rho}}(\R^n)
\end{align*}
satisfy the hypotheses of Lemma \ref{lem:equiint} with $\lambda=1$ and $x_0=0$ along the vanishing sequence $\sigma_\rho:=\frac{\e_\rho}{\rho}$. We start by observing that $\tilde{v}_\rho$ satisfy the following conditions.
\begin{enumerate}[label=(\roman*)]
\setlength{\itemsep}{3pt}
\item $\|\tilde{v}_\rho\|_{\infty}\leq 9\overline{k}$;
\item $|\tilde{v}_\rho^i-\tilde{v}_\rho^j|\leq \overline{k}\Lambda_\rho|i-j|$ for all $i,j\in Z_{\sigma_\rho}(\R^n)$;
\item $\tilde{v}_\rho^i=\frac{1}{\rho}T_{\bar{k}\rho}u_{\e_\rho}^{\rho i}$ if $\frac{\rho}{\e_\rho}i\in\mathcal{R}_\rho$.
\end{enumerate} 
Note that (ii) implies that $|\nabla_{\sigma_\rho}\tilde{v}_\rho|^p(i)\leq c\Lambda_\rho^p$ for every $i\in Z_{\sigma_\rho}(\R^n)$. We thus obtain the estimate
\begin{align}\label{est:gradvtilde1}
\sum_{i\in Z_{\sigma_\rho}(B_{\bar{k}})}\sigma_\rho^n|\nabla_{\sigma_\rho}\tilde{v}_\rho|^p(i) &\leq\dsum{i\in Z_{\e_\rho}(B_{\bar{k}\rho})}{\frac{i}{\e_\rho}\in\mathcal{R}_\rho}\frac{\e_\rho^n}{\rho^n}\dsum{k=1}{\frac{i}{\e_\rho}+e_k\in\mathcal{R}_\rho}^n|D_{\e_\rho}^k T_{\bar{k}\rho}u_{\e_\rho}^i|^p+c\Lambda_\rho^p\frac{\e_\rho^n}{\rho^n}\#(Z_1(B_{\bar{k}\frac{\rho}{\e_\rho}})\setminus\mathcal{R}_\rho).
\end{align}
Thanks to \eqref{est:Rrho} we can bound the second term on the right-hand side of \eqref{est:gradvtilde1} by a constant. Moreover, the definition of the maximal function together with the choice of the set $\mathcal{R}_\rho$ implies that for all $i\in Z_{\e_\rho}(B_{\bar{k}\rho})$ with $i/\e_\rho\in\mathcal{R}_\rho$ we have
\begin{align*}
\sum_{k=1}^n|D_{\e_\rho}^k T_{\bar{k}\rho}u_{\e_\rho}^i|^p\leq |\nabla_{\e_\rho}T_{\bar{k}\rho}u_{\e_\rho}|^p(i)=|\nabla_1 v_\rho|^p\leq\Lambda_\rho^p=\rho^\frac{p}{1-p}<\frac{1}{\e_\rho},
\end{align*}
where in the last step we have used that $\e_\rho<\rho^{\frac{p}{p-1}}$. Hence we can bound the first term on the right-hand side of \eqref{est:gradvtilde1} by the energy and use \eqref{est:energy} to deduce that
\begin{align*}
\sum_{i\in Z_{\sigma_\rho}(B_{\bar{k}})}\sigma_\rho^n|\nabla_{\sigma_\rho}\tilde{v}_\rho|^p(i) &\leq \frac{c}{\rho^n}F_{\e_\rho}(T_{\bar{k}\rho}u_{\e_\rho},B_{\bar{k}\rho})+c\leq c.
\end{align*}
Thanks to our choice of $\bar{k}$ Lemma \ref{lem:equiint} then provides us with a subsequence $(\rho_h)$ and functions $w_h:Z_{\sigma_h}(\R^n)\to\R^d$ such that $|\nabla_{\sigma_h}w_h|^p$ is equiintegrable on $B_2$ and
\begin{align}\label{vanishingset}
\lim_{h\to+\infty}\sigma_h^n\#\{i\in Z_{\sigma_h}(B_2)\colon \tilde{v}_{\rho_h}\nequiv w_h\ \text{on}\ \clBone_{\sigma_h}(i)\}=0,
\end{align} 
where we have set $\sigma_h:=\sigma_{\rho_h}$. Moreover, upon truncation we can assume that $\|w_h\|_{\infty}\leq 27\bar{k}$.
\end{step}
\begin{step}{Substep 1c} Conclusion of the lower-bound inequality

\noindent We continue by proving that the sequence $(w_h)$ obtained in Substep 1b converges to $u_M$ in $L^p(B_1;\Rd)$. To simplify notation we set $\e_h:=\e_{\rho_h}$. We start by estimating
\begin{align*}
\|w_h-u_m\|_{L^p(B_1;\Rd)}\leq\|w_h-\frac{1}{\rho_h}T_{\bar{k}\rho_h}u_{\e_h}(\rho_h\cdot)\|_{L^p(B_1;\Rd)}+\|\frac{1}{\rho_h}T_{\bar{k}\rho_h}u_{\e_h}(\rho_h\cdot)-u_M\|_{L^p(B_1;\R^d)}.
\end{align*}
By a change of variables and \eqref{est:lpnorm} we obtain
\begin{align*}
\|\frac{1}{\rho_h}T_{\bar{k}\rho_h}u_{\e_h}(\rho_h\cdot)-u_M\|^p_{L^p(B_1;\R^d)}\leq\frac{1}{\rho_h^{n+p}}\int_{B_{\rho_h}}\hspace*{-0.5em}|T_{\bar{k}\rho_h}u_{\e_h}-u_M|^p\dx\leq\rho_h\ \to 0\ \text{as}\ h\to+\infty.
\end{align*}
Moreover, we denote by $\mathcal{U}_h$ the set in \eqref{vanishingset} and we remark that for all $i\in Z_{\sigma_h}(B_2)\setminus\mathcal{U}_h$ with $i/\sigma_h\in\mathcal{R}_{\rho_h}$ we have $w_h^i=1/\rho_hT_{\bar{k}\rho_h}u_{\e_h}^{\rho_h i}$. 
Thus, the uniform bound on $\|w_h\|_{\infty}$ together with \eqref{est:Rrho}, \eqref{vanishingset} yield
\begin{align*}
\|w_h\hspace{-0.2em}-\frac{1}{\rho_h}T_{\bar{k}\rho_h}u_{\e_h}(\rho_h\cdot)\|^p_{L^p(B_1;\Rd)}&\leq c|M|^p\sigma_h^n\big(\#\mathcal{U}_h\hspace{-0.2em}+\hspace{-0.2em}\#(Z_1(B_{2\frac{\rho_h}{\e_h}})\setminus\mathcal{R}_{\rho_h})\big)\leq c|M|^p\big(\sigma_h^n\#\mathcal{U}_h\hspace{-0.2em}+\hspace{-0.2em}\rho_h^{\frac{p}{p-1}}\big),
\end{align*}
where the second inequality follows from \eqref{est:Rrho}.
Thanks to \eqref{vanishingset} we conclude that $w_h\to u_M$ in $L^p(B_1;\Rd)$.

We finally show that up to a small error $1/\rho_h^n F_{\e_h}(u_{\e_h},B_{\rho_h})$ is asymptotically bounded from below by $|B_1|f_{\rm hom}$. Then the required inequality follows from \eqref{est:bulk1} by letting $h\to +\infty$. We start by introducing the sets
\begin{align*}
\mathcal{U}_h^L &:=\{i\in Z_{\sigma_h}(B_1)\colon \tilde{v}_{\rho_h}\nequiv w_h\ \text{on}\ Q_{\sigma_h L}(i)\},\\
\mathcal{V}_h &:=\{i\in Z_1(B_{\frac{\rho_h}{\e_h}})\colon Z_1(Q_L(i))\subset\mathcal{R}_{\rho_h},\ \sigma_h i\in Z_{\sigma_h}(B_1)\setminus\mathcal{U}_h^L\}.
\end{align*}
and by observing that Remark \ref{rem:equiint} and \eqref{est:Rrho} yield
\begin{align}\label{vanishingset2}
\sigma_h^n\#(Z_1(B_{\frac{\rho_h}{\e_h}})\setminus\mathcal{V}_h)\leq\sigma_h^n\big(\#\mathcal{U}_h^L+cL^n\#(Z_1(B_{\bar{k}\frac{\rho_h}{\e_h}})\setminus\mathcal{R}_{\rho_h})\big)\ \to 0\ \text{as}\ h\to+\infty.
\end{align}
Moreover, thanks to the locality property \eqref{cond:finite-range1} we have
\begin{align*}
\frac{1}{\rho_h^n}F_{\e_h}(u_{\e_h},B_{\rho_h}) &\geq\frac{1}{\rho_h^n}F_{\e_h}(T_{\bar{k}\rho_h}u_{\e_h},B_{\rho_h})\geq\sum_{i\in\mathcal{V}_h}\sigma_h^n\psi_{\e_h i}^{\e_h}(\{\e_h\bar{v}_{\rho_h}^{i+\frac{j}{\e_h}}\}_{j\in Z_{\e_h}(\R^n)})\\
&\geq\sum_{i\in\mathcal{V}_h}\sigma_h^n\psi_i^b(\{\bar{v}_{\rho_h}^{i+j}\}_{j\in\Z^n})-\eta,
\end{align*}
where the last inequality follows from \eqref{est:gradvbar} and \ref{Hpsi1} together with the fact that $\e_h<\bar{\e}_{\rho_h}$. By construction $\bar{v}_{\rho_h}^{i+j}=1/\sigma_h w_h^{\sigma_h(i+j)}$ for every $i\in\mathcal{V}_h$ and every $j\in Z_1(LQ)$, hence we obtain
\begin{align}\label{est:bulk2}
\frac{1}{\rho_h^n}F_{\e_h}(u_{\e_h},B_{\rho_h}) &\geq\hspace*{-0.5em}\sum_{i\in Z_{\sigma_h}(B_1)}\hspace*{-1em}\sigma_h^n\psi_{\frac{i}{\sigma_h}}^b(\{\frac{1}{\sigma_h}w_h^{i+\sigma_hj}\}_{j\in\Z^n})-\hspace*{-1.5em}\sum_{i\in Z_1(B_{\frac{\rho_h}{\e_h}})\setminus\mathcal{V}_h}\hspace*{-1.5em}\sigma_h^n\psi_i^b(\{\frac{1}{\sigma_h}w_h^{\sigma_h(i+j)}\}_{j\in\Z^n})-\eta\nonumber\\
&\geq G_{\sigma_h}(w_h,B_1)-c_5\hspace{-1.5em}\sum_{i\in Z_1(B_{\frac{\rho_h}{\e_h}})\setminus\mathcal{V}_h}\hspace{-1.5em}\sigma_h^n\big(|\nabla_{\sigma_h,L} w_h|^p(\sigma_h i)+1)-\eta\nonumber\\
&\geq G_{\sigma_h}(w_h,B_1)-c_5\hspace{-1.5em}\sum_{i\in Z_1(B_{\frac{\rho_h}{\e_h}})\setminus\mathcal{V}_h}\hspace{-1.5em}\sigma_h^n\Big(\hat{c}_1\hspace{-0.5em}\sum_{j\in Z_{\sigma_h}(\sigma_h Q_L(i))}\hspace{-1.5em}|\nabla_{\sigma_h}w_h|^p(j)+1\Big)-\eta,
\end{align}
where $\hat{c}_1$ is given by \eqref{comparison:gradients}. In order to further estimate the second term in \eqref{est:bulk2} we consider the set
\begin{align*}
\mathcal{W}_h:=\{j\in Z_{\sigma_h}(B_{3/2})\colon\exists\ i\in Z_1(B_{\frac{\rho_h}{\e_h}})\setminus\mathcal{V}_h\ \text{s.t.}\ j\in\sigma_h LQ(i)\}
\end{align*}
and for every $j\in\mathcal{W}_h$ we define
\begin{align*}
\gamma_h(j):=\#\{i\in Z_1(B_{\frac{\rho_h}{\e_h}})\setminus\mathcal{V}_h\colon j\in\sigma_h LQ(i)\}.
\end{align*}
Then for $h$ sufficiently large we have
\begin{align}\label{est:bulk4}
\sum_{i\in Z_1(B_{\frac{\rho_h}{\e_h}})\setminus\mathcal{V}_h}\hspace{-1em}\sigma_h^n\sum_{j\in Z_{\sigma_h}(\sigma_h Q_L(i))}\hspace{-1.5em}|\nabla_{\sigma_h}w_h|^p(j)\leq \sum_{j\in\mathcal{W}_h}\sigma_h^n\gamma_h(j)|\nabla_{\sigma_h}w_h|^p(j)\leq c(n,L)\sum_{j\in\mathcal{W}_h}\sigma_h^n|\nabla_{\sigma_h}w_h|^p(j),
\end{align}
where in the second step we used that $\gamma_h(j)\leq\#Z_{\sigma_h}(Q_{\sigma_h L}(j))\leq c L^n$ for every $j\in\mathcal{W}_h$.
We eventually observe that $\#\mathcal{W}_h\leq cL^n\#(Z_1(B_{\frac{\rho_h}{\e_h}})\setminus\mathcal{V}_h)\to 0$ as $h\to+\infty$. Hence the equiintegrability of $|\nabla_{\sigma_h}w_h|^p$ on $B_2$ yields the existence of some $h_\eta>0$ such that
\begin{align*}
c_5\hat{c}_1 c(n,L)\sum_{j\in\mathcal{W}_h}\sigma_h^n|\nabla_{\sigma_h}w_h|^p(j)<\eta\quad\text{for every}\ h\geq h_\eta.
\end{align*}
As a consequence, combining \eqref{est:bulk2} and \eqref{est:bulk4} we obtain
\begin{align}\label{est:bulk5}
\frac{1}{\rho_h^n}F_{\e_h}(u_{\e_h},B_{\rho_h})\geq G_{\sigma_h}(w_h,B_1)-c_5\sigma_h^n\#(Z_1(B_{\frac{\rho_h}{\e_h}})\setminus\mathcal{V}_h)-2\eta,
\end{align}
for all $h\geq h_\eta$.
Thus, since $w_h\to u_M$ in $L^p(B_1,\R^d)$, from \eqref{est:bulk1}, \eqref{vanishingset2} and \eqref{est:bulk5} together with Theorem \ref{Gammaconvergence:bulk} and Remark \ref{rem:convergence:bulk} we deduce that
\begin{align*}
|B_1|\bar{f}(M) &\geq\liminf_{h\to+\infty}G_{\sigma_h}(w_h,B_1)-2\eta\geq G(u_M,B_1)-2\eta=|B_1|f_{\rm hom}(M)-2\eta
\end{align*}
and we conclude by letting $\eta\to 0$.
\end{step}
\begin{step}{Step 2} $\bar{f}\leq f_{\rm hom}$

\noindent In order to prove this inequality we choose a sequence $(u_\e)$ converging to $u_M$ in $L^p(\Omega;\Rd)$ and satisfying
\begin{align*}
\lim_{\e\to 0} G_\e(u_\e,B_{\rho_0})=G(u_M,B_{\rho_0}).
\end{align*}
Fix $\rho\in (0,(3\bar{k}^2)^{-1}\rho_0)$; then the truncated functions $T_{\bar{k}\rho}u_\e$ still provide a recovery sequence for $u_M$ on $B_\rho$. In particular, we obtain
\begin{align}\label{est:fhom1}
|B_1|f_{\rm hom}(M)=\frac{1}{\rho^n}G(u_M,B_\rho)\geq\frac{1}{\rho^n}\limsup_{\e\to 0}G_\e(T_{\bar{k}\rho}u_\e,B_\rho).
\end{align}
In order to use \ref{Hpsi1} to pass from $G_\e$ to $F_\e$ we need replace $T_{\bar{k}\rho u_\e}$ by a sequence of functions with equiintegrable discrete gradients. This can be done by using Lemma \ref{lem:equiint} along the vanishing sequence $\e$ with $\lambda=\rho$. We start by observing that thanks to \ref{Hb4} the functions $T_{\bar{k}\rho}u_\e$ satisfy the assumptions of Lemma \ref{lem:equiint}. In fact,
\begin{align*}
\frac{c_2}{(2n)^{p}}\sum_{i\in Z_\e(B_{\bar{k}\rho})}\e^n|\nabla_\e T_{\bar{k}\rho} u_\e|^p(i)
&\leq c_2\sum_{i\in Z_\e(B_{2\bar{k}\rho})}\e^n\sum_{k=1}^n|D_\e^kT_{\bar{k}\rho}u_\e^i|^p\leq G_\e(u_\e,B_{\rho_0})\leq c\rho^n,
\end{align*}
for some $c>0$ uniformly with respect to $\e$. Thus, Lemma \ref{lem:equiint} ensures the existence of a subsequence $\e_h$ and functions $w_h:Z_{\e_h}(\R^n)\to\R^d$ (possibly depending on $\rho$) such that $|\nabla_{\e_h}w_h|^p$ is equiintegrable on $B_{2_\rho}$ and such that
\begin{align}\label{vanishingset3}
\lim_{h\to +\infty}\e_h^n\#\{i\in Z_{\e_h}(B_{2\rho})\colon T_{\bar{k}\rho}u_{\e_h}\nequiv w_h\ \text{on}\ \clBone_{\e_h}(i)\}=0.
\end{align}
Moreover, upon truncation we can assume that $\|w_h\|_\infty\leq 9\bar{k}$. Denoting by $\mathcal{U}_{\e_h}$ the set in \eqref{vanishingset3} the uniform bound on $\|T_{\bar{k}\rho}u_{\e_h}\|_\infty$ and $\|w_h\|_\infty$ together with \eqref{vanishingset3} give
\begin{align*}
\|w_h-u_M\|_{L^p(B_\rho;\Rd)} &\leq\|w_h-T_{\bar{k}\rho}u_{\e_h}\|_{L^p(B_\rho;\Rd)}+\|T_{\bar{k}\rho}u_{\e_h}-u_M\|_{L^p(B_\rho;\Rd)}\\
&\leq c|M|(\e_h^n\#\mathcal{U}_{\e_h})^{\frac{1}{p}}+\|T_{\bar{k}\rho}u_{\e_h}-u_M\|_{L^p(B_\rho;\Rd)}\ \to 0\ \text{as}\ h\to+\infty.
\end{align*}
Hence, Theorem \ref{thm:int:rep} implies that
\begin{align}\label{est:fbar1}
|B_1|\bar{f}(M)=\frac{1}{\rho^n}F(u_M,B_\rho)\leq\frac{1}{\rho^n}\liminf_{h\to+\infty}F_{\e_h}(w_h,B_\rho),
\end{align}
and it remains to compare $F_{\e_h}(w_h,B_\rho)$ and $G_{\e_h}(T_{\bar{k}\rho}u_{\e_h},B_\rho)$. We start by comparing $G_{\e_h}(T_{\bar{k}\rho}u_{\e_h},B_\rho)$ and $G_{\e_h}(w_h,B_\rho)$. To this end we introduce the sets
\begin{align*}
\mathcal{U}_{\e_h}^L &:=\{i\in Z_{\e_h}(B_\rho)\colon T_{\bar{k}\rho}u_{\e_h}\nequiv w_h\ \text{on}\ Q_{\e_hL}(i)\},\\
\mathcal{V}_{\e_h}^L &:=\{j\in Z_{\e_h}(B_{3\rho/2})\colon \exists\ i\in\mathcal{U}_{\e_h}^L\ \text{s.t.}\ j\in Q_{\e_h L}(i)\},
\end{align*}
and we remark that as in Substep 1c one can show that
\begin{align*}
\lim_{h\to+\infty}\e_h^n\#\mathcal{U}_{\e_h}^L=0,\qquad\lim_{h\to+\infty}\e_h^n\#\mathcal{V}_{\e_h}^L=0.
\end{align*}
Thus, arguing as in \eqref{est:bulk4} and using the equiintegrability of $|\nabla_{\e_h}w_h|^p$ on $B_{2\rho}$ we deduce that there exists $h_1=h_1(\eta,\rho)>0$ such that for all $h\geq h_1$ we have
\begin{align*}
\frac{c_5}{\rho^n}\sum_{i\in\mathcal{U}_{\e_h}^l}\e_h^n(|\nabla_{\e_h,L}w_h|^p(i)+1)\leq\frac{c_5}{\rho^n}\hat{c}_1 c(n,L)\sum_{i\in\mathcal{V}_{\e_h}^L}\e_h^n\big(|\nabla_{\e_h}w_h|^p(i)+1\big)<\eta.
\end{align*}
As a consequence, thanks to the upper bound \ref{Hb7} we obtain
\begin{align}\label{est:Geps}
\frac{1}{\rho^n}G_{\e_h}(T_{\bar{k}\rho}u_{\e_h},B_\rho) &\geq\frac{1}{\rho^n}G_{\e_h}(w_h,B_\rho)-\frac{1}{\rho^n}\sum_{i\in\mathcal{U}_{\e_h}^L}\e_h^n\psi_{\frac{i}{\e_h}}^b(\{\tfrac{1}{\e_h}w_h^{i+\e_hj}\}_{j\in\Z^n})\nonumber\\
&\geq \frac{1}{\rho^n}G_{\e_h}(w_h,B_\rho)-\eta\quad\text{for all}\ h\geq h_1.
\end{align}
We finally estimate from below $G_{\e_h}(w_h,B_\rho)$ in terms of $F_{\e_h}(w_h,B_\rho)$. For every $\Lambda>0$ we set
\begin{align*}
\mathcal{S}_{\e_h}(\Lambda):=\{i\in Z_{\e_h}(B_{2\rho})\colon |\nabla_{\e_h,L}w_h|^p(i)\geq\Lambda\}.
\end{align*}
For every $i\in\mathcal{S}_{\e_h}(\Lambda)$ Lemma \ref{lem:comp:gradients} gives
\begin{align*}
\Lambda\leq|\nabla_{\e_h,L}w_h|^p(i)\leq \hat{c}_1\sum_{j\in Z_{\e_h}(Q_{\e_h L}(i))}|\nabla_{\e_h} w_h|^p(j)\leq\hat{c}_1\#Z_1(LQ)\max_{j\in Z_{\e_h}(Q_{\e_h L}(i))}|\nabla_{\e_h} w_h|^p(j).
\end{align*}
In particular, for every $i\in\mathcal{S}_{\e_h}(\Lambda)\cap B_\rho$ there exists $j_i\in Z_{\e_h}(B_{2\rho})$ with $|\nabla_{\e_h}w_h|^p(j_i)\geq \Lambda/(\hat{c}_1\#Z_1(LQ))$. Setting $\hat{c}:=\hat{c}_1\#Z_1(LQ)$ this gives
\begin{align*}
\sum_{i\in\mathcal{S}_{\e_h}(\Lambda)\cap B_\rho}|\nabla_{\e_h,L}w_h|^p(i) &\leq\sum_{j\in\mathcal{S}_{\e_h}(\Lambda/\hat{c})}|\nabla_{\e_h}w_h|^p(j)\#\{i\in\mathcal{S}_{\e_h}(\Lambda)\colon j\in Q_{\e_h L}(i)\}\\
&\leq\#Z_1(LQ)\sum_{j\in\mathcal{S}_{\e_h}(\Lambda/\hat{c})}|\nabla_{\e_h}w_h|^p(j).
\end{align*}
Thus, for fixed $\eta>0$ the equiintegrability of $|\nabla_{\e_h}w_h|^p$ on $B_{2\rho}$ ensures the existence of $\bar{\Lambda}=\bar{\Lambda}(\eta,\rho)>0$ and $h_2=h_2(\eta,\rho)>0$ such that for every $h\geq h_2$ we have
\begin{align}\label{est:singularset}
\frac{c_5}{\rho^n}\sum_{i\in \mathcal{S}_{\e_h}(\bar{\Lambda})\cap B_\rho}\e_h^n(|\nabla_{\e_h,L}w_h|^p(i)+1)\leq\frac{c_5}{\rho^n}\#Z_1(LQ)\sum_{j\in \mathcal{S}_{\e_h}(\bar{\Lambda}/\hat{c})}\e_h^n(|\nabla_{\e_h}w_h|^p(i)+1)<\eta.
\end{align}
In addition, since $|\nabla_{\e_h,L}w_h|(i)<\bar{\Lambda}^{\frac{1}{p}}$ for all $i\in Z_{\e_h}(B_\rho)\setminus S_{\e_h}(\bar{\Lambda})$, in view of \ref{Hpsi1} there exists $h_3=h_3(\eta,\rho)>0$ such that for all $h\geq h_3$ and for all $i\in Z_{\e_h}(B_\rho)\setminus S_{\e_h}(\bar{\Lambda})$ there holds
\begin{align}\label{est:Hpsi}
|\psi_i^{\e_h}(\{w_h^{i+j}\}_{j\in Z_{\e_h}(\Rn)})-\psi_{\frac{i}{\e_h}}^b(\{\tfrac{1}{\e_h}w_h^{i+\e_hj}\}_{j\in\Z^n})|<\frac{\eta}{|B_1|}.
\end{align}
Combining \eqref{est:singularset} and \eqref{est:Hpsi} in view of Remark \ref{rem:upperbound} we deduce that for all $h\geq\max\{h_2,h_3\}$ we have
\begin{align}\label{est:GepsFeps}
\frac{1}{\rho^n}G_{\e_h}(w_h,B_\rho) &\geq\frac{1}{\rho^n}\sum_{i\in Z_{\e_h}(B_\rho)\setminus\mathcal{S}_{\e_h}(\Lambda_{\eta,\rho})}\hspace*{-2em}\e_h^n(\psi_i^{\e_h}(\{w_h^{i+j}\}_{j\in Z_{\e_h}(\R^n)})-\eta)\nonumber\\
&\geq\frac{1}{\rho^n} F_{\e_h}(w_h,B_\rho)-\eta-o(\e_h)-\frac{c_5}{\rho^n}\sum_{i\in\mathcal{S}_{\e_h}(\bar{\Lambda})\cap B_\rho}\e_h^n(|\nabla_{\e_h,L}w_h|^p(i)+1)\nonumber\\
&\geq\frac{1}{\rho^n}F_{\e_h}(w_h,B_\rho)-2\eta-o(\e_h).
\end{align}
Eventually, gathering \eqref{est:fbar1}, \eqref{est:fhom1}, \eqref{est:Geps} and \eqref{est:GepsFeps} we obtain
\begin{align*}
|B_1|f_{\rm hom}(M)\geq\frac{1}{\rho^n}\liminf_{h\to+\infty}F_{\e_h}(w_h,B_\rho)-3\eta\geq |B_1|\bar{f}(M)-3\eta,
\end{align*}
hence we may conclude letting $\eta\to 0$.
\end{step}
\end{proof}
\subsubsection{The surface energy density}\label{sect:surface}
In this section we finally characterize the surface-energy density of the $\Gamma$-limit. We start by proving some properties of the unscaled interaction-energy densities $\psi_i^s$. Since these properties can be obtained in a similar way as the corresponding properties of $\psi_i^b$ in Lemma \ref{lem:hypotheses:psi} we only sketch the proof.
\begin{lem}\label{lem:properties:psi:s}
Suppose that $\psi_i^\e:(\R^d)^{Z_\e(\R^n)}\to[0,+\infty)$ are $\e K$-periodic in $i$, satisfy \ref{H1}--\ref{H5} with $Z_\e(\Omega_i)$ replaced by $Z_\e(\R^n)$, and suppose that in addition  \eqref{cond:finite-range0} is satisfied. Assume moreover that there exists $\psi_i^s:(\R^d)^{\Z^n}\to[0,+\infty)$ such that \ref{Hpsi2} holds true. Then the functions $\psi_i^s$ are $K$-periodic in $i$ and satisfy Hypotheses \ref{H1}--\ref{H6} with $Z_\e(\Omega_i)$ replaced by $\Z^n$. Moreover, the following holds for every $i\in\Z^n$.
\begin{enumerate}[label={\rm(${\rm H_s}$\arabic*)}]
\setcounter{enumi}{2}
\item\label{Hs3}{\rm (upper bound for constant functions)} For all $z:\Z^n\to\Rd$ with $z\equiv w$ for some $w\in\Rd$ we have $\psi_i^s(\{z^j\}_{j\in\Z^n})=0$;
\item\label{Hs3a}{\rm (upper bound)} there exists $c_6=c_6(n,L)>0$ such that for all $z:\Z^n\to\Rd$ there holds
\begin{align*}
\psi_i^s(\{z^j\}_{j\in\Z^n})\leq c_6(\|z\|_{L^\infty(LQ)}+1);
\end{align*}
\item\label{Hs4}{\rm (locality)} for all $z,w:\Z^n\to\R^d$ with $z^j=w^j$ for all $j\in Z_1(LQ)$ we have
\begin{align*}
\psi_i^s(\{z^j\}_{j\in\Z^n})=\psi_i^s(\{w^j\}_{j\in\Z^n}).
\end{align*}
In particular, $\psi_i^s(\{z^j\}_{j\in\Z^n})=0$ for all $z:\Z^n\to\Rd$ with $z\equiv w$ on $Z_1(LQ)$ for some $w\in\Rd$.
\end{enumerate}
\end{lem}
\begin{proof}
The periodicity of $\psi_i^s$, \ref{H1}--\ref{H6} and \ref{Hs4} follow from the corresponding properties of $\psi_i^\e$ as in the case of $\psi_i^b$. Thus, we only prove \ref{Hs3} and \ref{Hs3a} here. To this end, fix $\eta>0$ and suppose that $z:\Z^n\to\Rd$ is such $z\equiv w$ for some $w\in\Rd$. Then $|\nabla_{1,L}z|(0)=0$ and according to \ref{Hpsi2} we find $\hat{\e}=\hat{\e}(\eta)>0$ such that $\psi_i^s(\{z^j\}_{j\in\Z^n})< \e\psi_{\e i}^\e(\{z^\frac{j}{\e}\}_{j\in Z_\e(\Rn)})+\eta$ for every $\e\in(0,\hat{\e})$ and every $i\in\Z^n$. Thus, \eqref{upperbound:constant} gives
\begin{align*}
\psi_i^s(\{z^j\}_{j\in\Z^n})<\e(c_1+1)+\eta,
\end{align*}
and we obtain \ref{Hs3} by letting first $\e\to 0$ and then $\eta\to 0$.

We continue proving \ref{Hs3a}. Let $\e_0$ and $\bar{c}_5$ be as in \eqref{eps0} and let $z:\Z^n\to\R^d$. Note that either $|\nabla_{1,L}z|(0)=0$ or we can find $\e(z)\in(0,\e_0)$ such that $\e^\frac{1-p}{p}|\nabla_{1,L}z|(0)\geq \Lambda(1)$ for any $\e\in(0,\e(z))$. Thanks to \ref{Hpsi2} there exists $\hat{\e}\in(0,\e(z))$ such that $\psi_i^s(\{z^j\}_{j\in\Z^n})< \e\psi_{\e i}^\e(\{z^\frac{j}{\e}\}_{j\in Z_\e(\Rn)})+1$ for every $\e\in(0,\hat{\e})$ and every $i\in\Z^n$. Arguing as in the proof of Lemma \ref{lem:hypotheses:psi} to obtain \ref{Hb7} we deduce
\begin{align*}
\psi_i^s(\{z^j\}_{j\in\Z^n}) &<\e(c_1+1)+1+\e\bar{c}_5\hspace*{-0.5em}\sum_{j\in Z_1(LQ)}\dsum{\xi\in Z_1(LQ)}{j+\xi\in LQ}\hspace*{-0.5em}\frac{1+|z^{\frac{j}{\e}+\xi}-z^0|}{\e}\\
&\leq\e(c_1+1)+1+\bar{c}_5(1+2\|z\|_{L^\infty(LQ;\Rd)})(\#Z_1(LQ))^2,
\end{align*}
hence \ref{Hs3a} follows by setting $c_6:=2\max\{\bar{c}_5(\#Z_1(LQ))^2,1\}$ and letting $\e\to 0$.
\end{proof}
\begin{rem}\label{rem:continuity:psi}
Thanks to \ref{Hs3} and \ref{Hs4} the continuity assumption \ref{Hpsi3} reads as follows. For every $z,w:\Z^n\to\Rd$ with $|\nabla_{1,L}z|(0)>0$ there holds
\begin{equation*}
\psi_i^s(\{z^j\}_{j\in\Z^n})\geq \psi_i^s(\{w^j\}_{j\in\Z^n})-c_s\sum_{j\in Z_1(Q_L(i))}\dsum{\xi\in Z_1(Q_L(i))}{j+\xi\in Q_L(i)}|z^{j+\xi}-w^{j+\xi}|\ \text{for every}\ i\in\Z^n.
\end{equation*}
\end{rem}
\noindent On account of Lemma \ref{lem:properties:psi:s} we now prove the following proposition.
\begin{prop}\label{prop:surface}
Let $F_\e$ be given by \eqref{def:energy} with $\phi_i^\e:(\R^d)^{Z_\e(\Omega_i)}\to[0,+\infty)$ as in \eqref{def:phi:periodic} and assume that the functions $\psi_i^\e:(\Rd)^{Z_\e(\Rn)}\to[0,+\infty)$ are $\e K$-periodic in $i$, satisfy \ref{H1}--\ref{H5} with $Z_\e(\Omega_i)$ replaced by $Z_\e(\Rn)$, and the locality condition \eqref{cond:finite-range0}. Suppose in addition that there exist $\psi_i^s:(\Rd)^{Z^n}\to[0,+\infty)$ such that \ref{Hpsi2} and \ref{Hpsi3} are satisfied. Then for each pair $(\zeta,\nu)\in\R^{d}\times S^{n-1}$ there exists the limit defining $g_{\rm hom}$ in \eqref{ghom} and $g_{\rm hom}(\zeta,\nu)=\bar{g}(\zeta,\nu)$, where $\bar{g}$ is as in \eqref{hom:integrands}.
\end{prop}
\begin{proof}
Having Lemma \ref{lem:properties:psi:s} at hand the existence of the limit in \eqref{ghom} can be proved as in \citep[Proposition 4.5]{BC17} and we thus omit its proof here. 

Let $(\zeta,\nu)\in\Rd\times S^{n-1}$ be fixed and let us show that $\bar{g}(\zeta,\nu)=g_{\rm hom}(\zeta,\nu)$. To reduce notation for every $T>0$ we set
\begin{align*}
g_T(\zeta,\nu):=\inf\Big\{\sum_{i\in Z_1(TQ^\nu)}\psi_i^s(\{u^{i+j}\}_{j\in\Z^n})\colon u\in\A_1^{\sqrt{n}L}(u_{\zeta,\nu},TQ^\nu)\Big\},
\end{align*}
so that $g_{\rm hom}(\zeta,\nu)=\lim_T 1/T^{n-1} g_T(\zeta,\nu)$.
\begin{step}{Step 1} $\bar{g}(\zeta,\nu)\geq g_{\rm hom}(\zeta,\nu)$

\noindent Let $\bar{g}$ be as in \eqref{hom:integrands}; thanks to formula \ref{derivationformula} in Theorem \ref{thm:int:rep} and Proposition \ref{cor:trans-inv} there exists $x_0\in\Omega$ such that
\begin{align*}
\bar{g}(\zeta,\nu)=\limsup_{\rho\to 0}\lim_{\delta\to 0}\limsup_{\e\to 0}\frac{1}{\rho^{n-1}}\inf\{F_\e(u,Q_\rho^\nu(x_0)\colon u\in\A_{\e}^{\delta}(u_{\zeta,x_0}^\nu, Q_\rho^\nu(x_0))\}.
\end{align*}
Note that to simplify notation we do not relabel the $\Gamma$-converging subsequence. Moreover, from now on we assume $x_0=0$. We fix a number $\alpha\in (0,(p-1)/p)$ whose meaning will become clear later and for every $\rho>0$ we denote by $N_\rho:=\lfloor \rho^{-\alpha}\rfloor$ the integer part of $\rho^{-\alpha}$. We further write $\zeta=(\zeta^1,\ldots,\zeta^d)$ and we choose $\rho\in(0,1)$ with $Q_{2\rho}\wcont\Omega$ such that $2/N_\rho<|\zeta^m|$ for every $m\in\{1,\ldots,d\}$ with $\zeta^m\neq 0$. Let $\delta\in (0,\rho/2)$ and for every $\e>0$ with $\e\sqrt{n} L<\delta$ let $u_\e\in\A_{\e}^{\delta}(u_\zeta^\nu,Q_\rho^\nu)$ be such that
\begin{align}\label{energybound:surface}
F_\e(u_\e,Q_\rho^\nu)\leq F_\e(u_\zeta^\nu,Q_\rho^\nu)\leq c\rho^{n-1}.
\end{align}
Since $\e\sqrt{n}L<\delta<\rho/2$ and $Q_{2\rho}\wcont\Omega$ we can extend $u_\e$ by $0$ outside $\Omega$ without modifying the energy or changing the boundary conditions. Moreover, by truncation we can assume that $\|u_\e\|_{L^\infty}\leq 3|\zeta|$. 

Let us fix $\eta>0$; in the remaining part of this step we construct functions $w_\e:\Z^n\to\Rd$ which are admissible for the minimum problem defining $g_{T_\e}(\zeta,\nu)$ with $T_\e=\rho/\e$ and satisfying for $\e$ sufficiently small (depending on $\eta$) the estimate
\begin{align}\label{est:surface:lb}
\frac{1}{\rho^{n-1}}F_\e(u_\e,Q_\rho^\nu)\geq\frac{1}{T_\e^{n-1}}\sum_{i\in Z_1(T_\e Q^\nu)}\psi_i^s(\{w_\e^{i+j}\}_{j\in\Z^n})-R(\e,\rho)-c\eta,
\end{align}
where the remainder $R(\e,\rho)$ is such that $\lim_\rho\lim_\e R(\e,\rho)=0$ and the constant $c$ depends only on $n,L$ and $\zeta$. Passing to the limit first in $\e$ then in $\delta$ and finally in $\rho$, thanks to the arbitrariness of $u_\e\in\A_\e^\delta(u_{\zeta}^\nu,Q_\rho^\nu)$ we may then deduce that
\begin{align}\label{est:gbar-ghom}
\bar{g}(\zeta,\nu)\geq\liminf_{\e\to 0}\frac{1}{T_\e^{n-1}}g_{T_\e}(\zeta,\nu)-c\eta=g_{\rm hom}(\zeta,\nu)-c\eta,
\end{align} 
which will eventually give the desired inequality by letting $\eta\to 0$.

To obtain the required sequence $(w_\e)$ we carefully combine the arguments used in \citep[Proposition 5.21]{Ruf17} in the discrete setting with those used in \citep[Proposition 6.2]{BDV96} and \citep[Theorem 5.2(d)]{CDSZ17} in the continuum setting. We start by introducing some notation. For every $m\in\{1,\ldots, d\}$ we denote by $(u_\e^i)^m$ the $m$-th component of $u_\e$ and for every $t\in\R$ we consider the superlevel set
\begin{align*}
\mathcal{S}_\e^m(t):=\{i\in Z_\e(Q_\rho^\nu)\colon (u_\e^i)^m\geq t\}.
\end{align*}
Further we introduce the set
\begin{align*}
\mathcal{R}_\e^m(t):=\{i\in Z_\e(Q_\rho^\nu)\colon\exists\ \xi\in Z_1(LQ)\ \text{s.t.}\ i+\e \xi\in Z_\e(\Rn)\setminus\mathcal{S}_\e^m(t),\ i\in\mathcal{S}_\e^m(t)\ \text{or vice versa}\}.
\end{align*}
Finally, let $N\in\N$ with $3|\zeta|+1/N_\rho\leq N$; note that for any $t\in[-N,N]$ and any $m\in\{1,\ldots,d\}$ a point $i\in Z_\e(Q_\rho^\nu)$ belongs to $\mathcal{R}_\e^m(t)$ if and only if $t\in[(u_\e^i)^m,(u_\e^{i+\e \xi})^m)$ or $t\in((u_\e^{i+\e \xi})^m,(u_\e^i)^m]$ for some $\xi\in Z_1(LQ)$. Thus, for any $i\in Z_\e(Q_\rho^\nu)$ we have
\begin{align}\label{est:coarea1}
\int_{-N}^N\chi_{\mathcal{R}_\e^m(t)}(i)\dt\leq\e|\nabla_{\e,L}u_\e|(i).
\end{align}
We choose $\Lambda(\eta)$ according to \ref{Hpsi2} and denote by
\begin{align*}
\J_\e:=\big\{i\in Z_\e(Q_\rho^\nu)\colon |\nabla_{\e,L}u_\e|^p(i)\geq\frac{\Lambda(\eta)^p}{\e}\big\}
\end{align*}
the set of jump points. Without restriction we assume that $\Lambda(\eta)\geq 1$. Summing up \eqref{est:coarea1} over all $i\in Z_\e(Q_\rho^\nu)\setminus\J_\e$ from H\"older's inequality we deduce that
\begin{align}\label{est:coarea2}
\e^{n-1}\int_{-N}^N\#(\mathcal{R}_\e^m(t)\setminus\mathcal{J}_\e)\dt &\leq\sum_{i\in Z_\e(Q_\rho^\nu)\setminus\J_\e}\e^n|\nabla_{\e,L}u_\e|(i)\nonumber\\
&\leq \e^{\frac{n(p-1)}{p}}\Big(\#\big(Z_\e(Q_\rho^\nu)\setminus\J_\e\big)\Big)^{\frac{p-1}{p}}\Big(\sum_{i\in Z_\e(Q_\rho^\nu)\setminus\J_\e}\e^n|\nabla_{\e,L}u_\e|^p(i)\Big)^{\frac{1}{p}}\nonumber\\
&\leq c\Lambda(\eta)\rho^{\frac{n(p-1)}{p}}\Big(\sum_{i\in Z_\e(Q_\rho^\nu)}\hspace*{-0.5em}\e^n\min\big\{|\nabla_{\e,L}u_\e|^p(i),\frac{1}{\e}\big\}\Big)^{\frac{1}{p}}.
\end{align}
Moreover, thanks to Estimate \ref{comparison:gradients:min} in Lemma \ref{lem:comp:gradients} and \ref{H3} we have
\begin{align}\label{est:Lgradient1}
&\sum_{i\in Z_\e(Q_\rho^\nu)}\hspace*{-0.5em}\e^n\min\big\{|\nabla_{\e,L}u_\e|^p(i),\frac{1}{\e}\big\}\leq 2\hat{c}_2\hspace{-2em}\sum_{i\in Z_\e(Q_\rho^\nu+\e L[-1,1]^n)}\hspace*{-2.5em}\e^n\min\Big\{\sum_{k=1}^n|D_\e^ku(i)|^p,\frac{1}{\e}\Big\}\nonumber\\
&\hspace*{2em}\leq 2\hat{c}_2\Bigg(\frac{1}{c_2}F_\e(u_\e,Q_\rho^\nu)+\hspace{-2em}\sum_{i\in Z_\e(Q_\rho^\nu+\e L[-1,1]^n)\setminus Q_\rho^\nu}\hspace*{-3em}\e^n\min\Big\{\sum_{k=1}^n|D_\e^ku_\zeta^\nu(i)|^p,\frac{1}{\e}\Big\}\Bigg),
\end{align}
where in the second step we used the boundary conditions satisfied by $u_\e$. Note that the last term on the right-hand side of \eqref{est:Lgradient1} can be bounded by
\begin{align*}
\e^{n-1}\#\{i\in Z_\e(Q_\rho^\nu+\e L[-1,1])\setminus Q_\rho^\nu\colon\dist(i,\Pi_\nu)\leq\e\}\leq c(L)\e.
\end{align*}
Inserting the above estimate and the energy bound \eqref{energybound:surface} in \eqref{est:Lgradient1}, the estimate in \eqref{est:coarea2} can be continued to
\begin{align*}
\e^{n-1}\int_{-N}^N\#(\mathcal{R}_\e^m(t)\setminus\mathcal{J}_\e)\dt\leq c\Lambda(\eta)\rho^{\frac{n(p-1)}{p}}\big(\rho^{n-1}+\e\big)^{\frac{1}{p}}\leq c\Lambda(\eta)\big(\rho^{\frac{np-1}{p}}+\rho^\frac{n(p-1)}{p}\e^{\frac{1}{p}}\big).
\end{align*}
Hence for every integer $l$ with $-N N_\rho\leq l\leq N N_\rho$ there exists $t_l^m\in [l/N_\rho,(l+1)/N_\rho)$ such that
\begin{align}\label{est:optimaljumpset}
\e^{n-1}\sum_{l=-N N_\rho}^{N N_\rho-1}\hspace{-1em}\#(\mathcal{R}_\e^m(t_l^m)\setminus\J_\e)\leq \e^{n-1} N_\rho\int_{-N}^N\#(\mathcal{R}_\e^m(t)\setminus\J_\e)\dt\leq c\Lambda(\eta)\big(\rho^{\frac{np-1}{p}-\alpha}+\e^{\frac{1}{p}}\rho^{\frac{n(p-1)}{p}-\alpha}\big).
\end{align}
Note that $\alpha$ was chosen such that $(np-1)/p-\alpha>n-1$. Moreover, since $\|u_\e\|_{L^\infty}\leq N-1/N_\rho$ the sets $\mathcal{S}_\e^m(t_l^m)\setminus\mathcal{S}_\e^m(t_{l+1}^m)$, $m\in\{1,\ldots,d\}$, $l=-N N_\rho,\ldots,N N_\rho-1$ form a partition of $Z_\e(Q_\rho^n)$. Thus, we can define a discrete function $v_\e$ componentwise by its restriction to $\mathcal{S}_\e^m(t_l^m)\setminus\mathcal{S}_\e^m(t_{l+1}^m)$ setting
\begin{align*}
(v_\e^i)^m_{|\mathcal{S}_\e^m(t_l^m)\setminus\mathcal{S}_\e^m(t_{l+1}^m)}:=
\begin{cases}
0 &\text{if}\ t_{l}^m\leq 0< t_{l+1}^m,\\
\zeta^m &\text{if}\ t_{l}^m\leq\zeta^m<t_{l+1}^m,\\
t_l^m &\text{otherwise}.
\end{cases}
\end{align*}
Note that $v_\e$ is well-defined since $2/N_\rho<|\zeta^m|$ if $\zeta^m\neq 0$, so that in this case $\zeta^m$ and $0$ can not belong to the same interval $[t_l^m,t_{l+1}^m)$.

We claim that the required sequence $(w_\e)$ is obtained by setting $w_\e^i:=v_\e^{\e i}$ for every $i\in\Z^n$. First note that by construction the functions $v_\e$ satisfy the required boundary conditions, \ie $v_\e\in\A_{\e}^{\delta}(u_\zeta^\nu,Q_\rho^\nu)$. Thus, since $\e L<\delta$ the rescaled functions $w_\e$ are admissible for the minimum problem defining $g_{T_\e}(\zeta,\nu)$. We finally show that there exists $\hat{\e}=\hat{\e}(\eta)>0$ such that for all $\e\in(0,\hat{\e})$ the functions $w_\e$ satisfy \eqref{est:surface:lb}.
To this end we show that $\psi_i^s(\{w_\e^{i+j}\}_{j\in\Z^n})$ essentially only gives a contribution to the energy when $\e i\in\J_\e$, in which case it will turn out to be comparable to $\e\psi_i^\e(\{u_\e^{\e i+j}\}_{j\in Z_\e(\R^n)})$ thanks to \ref{Hpsi2} and \ref{Hpsi3}. We start by introducing the rescaled functions $\tilde{u}_\e$ defined by setting $\tilde{u}_\e^i:=u_\e^{\e i}$ for every $i\in\Z^n$ and we observe that
for $i\in Z_1(T_\e Q^\nu)$ with $\e i\in\J_\e$ we have
\begin{align}\label{lb:L-gradient}
\e^{\frac{1-p}{p}}|\nabla_{1,L}\tilde{u}_\e|(i)=\e^{\frac{1}{p}}|\nabla_{\e,L}u_\e|(\e i)\geq \Lambda(\eta).
\end{align}
Hence, from \ref{Hpsi2} we deduce the existence of $\hat{\e}=\hat{\e}(\eta)>0$ such that for every $\e\in(0,\hat{\e})$ and every $i\in\Z^n$ with $\e i\in\J_\e$ there holds
\begin{align}\label{est:psi-ue-phi-ue}
\e\phi_i^\e(\{u_\e^{\e i+j}\}_{j\in Z_\e(\Omega_i)})=\e\psi_{\e i}^\e(\{u_\e^{\e i+j}\}_{j\in Z_\e(\Rn)})\geq\psi_i^s(\{\tilde{u}_\e^{i+j}\}_{j\in\Z^n})-\eta.
\end{align}
We now compare $\psi_i^s(\{\tilde{u}_\e^{i+j}\}_{j\in\Z^n})$ and $\psi_i^s(\{w_\e^{i+j}\}_{j\in\Z^n})$. By construction we have
\begin{align}\label{est:linfty}
\|w_\e-\tilde{u}_\e\|_{L^\infty}=\|v_\e-u_\e\|_{L^\infty}\leq\frac{2\sqrt{d}}{N_\rho}\leq 4\sqrt{d}\rho^\alpha.
\end{align} 
For every $i\in\Z^n$ with $|\nabla_{1,L}\tilde{u}_e|(i)>0$ \eqref{est:linfty} together with \ref{Hpsi3} and Remark \ref{rem:continuity:psi}  gives
\begin{align}\label{est:psi-ue-we}
\psi_i^s(\{\tilde{u}_\e^{i+j}\}_{j\in\Z^n})\geq \psi_i^s(\{w_\e^{i+j}\}_{j\in\Z^n}) - c_s\hspace{-1.5em}\sum_{j\in Z_1(Q_L(i))}\dsum{\xi\in Z_1(Q_L(i)}{j+\xi\in Q_L(i)}\hspace{-1em}|w_\e^{j+\xi}-\tilde{u}_\e^{j+\xi}|\geq \psi_i^s(\{w_\e^{i+j}\}_{j\in\Z^n})- c\rho^\alpha,
\end{align}
where $c>0$ depends only on $n$, $d$ and $L$. In particular, \eqref{est:psi-ue-we} holds for every $i\in\Z^n$ with $\e \in\J_\e$ thanks to \eqref{lb:L-gradient}.
Gathering \eqref{est:psi-ue-we} and \eqref{est:psi-ue-phi-ue} we thus obtain
\begin{align}\label{est:surface1}
\frac{1}{\rho^{n-1}}F_\e(u_\e,Q_\rho^\nu) &\geq\frac{\e^{n-1}}{\rho^{n-1}}\sum_{i\in Z_\e(Q_\rho^\nu)\cap\J_\e}\hspace{-1em}\e\phi_i^\e(\{u_\e^{i+j}\}_{j\in Z_\e(\Omega_i)})\nonumber\\
&\geq\frac{1}{T_\e^{n-1}}\dsum{i\in Z_1(T_\e Q^\nu)}{\e i\in\J_\e}\hspace{-1em}\psi_i^s(\{w_\e^{i+j}\}_{j\in\Z^n})-\big(\rho^\alpha+\eta\big)\frac{\e^{n-1}}{\rho^{n-1}}\#(Z_\e(Q_\rho^\nu)\cap\J_\e).
\end{align}
Moreover, since $1/\e\leq|\nabla_{\e,L}u_\e|^p(i)$ for every $i\in\J_\e$, we can argue as in \eqref{est:Lgradient1} to bound the cardinality of the set $Z_\e(Q_\rho^\nu)\cap\J_\e$ via
\begin{align}\label{card:jumppoints}
\frac{\e^{n-1}}{\rho^{n-1}}\#(Z_\e(Q_\rho^\nu)\cap\J_\e) &\leq\frac{1}{\rho^{n-1}}\hspace*{-1em}\sum_{i\in Z_\e(Q_\rho^\nu)\cap\J_\e}\hspace*{-1.5em}\e^n\min\big\{|\nabla_{\e,L}u_\e|^p(i),\frac{1}{\e}\big\}\leq\frac{c}{\rho^{n-1}}\big(F_\e(u_\e,Q_\rho^\nu)+\e)\leq c+\frac{c\e}{\rho^{n-1}},
\end{align}
where the last inequality follows from \eqref{energybound:surface}. 
It then remains to show that the contributions of $\psi_i^s(\{w_\e^{i+j}\}_{j\in\Z^n})$ for $\e i\not\in\J_\e$ are negligible. First note that for every $i\in Z_1(T_\e Q^\nu)$ with $w_\e\equiv w_\e^i$ on $Z_1(Q_L(i))$ Hypotheses \ref{Hs4} gives $\psi_i^s(\{w_\e^{i+j}\}_{j\in\Z^n})=0$. On the other hand, if $i\in Z_1(T_\e Q^\nu)$ is such that $w_\e\nequiv w_\e^i$ on $Z_1(Q_L(i))$ then $i$ belongs to $\mathcal{R}_\e^m(t_l^m)$ for some $m\in\{1,\ldots,d\}$ and $l\in\{-NN_\rho,\ldots,NN_\rho-1\}$. Thus, we have
\begin{align}\label{est:regularpoints1}
\frac{1}{T_\e^{n-1}}\dsum{i\in Z_1(T_\e Q^\nu)}{\e i\not\in\J_\e}\psi_i^s(\{w_\e^{i+j}\}_{j\in\Z^n})\leq\frac{1}{T_\e^{n-1}}\sum_{m=1}^d\sum_{l=-NN_\rho}^{NN_\rho-1}\sum_{\e i\in\mathcal{R}_\e^m(t_l^m)\setminus\J_\e}\hspace*{-1.5em}\psi_i^s(\{w_\e^{i+j}\}_{j\in\Z^n}).
\end{align}
We finally observe that \eqref{est:linfty} and our choice of $\rho$ imply that $\|w_\e\|_{L^\infty}\leq 4|\zeta|$, so that we can use the upper bound in \ref{Hs3a} together with \eqref{est:optimaljumpset} to bound the sum on the right-hand side of \eqref{est:regularpoints1}. In fact, we have
\begin{align}\label{est:regularpoints2}
\frac{1}{T_\e^{n-1}}\sum_{m=1}^d\sum_{l=-NN_\rho}^{NN_\rho-1}\sum_{\e i\in\mathcal{\R}_\e^m(t_l^m)\setminus\J_\e}\hspace*{-1.5em}\psi_i^s(\{w_\e^{i+j}\}_{j\in\Z^n}) &\leq c_6(4|\zeta|+1)\frac{\e^{n-1}}{\rho^{n-1}}\sum_{m=1}^d\sum_{l=-N N_\rho}^{N N_\rho-1}\#(\mathcal{R}_\e^m(t_l^m)\setminus\J_\e)\nonumber\\
&\leq c\Lambda(\eta)\big(\rho^{\frac{p-1}{p}-\alpha}+\e^{\frac{1}{p}}\rho^{\frac{p-n}{p}-\alpha}\big).
\end{align}
Gathering \eqref{est:surface1}-\eqref{est:regularpoints2} we deduce that the sequence $(w_\e)$ satisfies \eqref{est:surface:lb} with 
\begin{align*}
R(\e,\rho)=c\Lambda(\eta)\big(\rho^\alpha+\e\rho^{1-n}+\rho^{\frac{p-1}{p}-\alpha}+\e^{\frac{1}{p}}\rho^{\frac{p-n}{p}-\alpha}\big) \to 0\quad \text{as first}\ \e\to 0\ \text{and then}\ \rho\to 0,
\end{align*}
where the convergence of $R(\e,\rho)$ is guaranteed by the choice of $\alpha\in(0,(p-1)/p)$. Thus the argument in \eqref{est:gbar-ghom} concludes this step providing us with the inequality $\bar{g}\geq g_{\rm hom}$.
\end{step}
\begin{step}{Step 2} $\bar{g}(\zeta,\nu)\leq g_{\rm hom}(\zeta,\nu)$

\noindent In order to prove the inequality we construct a recovery sequence for $u_{\zeta,x_0}^\nu$ on $Q_\rho^\nu(x_0)$, where $x_0\in\Omega$ and $\rho>0$ are such that $Q_\rho^\nu(x_0)\wcont\Omega$.
To simplify the exposition we only consider the case $\nu=e_n$ here and we assume that $x_0=0$ and $\rho=1$. We fix $\eta>0$ and set
\begin{align*}
Q(\eta):=\big(-1/2,1/2\big)^{n-1}\times\big(-\eta/2,\eta/2\big).
\end{align*}
Moreover, we choose $T=T(\eta)\in\N$ as a multiple of $K$ with $1/T<\eta$ and $u_T\in\A_1^{\sqrt{n}L}(u_{\zeta}^{e_n},TQ)$ satisfying
\begin{align}\label{est:almostoptimal}
\frac{1}{T^{n-1}}\sum_{i\in Z_1(TQ)}\psi_i^s(\{u_T^{i+j}\}_{j\in\Z^n})\leq g_{\rm hom}(\zeta,e_n)+\eta.
\end{align}
Starting from $u_T$ we now construct a sequence $(u_\e)$ converging in $L^1(\Omega;\Rd)$ to $u_\zeta^{e_n}$ and satisfying
\begin{align}\label{est:surface:ub1}
\limsup_{\e\to 0}F_\e(u_\e,Q(\eta))\leq g_{\rm hom}(\zeta,e_n)+c\eta,
\end{align}
where the constant $c>0$ depends only on $L,n,\zeta$. Then Proposition \ref{cor:trans-inv} gives
\begin{align*}
\bar{g}(\zeta,e_n)=F(u_\zeta^{e_n},Q(\eta))\leq\liminf_{\e\to 0}F_\e(u_\e,Q(\eta))\leq g_{\rm hom}(\zeta,e_n)+c\eta,
\end{align*}
and we obtain the required inequality thanks to the arbitrariness of $\eta>0$.
 
As a first step we define a function $\bar{u}_T:\Z^n\to\Rd$ which is $T$-periodic in the directions $(e_1,\ldots,e_{n-1})$ inside the stripe $\{|\langle x,e_n\rangle|<T/2\}$ by setting
\begin{align*}
\bar{u}_T:=
\begin{cases}
u_T^{i-Tj'} &\text{if}\ i\in Z_1(Tj'+TQ)\ \text{for some}\ j'\in\Z^{n-1}\times\{0\},\\
u_\zeta^{e_n}(i) &\text{otherwise in}\ \Z^n.
\end{cases}
\end{align*}
For every $\e>0$ and every $i\in Z_\e(\R^n)$ we then set $u_\e^i:=u_T^{i/\e}$ and we observe that as $\e\to 0$ the sequence $(u_\e)$ converges in $L^1(\Omega;\Rd)$ to $u_{\zeta}^{e_n}$. It remains to show that $(u_\e)$ satisfies \eqref{est:surface:ub1}. To this end, for every $\e>0$ we consider the stripe
\begin{align*}
S_\e(T):=\{x\in\Rn\colon |\langle x,e_n\rangle|<\e T/2\}.
\end{align*}
For $\e<\eta/T$ we can rewrite the energy as
\begin{align}\label{est:surface:ub2}
F_\e(u_\e,Q(\eta))=\hspace*{-1em}\sum_{i\in Z_1(1/\e Q)\cap S_1(T)}\hspace*{-1.5em}\e^n\psi_{\e i}^\e(\{\bar{u}_T^{i+\tfrac{j}{\e}}\}_{j\in Z_\e(\Rn)})+\hspace*{-1.5em}\sum_{i\in Z_\e(Q(\eta))\setminus S_\e(T)}\hspace*{-1.5em}\e^n\psi_i^\e(\{u_\zeta^{e_n}(i+j)\}_{j\in Z_\e(\Rn)}).
\end{align}
Thanks to the upper bound for constant functions \eqref{upperbound:constant} the second term on the right-hand side of \eqref{est:surface:ub2} is at most proportional to $\eta$. In fact we have
\begin{align}\label{est:surface:ub3}
\sum_{i\in Z_\e(Q(\eta))\setminus S_\e(T)}\hspace*{-1.5em}\e^n\psi_i^\e(\{u_\zeta^{e_n}(i+j)\}_{j\in Z_\e(\Rn)})\leq (c_1+1)\e^n\#\{i\in Z_\e(Q(\eta))\}\leq c\eta
\end{align}
with $c$ depending only on $n$. We continue estimating the first term on the right-hand side of \eqref{est:surface:ub2}. Since $T$ is fixed, the function $\bar{u}_T$ takes only finitely many values. Thus, there exists $\e_0=\e_0(T,\eta)>0$ such that for every $\e\in(0,\e_0)$ and every $i\in\Z^n$ we either have $\e^{\frac{1-p}{p}
}|\nabla_{1,L}\bar{u}_T|(i)\geq \Lambda(\eta/T)$ or $|\nabla_{1,L}\bar{u}_T|(i)=0$, where $\Lambda(\eta/T)$ is given by \ref{Hpsi2}. As a consequence, setting $\e_1:=\min\{\e_0,\hat{\e}(\eta/T)\}$ with $\hat{\e}(\eta/T)$ again given by \ref{Hpsi2}, for every $\e\in (0,\e_1)$ and every $i\in\Z^n$ we obtain
\begin{align*}
|\psi_i^s(\{\bar{u}_T^{i+j}\}_{j\in\Z^n}-\e\psi_{\e i}^\e(\{\bar{u}_T^{i+\tfrac{j}{\e}}\}_{j\in Z_\e(\Rn)}|<\frac{\eta}{T}.
\end{align*}
Combining the above estimate with \eqref{est:surface:ub2} and \eqref{est:surface:ub3} we deduce that for every $\e\in (0,\e_1)$ there holds
\begin{align}\label{est:surface:ub4}
F_\e(u_\e,Q(\eta))\leq\hspace*{-1.5em}\sum_{i\in Z_1(1/\e Q)\cap S_1(T)}\hspace*{-1.5em}\e^{n-1}\psi_{i}^s(\{\bar{u}_T^{i+j}\}_{j\in \Z^n})+\frac{\eta}{T}\e^{n-1}\#(Z_1(\frac{1}{\e}Q)\cap S_1(T))+c\eta.
\end{align}
Note that there exists a constant $c>0$ depending only on $n$ such that 
\begin{align*}
\e^{n-1}\#(Z_1(\frac{1}{\e}Q)\cap S_1(T))\leq cT,\quad\text{for every}\ \e>0.
\end{align*}
Thus, setting
\begin{align*}
\mathcal{Z}_\e(T):=\{j'\in\Z^{n-1}\times\{0\}\colon \e Tj'+\e TQ\cap Q\neq\emptyset\}
\end{align*}
the estimate in \eqref{est:surface:ub4} can be continued to
\begin{align}\label{est:surface:ub5}
F_\e(u_\e,Q(\eta)) &\leq\e^{n-1}\sum_{j'\in\mathcal{Z}_\e(T)}\sum_{i\in Z_1(Tj'+T\overline{Q})}\psi_i^s(\{\bar{u}_T^{i+j}\}_{j\in\Z^n})+c\eta.
\end{align}
Note that for every $j'\in\mathcal{Z}_\e(T)$ and for every $i\in Z_1(Tj'+T\overline{Q})$ we have
\begin{align}\label{eq:periodic:extension}
\bar{u}_T^{i+j}=u_T^{i-Tj'+j}\quad\text{for every}\ j\in Z_1(LQ).
\end{align}
In fact, the above equality holds true by definition of $\bar{u}_T$ if $i\in Z_1(Tj'+TQ)$ is such that $Q_L(i)\subset Tj'+TQ$. If instead $i\in Z_1(Tj'+TQ)$ is such that $Q_L(i)\cap(\Rn\setminus Tj'+TQ)\neq\emptyset$, then the boundary conditions satisfied by $u_T$ together with the fact that $\langle j',e_n\rangle=0$ ensure that
\begin{align*}
\bar{u}_T^{i+j}=u_\zeta^{e_n}(i+j)=u_T^{i-Tj'+j}.
\end{align*}
Moreover, in combination with the locality property and periodicity, \eqref{eq:periodic:extension} gives
\begin{align*}
\sum_{i\in Z_1(Tj'+T\overline{Q})}\psi_i^s(\{\bar{u}_T^{i+j}\}_{j\in\Z^n})=\sum_{i\in Z_1(Tj'+T\overline{Q})}\psi_i^s(\{u_T^{i-Tj'+j}\}_{j\in\Z^n})=\sum_{i\in Z_1(T\overline{Q})}\psi_i^s(\{u_T^{i+j}\}_{j\in\Z^n}).
\end{align*}
Thus, since $\#\mathcal{Z}_\e(T)\leq(\lfloor\tfrac{1}{\e T}\rfloor+1)^{n-1}$, from \eqref{est:surface:ub5} we deduce that
\begin{align*}
&F_\e(u_\e,Q(\eta))\leq (\e T)^{n-1}\Big(\Big\lfloor\frac{1}{\e T}\Big\rfloor+1\Big)^{n-1}\frac{1}{T^{n-1}}\sum_{i\in Z_1(T\overline{Q})}\psi_i^s(\{u_T^{i+j}\}_{j\in\Z^n})+c\eta\\
&\hspace*{1em}\leq(\e T)^{n-1}\Big(\Big\lfloor\frac{1}{\e T}\Big\rfloor+1\Big)^{n-1}\Big(g_{\rm hom}(\zeta,e_n)+\eta+\frac{1}{T^{n-1}}\sum_{i\in Z_1(\partial TQ)}\psi_i^s(\{u_\zeta^{e_n}(i+j)\}_{j\in\Z^n})\Big)+c\eta,
\end{align*}
where to establish the second inequality we also used \eqref{est:almostoptimal} and the boundary conditions satisfied by $u_T$. We finally remark that for every $i\in Z_1(\partial TQ)$ with $|\langle i,e_n\rangle|\geq L/2$ the function $u_\zeta^{e_n}(i+\cdot)$ coincides with the constant function ${\rm sign}\langle i,e_n\rangle$ on $LQ$, so that $\psi_i^s(\{u^{i+j}\}_{j\in\Z^n})=0$. If instead $|\langle i,e_n\rangle|<L/2$ we use the upper bound in \ref{Hs3} to deduce that $\psi_i^s(\{u^{i+j}\}_{j\in\Z^n})\leq c_6(|\zeta|+1)$. Hence, we obtain
\begin{align*}
\frac{1}{T^{n-1}}\sum_{i\in Z_1(\partial TQ)}\psi_i^s(\{u_\zeta^{e_n}(i+j)\}_{j\in\Z^n})\leq c\#Z_1(\partial TQ\cap\{|\langle i,e_n\rangle|<L/2\})\leq \frac{c}{T}<c\eta,
\end{align*}
where the constant $c$ depends only on $n,L,\zeta$. Letting $\e\to 0$ we eventually find
\begin{align*}
\limsup_{\e\to 0}F_\e(u_\e,Q(\eta))\leq g_{\rm hom}(\zeta,e_n)+c\eta,
\end{align*}
that is, the sequence $(u_\e)$ satisfies \eqref{est:surface:ub1} and we may conclude.
\end{step}
\end{proof}
\begin{proof}[Proof of Theorem \ref{thm:homogenization}]
The result follows combining Theorem \ref{thm:int:rep}, Proposition \ref{cor:trans-inv}, Proposition \ref{prop:bulk} and Proposition \ref{prop:surface}.
\end{proof}
\section{Examples}\label{sect:examples}
\subsection{Pair interactions}
In the special case of interaction-energy densities $\phi_i^\e$ that take into account only pairwise interactions of the point $i$ with the remaining lattice points Theorem \ref{thm:int:rep} provides an analogous result to \citep[Theorem 3.1]{AC04} in the $GSBV$-setting (see also \citep{BG02} and \citep{Chambolle99} for the case of interaction-energy densities that are independent of the position $i$). More in detail, our result can be applied to energies of the form
\begin{align*}
F_\e(u)=\sum_{i\in Z_\e(\Omega)}\e^n\dsum{\xi\in\Z^n}{i+\e\xi\in\Omega}f_\e^\xi(i,D_\e^\xi u(i)),
\end{align*}
\ie when $\phi_i^\e:(\Rd)^{Z_\e(\Omega_i)}\to[0,+\infty)$ are given by
\begin{align*}
\phi_i^\e(\{z^j\}_{j\in Z_\e(\Omega_i)}):=\dsum{\xi\in\Z^n}{i+\e\xi\in\Omega}f_\e^\xi(i,D_\e^\xi z(0)).
\end{align*}
Here we assume that for every $\e>0$ and every $i\in Z_\e(\Omega)$ the function $f_\e^\xi(i,\cdot):\R^d\to[0,+\infty)$ is increasing in the sense that 
\begin{align}\label{cond:increasing}
f_\e(i,\zeta_1)\leq f_\e(i,\zeta_2)\ \text{for all}\ \zeta_1,\zeta_2\in\Rd\ \text{with}\ |\zeta_1|\leq|\zeta_2|. 
\end{align}
Moreover, we suppose that there exist constants $a_\e^{\xi},\hat{a}_\e^{\xi}\geq 0$ and $b_\e^{\xi},\hat{b}_\e^{i\xi}\geq 0$ such that for every $\e>0$ and every $\xi\in\Z^n$ we have
\begin{align}\label{growthcond:f}
\min\Big\{a_\e^{\xi}|\zeta|^p,\frac{b_\e^{\xi}}{\e}\Big\}\leq f_\e(i,\zeta)\leq\min\Big\{\hat{a}_\e^{\xi}|\zeta|^p,\frac{\hat{b}_\e^{\xi}}{\e}\Big\}\ \text{for every}\ (i,\zeta)\in Z_\e(\Omega)\times\Rd,
\end{align}
where the constants $a_\e^{\xi},\hat{a}_\e^{\xi},b_\e^{\xi},\hat{b}_\e^{\xi}$ satisfy the following hypotheses.
\begin{enumerate}[label={\rm (${\rm H}_{\rm pw}$\arabic*)}]
\setlength{\itemsep}{3pt}
\item\label{Hpw1} (upper bound) We have 
\begin{align}\label{Hpw1:summability1}
\limsup_{\e\to 0}\sum_{\xi\in\Z^n}(\hat{a}_\e^\xi+\hat{b}_\e^\xi)<+\infty
\end{align}
and for every $\eta>0$ there exists $M_\eta>0$ such that
\begin{align}\label{Hpw1:summability2}
\limsup_{\e\to 0}\dsum{\alpha\in\N}{\alpha >M_\eta}\dsum{\xi\in\Z^n}{|\xi|_\infty\geq\min\{\e\frac{\alpha}{2},\e\frac{M_\eta}{\sqrt{n}}\}}(\hat{a}_\e^\xi+\hat{b}_\e^\xi)<\eta;
\end{align}
\item\label{Hpw2} (lower bound) there exist $a,b>0$ such that $a_\e^{e_k}\geq a$, $b_\e^{e_k}\geq b$ for every $\e>0$ and every $k\in\{1,\ldots,n\}$;
\item\label{Hpw3} (relative control) there exists $\gamma>0$ such that for every $\e>0$ and every $\xi\in\Z^n$ with $\hat{a}_\e^\xi\neq 0$ there holds $|\xi|\hat{b}_\e^\xi\leq\gamma\hat{a}_\e^\xi$.
\end{enumerate}
Under the above assumptions $\phi_i^\e$ satisfy hypotheses \ref{H1}--\ref{H5}. In fact, \ref{H1} is automatically satisfied, since $\phi_i^\e$ depends on $\{z^j\}_{j\in\Z_\e(\Omega_i)}$ only through differences $z^j-z^l$, and \eqref{cond:increasing} ensures that \ref{H6} holds true. Moreover, for $\e$ small enough the upper bound \ref{H2} is satisfied with $c_1:=\limsup_\e\sum_\xi \hat{a}_\e^\xi+1$, which is finite thanks to \eqref{Hpw1:summability1}. The lower bound \ref{H3} holds true in view of \ref{Hpw2}. 

To verify the mild non-locality condition \ref{H4} we observe that for any $\e>0$, $i\in Z_\e(\Omega)$, $\alpha\in\N$ and $z,w:Z_\e(\Omega_i)\to\Rd$ with $z^j=w^j$ for all $j\in Z_\e(\e\alpha Q)$ we have
\begin{align*}
\phi_i^\e(\{z^j\}_{j\in Z_\e(\Omega_i)}) &=\sum_{\xi\in Z_1(\alpha Q)}f_\e(i,D_\e^\xi w^0)+\dsum{|\xi|_\infty\geq\frac{\alpha}{2}}{i+\e\xi\in\Omega}f_\e^\xi(i,D_\e^\xi z(0))\\
&\leq \phi_i^\e(\{w^j\}_{j\in Z_\e(\Omega_i)})+\dsum{|\xi|_\infty\geq\frac{\alpha}{2}}{i+\e\xi\in\Omega}\min\Big\{\hat{a}_\e^{\xi}|D_\e^\xi z(0)|^p,\frac{\hat{b}_\e^{\xi}}{\e}\Big\},
\end{align*}
where the second inequality follows from the positiveness of the $f_\e^\xi$ and \eqref{growthcond:f}. Thus, the required sequence $c_{\e,\alpha}^{j,\xi}$ in \ref{H4} is obtained by setting
\begin{align*}
c_{\e,\alpha}^{j,\xi}:=
\begin{cases}
\hat{a}_\e^\xi+\hat{b}_\e^\xi &\text{if}\ |\xi|_\infty\geq\frac{\alpha}{2},\ j=0,\\
0 &\text{otherwise},
\end{cases}
\end{align*}
which satisfies \eqref{H4:summability1} and \eqref{H4:summability2} thanks to \eqref{Hpw1:summability1} and \eqref{Hpw1:summability2}, respectively. 

It remains to establish \ref{H5}. To this end, let $z,w: Z_\e(\Omega_i)\to \Rd$ and $\varphi:Z_\e(\Omega_i)\to[0,1]$ a cut-off and set $v:=\varphi z+(1-\varphi)w$. Let us show that $\phi_i^\e(\{v^j\}_{j\in Z_\e(\Omega_i)})\leq R_i^\e(z,w,\varphi)$ with $R_i^\e(z,w,\varphi)$ as in \ref{H5}. We start by observing that
\begin{align}\label{gradient:cutoff}
D_\e^\xi v(0)=\varphi(0) D_\e^\xi z(0)+(1-\varphi(0))D_\e^\xi w(0)+D_\e^\xi\varphi(0)(z^{\e\xi}-w^{\e\xi})\ \text{for every}\ \xi\in\Z^n.
\end{align}
Thus, \eqref{growthcond:f} together with the convexity of $|\cdot|^p$ and the subadditivity of the $\min$ ensure that
\begin{align*}
\phi_i^\e(\{v^j\}_{j\in Z_\e(\Omega_i)} &\leq\hspace{-0.7em}\dsum{\xi\in\Z^n}{i+\e\xi\in\Omega}\hspace{-0.7em}\min\Big\{\hat{a}_\e^{\xi}|D_\e^\xi z(0)|^p,\frac{\hat{b}_\e^{\xi}}{\e}\Big\}+\min\Big\{\hat{a}_\e^{\xi}|D_\e^\xi w(0)|^p,\frac{\hat{b}_\e^{\xi}}{\e}\Big\}+\hat{a}_\e^\xi|D_\e^\xi \varphi(0)|^p|z^{\e\xi}-w^{\e\xi}|^p.
\end{align*}
Eventually, from \ref{Hpw3} we deduce that for every $\xi\in\Z^n$ there holds
\begin{align*}
\min\Big\{\hat{a}_\e^{\xi}|D_\e^\xi z(0)|^p,\frac{\hat{b}_\e^{\xi}}{\e}\Big\}=\hat{a}_\e^{\xi}\min\Big\{|D_\e^\xi z(0)|^p,\frac{\hat{b}_\e^{\xi}}{\hat{a}_\e^{\xi}\e}\Big\}\leq \hat{a}_\e^{\xi}\min\Big\{|D_\e^\xi z(0)|^p,\frac{\gamma}{\e|\xi|}\Big\},
\end{align*}
and the same estimate holds with $w$ in place of $z$. Since moreover 
\begin{align}\label{est:grad:phi}
|D_\e^\xi\varphi(0)|^p\leq\sup_{\substack{l\in Z_\e(\Omega_i)\\k\in\{1,\ldots,n\}}}|D_\e^k\varphi(l)|^p\ \text{for every}\ \xi\in\Z^n\ \text{with}\ i+\e\xi\in \Omega,
\end{align}
we obtain $\phi_i^\e(\{v^j\}_{j\in Z_\e(\Omega_i)})\leq R_i^\e(z,w,\varphi)$ with
\begin{align*}
R_i^\e(z,w,\varphi)=\hspace{-0.7em}\dsum{\xi\in\Z^n}{i+\e\xi\in\Omega}\hspace{-0.7em}\hat{a}_\e^\xi(\gamma+1)\Big( &\min\Big\{|D_\e^\xi z(0)|^p,\frac{\gamma}{\e|\xi|}\Big\}+\min\Big\{|D_\e^\xi w(0)|^p,\frac{\gamma}{\e|\xi|}\Big\}\\
&+\hspace{-0.7em}\sup_{\substack{l\in Z_\e(\Omega_i)\\k\in\{1,\ldots,n}}\hspace{-1em}|D_\e^k\varphi(l)|^p|z^{\e\xi}-w^{\e\xi}|^p\Big).
\end{align*}
It then suffices to remark that \eqref{H5:summability} is satisfied due to \eqref{Hpw1:summability1} to conclude.
\subsection{Multibody weak-membrane energies}
A prototipical example of functionals $F_\e$ as in \eqref{def:energy} where the interaction-energy densities $\phi_i^\e$ do not depend only on pairwise interactions of $i$ with $i+\e\xi$ but on multiple interactions of $i$ with $i+\e\xi_1,\ldots, i+\e\xi_N$ for some $N\in\N$ are so-called generalized weak-membrane energies, that have been studied in detail in \citep{Ruf17}. In our setting a generalized weak-membrane energy can be written as in \eqref{def:energy} with $\phi_i^\e$ given by
\begin{align}\label{def:wm}
\phi_i^\e(\{z^j\}_{j\in Z_\e(\Omega_i)}):=f_\e\Big(i,\sum_{\xi\in Z_1(LQ)}\dsum{j\in Z_\e(\e LQ)}{j+\e\xi\in\e LQ}c^\xi|D_\e^\xi z(j)|^p\Big),
\end{align}
where $L\in\N$ is the maximal range of interaction, $c^\xi\geq 0$ for every $\xi\in Z_1(LQ)$, and for every $\e>0$ and $i\in Z_\e(\Omega)$ the function $f_\e(i,\cdot):[0,+\infty)\to[0,+\infty)$ is increasing and satisfies
\begin{align}\label{bounds:wm}
\min\Big\{a_\e^i t,\frac{b_\e^i}{\e}\Big\}\leq f_\e(i,t)\leq\min\Big\{\hat{a}_\e^i t,\frac{\hat{b}_\e^i}{\e}\Big\},
\end{align}
for some $a_\e^i,\hat{a}_\e^i,b_\e^i,\hat{b}_\e^i\geq 0$. By construction the functions $\phi_i^\e$ satisfy \ref{H1} and \ref{H6}. To ensure that Hypotheses \ref{H2}--\ref{H5} are fulfilled we assume that the following holds.
\begin{enumerate}[label={\rm (${\rm H}_{\rm wm}$\arabic*)}]
\setlength{\itemsep}{3pt}
\item\label{Hwm1} There exist $a,\hat{a},b,\hat{b}\in (0,+\infty)$ such that $a_\e^i\geq a$, $b_\e^i\geq b$, $\hat{a}_\e^i\leq\hat{a}$, $\hat{b}_\e^i\leq\hat{b}$ for every $\e>0$ and every $i\in Z_\e(\Omega)$;
\item\label{Hwm2} for every $k\in\{1,\ldots,n\}$ there holds $c^{e_k}>0$.
\end{enumerate}
The uniform bounds on $\hat{a}_\e^i$ in \ref{Hwm1} together with the upper bound in \eqref{bounds:wm} imply that \ref{H2} holds true with $c_1:=\hat{a}\max\{c^\xi\colon \xi\in Z_1(LQ)\}(\#Z_1(LQ))^2$, while thanks to the uniform bounds on $a_\e^i, b_\e^i$ in \ref{Hwm1}, \ref{Hwm2}, the lower bound in \eqref{bounds:wm} and the monotonicity of $f_\e(i,\cdot)$ Hypotheses \ref{H3} is satisfied with $c_2:=\min\{a,b\}\min\{c^{e_k}\colon 1\leq k\leq n\}>0$.

Moreover, the mild-nonlocality condition \ref{H4} holds true by construction, since only finite-range interactions are taken into account. More precisely, in view of \ref{Hwm1} we can choose the sequence $c_{\e,\alpha}^{j,\xi}$ in \ref{H4} as
\begin{align*}
c_{\e,\alpha}^{j,\xi}:=
\begin{cases}
\max\{\hat{a},\hat{b}\}c^\xi &\text{if}\ \alpha<L,\ \xi\in Z_1(LQ),\ j\in Z_\e(\e LQ),\\
0 &\text{otherwise},
\end{cases}
\end{align*}
which satisfies \eqref{H4:summability1} and \eqref{H4:summability2}. 

Eventually, for every $z,w:Z_\e(\Omega_i)\to\Rd$ and every cut-off $\varphi:Z_\e(\Omega_i)\to[0,1]$ we can combine \eqref{gradient:cutoff} and \eqref{est:grad:phi} with the upper bounds in \eqref{bounds:wm} and \ref{Hwm1} to deduce that
\begin{align*}
\phi_i^\e(\{\varphi^j z^j+(1-\varphi^j)w^j\}_{j\in Z_\e(\Omega_i)})\leq\max\{\hat{a},\hat{b}\}\hspace{-1em}\sum_{\xi\in Z_1(LQ)}\hspace{-1em}c^\xi\Bigg(  &\dsum{j\in Z_\e(\e LQ)}{j+\e\xi\in\e LQ}\sup_{\substack{l\in Z_\e(\Omega_i)\\k\in\{1,\ldots,n\}}}|D_\e^k\varphi(l)|^p|z^{\e\xi}-w^{\e\xi}|^p\\
&+\min\Big\{|D_\e^\xi z^j|^p,\frac{1}{\e}\Big\}+\min\Big\{|D_\e^\xi w^j|^p,\frac{1}{\e}\Big\}\Bigg),
\end{align*}
which gives \ref{H5} by setting $c_\e^{j,\xi}:=\max\{\hat{a},\hat{b}\}c^\xi$ for $\xi\in Z_1(LQ)$, $j\in Z_\e(\e LQ)$ and $c_\e^{j,\xi}:=0$ otherwise.

\smallskip

Under the above assumptions the functionals $F_\e$ defined according to \eqref{def:energy} with $\phi_i^\e$ as in \eqref{def:wm} satisfy all the assumptions of Theorem \ref{thm:int:rep} and thus $\Gamma$-converge up to subsequences to a free-discontinuity functional of the form \eqref{int:form}. We eventually give sufficient conditions under which the sequence $(F_\e)$ satisfies the assumptions of Theorem \ref{thm:homogenization}. The first condition is $\e K$-periodicity of $f_\e$ in $i$, that is $f_\e(i+\e Ke_k,\cdot)=f_\e(i,\cdot)$ for every $k\in\{1,\ldots,n\}$, every $\e>0$ and every $i\in Z_\e(\Omega)$. We then extend $f_\e$ to $Z_\e(\Rn)\times[0,+\infty)$ by periodicity and in the same way we extend $\phi_i^\e$ to $(\Rd)^{Z_\e(\Rn)}$ . Moreover, we can assume that $a_\e^i,\hat{a}_\e^i,b_\e^i,\hat{b}_\e^i$ are $\e K$-periodic in $i$. We finally show that \ref{Hpsi1}--\ref{Hpsi3} are satisfied if we assume that in addition for every $i\in Z_1([0,K)^n)$ there exist $a^i,b^i>0$ such that
\begin{align}\label{conv:aebe}
a_\e^{\e i}\to a^i,\ \hat{a}_\e^{\e i}\to a^i\qquad\text{and}\qquad b_\e^{\e i}\to b^i,\ \hat{b}_\e^{\e i}\to b^i\quad\text{as}\ \e\to 0,
\end{align}
that is, the functions $f_\e(i,\cdot)$ approach a single truncated potential. By periodicity \eqref{conv:aebe} extends to $i\in\Z^n$. We claim that the required functions $\psi_i^b,\psi_i^s:(\Rd)^{\Z^n}\to[0,+\infty)$ are obtained by setting
\begin{align*}
\psi_i^b(\{z^j\}_{j\in\Z^n}) &:= a^i\sum_{\xi\in Z_1(LQ)}\dsum{j\in Z_1(LQ)}{j+\xi\in LQ}|D_1^\xi z(j)|^p,\\
\psi_i^s(\{z^j\}_{j\in\Z^n}) &:=
\begin{cases}
0 &\text{if}\ z^j=z^0\ \text{for every}\ j\in Z-1(LQ),\\
b^i &\text{otherwise}.
\end{cases}
\end{align*}
First note that \ref{Hpsi3} is automatically satisfied. We next establish \ref{Hpsi1}. Let $\eta>0$, $\Lambda>0$ and suppose that $z:\Z^n\to\Rd$ is such that $|\nabla_{1,L} z|(0)<\Lambda$. Set $z_\e^j:=\e z^{\frac{j}{\e}}$ for every $j\in\Z_\e(\Rn)$. Arguing as in Lemma \ref{lem:hypotheses:psi} to establish \ref{Hb7} we deduce that
\begin{align}\label{bound:grad:ze}
\sum_{\xi\in Z_1(LQ)}\dsum{j\in Z_\e(\e LQ)}{j+\e\xi\in\e LQ}c^\xi|D_\e^\xi z_\e(j)|^p=\sum_{\xi\in Z_1(LQ)}\dsum{j\in Z_1(LQ)}{j+\xi\in LQ}|D_1^\xi z(j)|^p< 2^{p-1}\underset{\xi\in Z_1(LQ)}{\max} c^\xi(1+\#Z_1(LQ))\Lambda^p
\end{align}
Let us choose $\bar{\e}=\bar{\e}(\eta,\Lambda)>0$ sufficiently small such that
\begin{align}\label{closeness:bulk}
\Lambda_0:=2^{p-1}\underset{\xi\in Z_1(LQ)}{\max} c^\xi(1+\#Z_1(LQ))\Lambda^p\leq\frac{b}{\hat{a}\e},\qquad |a_\e^{\e i}-a^i|\leq\frac{\eta}{\Lambda_0},\qquad |\hat{a}_\e^{\e i}-a^i|\leq\frac{\eta}{\Lambda_0},
\end{align}
for every $\e\in (0,\bar{\e}$ and every $i\in Z_1([0,K)^n)$. The first condition in \eqref{closeness:bulk} together with \eqref{bound:grad:ze} and \ref{Hwm1} ensure that
\begin{align*}
a_\e^{\e i}\sum_{\xi\in Z_1(LQ)}\dsum{j\in Z_\e(\e LQ)}{j+\e\xi\in\e LQ}c^\xi|D_\e^\xi z_\e(j)|^p\leq\frac{b_\e^{\e i}}{\e}\quad\text{for every}\ i\in\Z^n.
\end{align*}
Thus, \eqref{bounds:wm} gives
\begin{align*}
a_\e^{\e i}\sum_{\xi\in Z_1(LQ)}\dsum{j\in Z_\e(\e LQ)}{j+\e\xi\in\e LQ}c^\xi|D_\e^\xi z_\e(j)|^p\leq\phi_{\e i}^\e(\{z_\e^j\}_{j\in Z_\e(\Rn)})\leq\hat{a}_\e^{\e i}\sum_{\xi\in Z_1(LQ)}\dsum{j\in Z_\e(\e LQ)}{j+\e\xi\in\e LQ}c^\xi|D_\e^\xi z_\e(j)|^p,
\end{align*}
which in view of the second and third estimate in \eqref{closeness:bulk} and \eqref{bound:grad:ze} finally gives
\begin{align*}
|\psi_i^b(\{z^j\}_{j\in\Z^n})-\phi_{\e i}^\e(\{z_\e^j\}_{j\in Z_\e(\Rn)})|<\eta.
\end{align*}
It remains to show that $\psi_i^s$ satisfies \ref{Hpsi2}. We start by choosing $\Lambda>0$ such that
\begin{align}\label{cond:Lambda}
\frac{\Lambda^p \underset{1\leq k\leq n}{\min} c^{e_k}}{\hat{c}_1n^{p-1}2^p}>\frac{\hat{b}}{a},
\end{align} 
where $\hat{c}_1$ is the constant provided by Lemma \ref{lem:comp:gradients}. Moreover, given $\eta>0$ we choose $\hat{\e}=\hat{\e}(\eta)$ small enough such that $|b_\e^{\e i}-b^i|<\eta$, $|\hat{b}_\e^{\e i}-b^i|<\eta$ for every $\e\in(0,\hat{\e})$ and every $i\in Z_1([0,K)^n)$. Let $\e\in(0,\hat{\e})$ and suppose that $z:\Z^n\to\R^d$ satisfies $\e^\frac{1-p}{p}|\nabla_{1,L}z|(0)\geq\Lambda$. Then Lemma \ref{lem:comp:gradients} together with Jensen's inequality yield
\begin{align*}
\Lambda^p\leq\e^{1-p}|\nabla_{1,L}z|^p(0)\leq\e^{1-p}\hat{c}_1 n^{p-1}2^p\sum_{k=1}^n\dsum{j\in Z_1(LQ)}{j+e_k\in LQ}|D_1^kz(j)|^p.
\end{align*}
In particular, the rescaled functions $\hat{z}_\e$ obtained by setting $\hat{z}_\e:=z^\frac{j}{\e}$ for every $j\in Z_\e(\Rn)$ satisfy
\begin{align*}
\sum_{k=1}^n c^{e_k}\hspace{-1em}\dsum{j\in Z_\e(\e LQ)}{j+\e e_k\in \e LQ}|D_\e^k\hat{z}_\e(j)|^p\geq\e^{-p}\underset{1\leq k\leq n}{\min} c^{e_k}\sum_{k=1}^n\dsum{j\in Z_1(LQ)}{j+e_k\in LQ}|D_1^kz(j)|^p\geq\frac{\Lambda^p \underset{1\leq k\leq n}{\min} c^{e_k}}{\hat{c}_1n^{p-1}2^p}\frac{1}{\e},
\end{align*}
hence the choice of $\Lambda$ in \eqref{cond:Lambda} and \ref{Hwm1} ensure that
\begin{align*}
\frac{b_\e^{\e i}}{\e}=\min\Big\{a_\e^{\e i}\sum_{k=1}^n c^{e_k}\hspace{-1em}\dsum{j\in Z_\e(\e LQ)}{j+\e e_k\in \e LQ}|D_\e^k\hat{z}_\e(j)|^p,\frac{1}{\e}\Big\}\leq\phi_{\e i}^\e(\{\hat{z}_\e^j\}_{j\in Z_\e(\Rn)})\leq\frac{\hat{b}_\e^{\e i}}{\e}
\end{align*}
for every $i\in Z_\e(\Rn)$. Eventually, since $\e\in(0,\hat{\e}(\eta))$, this gives
\begin{align*}
b^i-\eta\leq b_\e^{\e i}\leq\e\phi_{\e i}^\e(\{\hat{z}_\e^j\}_{j\in Z_\e(\Rn)})\leq\hat{b}_\e^{\e i}\leq b^i+\eta.
\end{align*}
If on the other hand $z:\Z^n\to\Rd$ is such that $|\nabla_{1,L}z|(0)=0$ we obtain 
\begin{align*}
\phi_{\e i}^\e(\{\hat{z}_\e^j\}_{j\in Z_\e(\Rn)})=0=\psi_i^s(\{z^j\}_{j\in Z^n})
\end{align*}
for every $i\in\Z^n$, and we conclude that the functions $\psi_i^s$ satisfy \ref{Hpsi2}.
\subsection{Weak membrane with long-range small-tail interactions}\label{sect:example:nonlocal}
In \cite{B00} the author studies the asymptotic behavior of  weak-membrane energies of the form
\begin{align}\label{ex:nonlocal}
F_\e(u)=\sum_{\xi\in\Z}\dsum{i\in Z_\e(\Omega)}{i+\e\xi\in\Omega}\e\rho_\e(\e\xi-i)\min\Big\{|D_\e^\xi u(i)|^2,\frac{1}{\e}\Big\},
\end{align}
where $\Omega\subset\R$ is an open, bounded interval.
Assuming only a locally uniform summability condition for the functions $\rho_\e:\e\Z\to[0,+\infty)$ it is shown that the $\Gamma$-limit is a non-local integral functional. Moreover, the author provides examples of specific functions $\rho_\e$ including very long-range interactions with small tails, for which the $\Gamma$-limit is a (local) free-discontinuity functional. Among them are the discrete functionals as in \eqref{ex:nonlocal} with $\rho_\e:\e\Z\to[0,+\infty)$ given by
\begin{align*}
\rho_\e(t):=
\begin{cases}
1 &\text{if}\ t=\e,\\
\sqrt{\e} &\text{if}\ t=\e\lfloor\frac{1}{\sqrt{\e}}\rfloor,\\
0 &\text{otherwise},
\end{cases}
\end{align*}
which are shown to $\Gamma$-converge to the functional
\begin{align*}
F(u)=\int_\Omega|u'|^2\dt+\sum_{t\in S_u}\min\{1+|u^+(t)-u^-(t)|^2,2\}.
\end{align*}
We observe that thanks to our very mild non-locality condition \ref{H4} the above example can be recast in our framework by setting
\begin{align*}
\phi_i^\e(\{z^j\}_{j\in Z_\e(\Omega_i)}):=\min\Big\{\Big|\frac{z^\e-z^0}{\e}\Big|^2,\frac{1}{\e}\Big\}+\sqrt{\e}\min\Big\{\Big|\frac{z^{\e\lfloor\frac{1}{\e}\rfloor}-z^0}{\e^2\lfloor\frac{1}{\sqrt{\e}}\rfloor}\Big|^2,\frac{1}{\e}\Big\}.
\end{align*}
Indeed, note that $\phi_i^\e$ satisfies \ref{H1}--\ref{H3} for every $\e>0$ and every $i\in Z_\e(\Omega)$. Moreover, \ref{H4} is satisfied with the sequence $(c_{\e,\alpha}^{j,\xi})$ defined by setting
\begin{align*}
c_{\e,\alpha}^{j,\xi}:=
\begin{cases}
1 &\text{if}\ \alpha\leq 2,\ j=0,\ \xi=1,\\
\sqrt{\e} &\text{if}\ \alpha\leq 2\lfloor\frac{1}{\sqrt{\e}}\rfloor,\ j=0,\ \xi=\lfloor\frac{1}{\sqrt{\e}}\rfloor,\\
0 &\text{otherwise}.
\end{cases}
\end{align*}
The sequence $(c_{\e,\alpha}^{j,\xi})$ fulfills the required summability condition \eqref{H4:summability1}, since
\begin{align*}
\sum_{\alpha\in\N}\sum_{j\in Z_\e(\R)}\sum_{\xi\in\Z}c_{\e,\alpha}^{j,\xi}=2+\sum_{\alpha=1}^{2\lfloor\frac{1}{\sqrt{\e}}\rfloor}\sqrt{\e}\leq 4\quad\text{for every}\ \e>0.
\end{align*}
Moreover, the decaying-tail condition \eqref{H4:summability2} is satisfied since $c_{\e,\alpha}^{j,\xi}=0$ for every $\alpha>2\lfloor\frac{1}{\sqrt{\e}}\rfloor$. Thus, for every $\eta>0$ the sequence $(M_\eta^\e)$ can be chosen independently of $\eta$ as $M_\eta^\e=2\lfloor\frac{1}{\sqrt{\e}}\rfloor$, which satisfies the constraint $\e M_\eta^\e\to 0$ as $\e\to 0$. Eventually, \ref{H5} can be verified by using expression \eqref{gradient:cutoff} together with the convexity of $z\mapsto z^p$ and the subadditivity of the $\min$.

\bigskip

\noindent{\bf Acknowledgments.}
The work of Annika Bach and Marco Cicalese was supported by the DFG Collaborative Research Center TRR 109, ``Discretization in Geometry and Dynamics''.
Andrea Braides acknowledges the MIUR Excellence Department Project awarded to the Department of Mathematics, University of Rome Tor Vergata, CUP E83C18000100006.
\addcontentsline{toc}{section}{References}
\bibliographystyle{plainnat}
\bibliography{references}
\end{document}